\documentclass{amsart}
\usepackage{amsfonts}
\usepackage{amsmath}
\usepackage{amssymb}
\usepackage{mathrsfs}
\usepackage{geometry}
\usepackage{graphicx}
\usepackage{hyperref}
\usepackage{microtype}
\usepackage{geometry}
\usepackage{mathrsfs}
\numberwithin{equation}{section}

\newcommand{\la}{\lambda}
\newcommand{\al}{\alpha}
\newcommand{\be}{\beta}
\newcommand{\ga}{\gamma}

\newcommand{\ve}{\varepsilon}

\newcommand{\R}{\mathbb{R}}

\newcommand{\Z}{\mathbb{Z}}
\newcommand{\T}{\mathbb{T}}

\newcommand{\Om}{\Omega}
\newcommand{\om}{\omega}
\newcommand{\ccc}{\cdot\cdot\cdot}

\newcommand{\n}[1]{\Vert #1\Vert}

\newcommand{\bbn}[1]{\Big\Vert #1 \Big \Vert}
\newcommand{\lr}[1]{\left\{ #1\right\}}
\newcommand{\lrc}[1]{\left[ #1\right]}
\newcommand{\lrs}[1]{\left( #1\right)}

\newcommand{\lra}[1]{\langle #1\rangle}

\newcommand{\bblra}[1]{\Big\langle #1\Big\rangle}
\newcommand{\abs}[1]{|#1|}
\newcommand{\babs}[1]{\big | #1 \big|}
\newcommand{\bbabs}[1]{\Big | #1 \Big|}
\newcommand{\wt}[1]{\widetilde{#1}}
\newcommand{\wq}{\infty}
\newcommand{\pa}{\partial}

\newcommand{\ce}{\mathcal{E}}

\newcommand{\ck}{{\mathcal K}}
\newcommand{\cl}{{\mathcal L}}

\newcommand{\cp}{{\mathcal P}}

\newcommand{\cs}{{\mathcal S}}

\newcommand{\bfx}{\mathbf x}
\newcommand{\bfz}{\mathbf z}
\newcommand{\bfr}{\mathbf r}

\newcommand{\ol}{\overline}

\newcommand*{\hx}{\mathbf{x}}
\newcommand*{\hz}{\mathbf{z}}

\begin{document}

\newtheorem{theorem}{Theorem}[section]
\newtheorem{lemma}[theorem]{Lemma}

\theoremstyle{definition}
\newtheorem{definition}[theorem]{Definition}
\newtheorem{example}[theorem]{Example}
\theoremstyle{remark}
\newtheorem{remark}[theorem]{Remark}
\numberwithin{equation}{section}

\newtheorem{proposition}[theorem]{Proposition}
\newtheorem{corollary}[theorem]{Corollary}
\newtheorem{goal}[theorem]{Goal}

\title[2D focusing NLS from 3D]{The rigorous derivation of the $\mathbb{T}^{2}$ focusing cubic NLS from 3D}


\author[S. Shen]{Shunlin Shen}
\address{School of Mathematical Sciences, Peking University, Beijing, 100871, China \& Department of Mathematics, University of Rochester, Rochester, NY 14627, USA}

\email{sshen20@ur.rochester.edu}



\date{}

\dedicatory{}

\begin{abstract}
 We derive rigorously the 2D periodic focusing cubic NLS as the mean-field limit of the 3D focusing quantum many-body dynamics describing a dilute Bose gas with periodic boundary condition in the $x$-direction and a well of infinite-depth in the $z$-direction. Physical experiments for these systems are scarce. We find that, to fulfill the empirical requirement for observing NLS dynamics in experiments, namely, the kinetic energy dominates the potential energy, it is necessary to impose an extra restriction on the system parameters. This restriction gives rises to an unusual coupling constant.
 \end{abstract}

\maketitle

\section{Introduction}
Boes-Einstein condensate (BEC) is the phenomenon that occurs when particles the same quantum state. The first experimental observation of BEC in an interacting atomic gas occurred in $1995$, using laser cooling techniques \cite{anderson1995observation,davis1995bose}.

Let $t\in \R$ be the time variable and $\bfr_{N}=\lrs{r_{1},r_{2},...,r_{N}}\in \R^{3N}$ be the position vector of $N$ particles in $\R^{3}.$ Then BEC naively means that the $N$-body wave function $\psi_{N}(t,\bfr_{N})$ satisfies
$$\psi_{N}(t,\bfr_{N})\sim \prod_{j=1}^{N}\varphi(t,r_{j})$$
up to a phase factor solely depending on $t$, for some one particle state $\varphi$. That is, every particle takes the same quantum state. Equivalently, there is the Penrose-Onsager formulation of BEC: if we take $\gamma_{N}^{(k)}$ be the $k$-particle marginal densities associated with $\psi_{N}$ by
\begin{align}\label{equ:introduction k-particle marginal densities}
\gamma_{N}^{(k)}(t,\bfr_{k},\bfr_{k}')=\int \psi_{N}(t,\bfr_{k},\bfr_{N-k})\overline{\psi}_{N}(t,\bfr_{k}',\bfr_{N-k})d\bfr_{N-k},\quad \bfr_{k},\bfr_{k}'\in \R^{3k}.
\end{align}
Then BEC equivalently means
\begin{align}\label{equ:introduction marginal densities, BEC}
\gamma_{N}^{(k)}(t,\bfr_{k},\bfr_{k}')\sim\prod_{j=1}^{k}\varphi(t,r_{j})\overline{\varphi}(t,r_{j}').
\end{align}

It is widely believed that the cubic nonlinear Schr\"{o}dinger equation (NLS)
\begin{align}
i\pa_{t}\varphi=-\Delta \varphi+\mu|\varphi|^{2}\varphi,
\end{align}
which is called focusing if $\mu<0$ and defocusing $\mu>0$, describes BEC in the sense that $\varphi$ satisfies NLS. In this paper, we are interested in the focusing case.
There have been many physical experiments \cite{cornish2000stable,donley2001dynamics,khaykovich2002formation,strecker2002formation} and mathematical results \cite{chen2016focusing,chen2017focusing,chen2017rigorous2dfocusing,Lewin2015Mean,lewin2016mean,lewin2017note,nam2019norm} regarding the focusing case. However, from the experiment \cite{cornish2000stable}, one infers that not only it is very difficult to prove the 3D focusing NLS as the mean-field limit of a 3D focusing quantum many-body dynamic, but such a limit also may not be true. Thus, in focusing settings, both physical experiments and mathematical results emphasize one dimensional and two dimensional behaviours. To our knowledge, physical experiments regarding the two dimensional behavior in the real-world three dimensional setting are limited
and the corresponding mathematical research only studies the two dimensional behaviour in 2D.
Therefore, we turn our attention to the derivation of 2D focusing NLS from 3D. Interestingly, our analysis produces an unusual microscopic-to-macroscopic coupling constant and might provide some suggestions to the experiment. To expect a two-dimensional behaviour, we should confine a large number of bosons inside a trap with strong confinement in one direction. We consider a simple physical model, namely, quantum many-body dynamics with periodic boundary condition in the $x$-direction and a well of infinite-depth in the $z$-direction\footnote{Our exact proof also works for the case in which we put $\R^{2}$ in the $x$-direction. We choose $\T^{2}$ here, considering all of these limits problem originated from the thermodynamic limit on $\mathbb{T}^{3}$ (see a survey in \cite{boccato2019bogoliubov}). }. Such model with strong restriction in one direction was first considered by Schnee and Yngvason \cite{schnee2007bosons} for the defocusing time-independent problem. Then, the defocusing time-dependent 3D-to-2D program was studied by X. Chen and Holmer in \cite{chen2013rigorous2dfrom3d}, in which they used the quadratic potential $|\cdot|^{2}$ to represent the trap.
Here, we model the trap by using a well of infinite-depth in the $z$-direction. That is, we consider the Hamiltonian (see \cite{schnee2007bosons})
\begin{align}
H_{N,L,a}=\sum_{j=1}^{N}(-\Delta_{r_{j}}+L^{-2}V^{\perp}(z_{j}/L))+\sum_{1\leq i< j\leq N}\frac{1}{a^{3\be-1}}V\lrs{\frac{r_{i}-r_{j}}{a^{\be}}},
\end{align}
where
$$V^{\perp}(z)=\begin{cases}
1, &z\in(-\pi/2,\pi/2),\\
\wq, &z\notin (-\pi/2,\pi/2).
\end{cases}$$
For parameter $a$, the scaling of the interaction potential, we consider the case, which is called Region I in \cite{schnee2007bosons}. Schnee and Yngvason define $g=g(N,L,a)$ as follows
 $$g\sim (-\ln (L^{2}\ol{\rho})+L/a)^{-1},$$
 where $\ol{\rho}$ is the mean density.
 The Gross-Pitaevskii limit means $Ng\sim 1$ and hence $\ol{\rho}\sim N$. Then, the term $a/L$ dominates in the definition of $g$. Therefore,
 $$1\sim Ng \sim Na/L\Longleftrightarrow a\sim L/N.$$

For mathematical convenience, we take $a=L/N$ and consider the Hamiltonian
\begin{align}\label{equ:introduction hamiltonian operator}
H_{N,L}=\sum_{j=1}^{N}-\Delta_{r_{j}}+\frac{L}{N-1}\sum_{1\leq i<j\leq N}V_{N,L}\lrs{r_{i}-r_{j}}.
\end{align}
acing on the Hilbert space $L_{s}^{2}(\Om_{L}^{\otimes N})$, the subspace of $L^{2}(\Om_{L}^{\otimes N})$ consisting of functions that are symmetric with respect to permutations of the $N$ particles,
where
$$V_{N,L}(r_{i}-r_{j})=(N/L)^{3\beta}V\lrs{(N/L)^{\beta}(r_{i}-r_{j})}$$
and the domain\footnote{When $L=1$, we take $\Om=\Om_{1}$ for convenience.} $\Om_{L}=(-\pi,\pi)^{2}\times (-L\pi/2,L\pi/2).$ As $L\to 0,$ we see that the particles are strongly confined in the $z$-direction. For more detailed analysis of system parameters, also see \cite{chen2013rigorous2dfrom3d}.

 We take the periodic boundary condition\footnote{To match the periodic condition, $V_{N,L}(r)$ is considered as the periodic extension in the $x$-direction of the rescaled $V$ which is compactly supported on $\Om.$} in the $x$-direction and Dirichlet boundary condition in the $z$-direction. We will derive rigorously $\T^{2}$ focusing cubic NLS from the 3D quantum many-body dynamic. For simplicity, we take
$\cos_{L}(z)=(2/\pi)^{1/2}\cos(z/L)/L^{2}$, which is the normalized ground state eigenfunction. With the lowest energy, we notice that, as $L\to 0$, $\cos_{L}(z)$ has infinite energy. Thus, our main theorem is better to be stated regarding the renormalization.

Let $\psi_{N,L}(t,\cdot)=e^{itH_{N,L}}\psi_{N,L}(0,\cdot)$
denote the evolution of this initial data corresponding to the Hamiltonian operator $(\ref{equ:introduction hamiltonian operator})$. Define the rescaled solution
\begin{align} \label{eq:introduction rescaled solution}
\wt{\psi}_{N,L}(t,\mathbf{r}_{N})\stackrel{def}{=}L^{N/2}\psi_{N,L}(t,\hx_{N},L\hz_{N}),\quad \mathbf{r}_{N}\in \T^{2N}\times (-\pi/2,\pi/2)^{N},
\end{align}
and the rescaled Hamiltonian
 \begin{align}
 \wt{H}_{N,L}=\sum_{j=1}^{N}\lrs{-\Delta_{x_{j}}-\frac{1}{L^{2}}\pa_{z_{j}}^{2}}+\frac{1}{N-1}\sum_{1\leq i<j\leq N}\wt{V}_{N,L}(r_{i}-r_{j}),
 \end{align}
 where
 \begin{align} \label{equ:introduction inteaction rescaled}
 \wt{V}_{N,L}(r)=L\lrs{N/L}^{3\beta}V\lrs{(N/L)^{\beta}x,L(N/L)^{\beta}z}.
 \end{align}
Then
$$\lrs{\wt{H}_{N,L}\wt{\psi}_{N,L}}(t,\bfx_{N},\bfz_{N})=L^{N/2}\lrs{H_{N,L}\psi_{N,L}}(t,\bfx_{N},L\bfz_{N}),$$
and hence, we have
\begin{align}
\wt{\psi}_{N,L}(t,\bfr_{N})=e^{it\wt{H}_{N,L}}\wt{\psi}_{N,L}(0,\bfr_{N}).
\end{align}

\begin{definition}
We denote $C_{gn}$ the sharp constant of the 2D inhomogeneous Gagliardo-Nirenberg estimate\footnote{There are many versions of the Gagliardo-Nirenberg inequalities on $\T^{2}$. Our proof works more or less the same.} on torus:
\begin{equation} \label{equ: 2d inhomogeneous GN inequality}
\n{\phi}_{L^{4}(\T^{2})}\leq C_{gn}\n{\phi}_{L^{2}(\T^{2})}^{\frac{1}{2}}\n{\sqrt{1-\Delta} \phi}_{L^{2}(\T^{2})}^{\frac{1}{2}}.
\end{equation}

\end{definition}

\begin{theorem}\label{thm:the main theorem}
Assume $L(N/L)^{\beta}\to 1^{-}$ \footnote{We use the notation $L(N/L)^{\beta}\to 1^{-}$ to denote $L(N/L)^{\beta}\leq 1$ and $L(N/L)^{\beta}\to 1.$} and the pair interaction $V$ is an even nonpositive smooth function compactly support on $\Om$ such that $\n{V}_{L_{z}^{\infty}L_{x}^{1}}\leq \frac{2\alpha}{C_{gn}^{4}}$ for some $\alpha\in (0,1)$. Let $\lr{\wt{\gamma}_{N,L}^{(k)}(t,\mathbf{r}_{k},\mathbf{r}_{k}')}$ be the family of marginal densities associated with the $3D$ rescaled Hamiltonian evolution $\wt{\psi}_{N,L}(t)=e^{it\wt{H}_{N,L}}\wt{\psi}_{N,L}(0)$ for $\beta\in (0,3/7)$. Suppose the initial datum $\wt{\psi}_{N,L}(0)$ satisfies the following:\\
$(i)$ $\wt{\psi}_{N,L}(0)$ is normalized, that is, $\n{\wt{\psi}_{N,L}(0)}_{L^{2}}=1$,\\
$(ii)$ $\wt{\psi}_{N,L}(0)$ is asymptotically factorized in the sense that
$$\lim_{N,1/L\to \infty}Tr\bbabs{\wt{\gamma}_{N,L}^{(1)}(0,x_{1},z_{1};x_{1}',z_{1}')-\frac{2}{\pi}\phi_{0}(x_{1})\overline{\phi_{0}}(x_{1}')
\cos(z_{1})\cos(z_{1}')}=0,$$
for some one particle state $\phi_{0}\in H^{1}(\T^{2})$,\\
$(iii)$ Away from the $z$-direction ground state energy, $\wt{\psi}_{N,L}(0)$ has finite energy per particle
$$\sup_{N,L}\lra{\wt{\psi}_{N,L}(0),(N^{-1}\wt{H}_{N,L}-1/L^{2})\wt{\psi}_{N,L}(0)}\leq C.$$
Then $\forall k\geq 1$, $t\geq 0$, we have the convergence in trace norm that
\begin{align*}
\lim_{\substack{N,1/L\to \infty\\L(N/L)^{\beta}\to 1^{-}}} Tr\bbabs{\wt{\gamma}_{N,L}^{(k)}(t,\mathbf{x}_{k},\mathbf{z}_{k};\mathbf{x}_{k}',\mathbf{z}_{k}')-\prod_{j=1}^{k}\frac{2}{\pi}\phi(t,x_{j})
\overline{\phi}(t,x_{j}')\cos(z_{j})\cos(z_{j}')}=0,
\end{align*}
where $\phi(t,x)$ solves the $2D$ periodic focusing cubic NLS with coupling constant
$$g_{0}=\frac{4}{\pi^{2}}\int \int  V(x,z_{1}-z_{2})dx|\cos(z_{1})\cos(z_{2})|^{2}dz_{1}dz_{2}<0,$$ that is
\begin{equation} \label{eq:focusing cubic NLS}
i\pa_{t}\phi=-\Delta_{x}\phi+g_{0}|\phi|^{2}\phi,
\end{equation}
with initial condition $\phi(0,x)=\phi_{0}(x).$
\end{theorem}

It is well-known that Theorem $\ref{thm:the main theorem}$ is equivalent to Theorem $\ref{thm:the main theorem 2}$ by the method of Erd{\"o}s, Schlein, and Yau \cite{erdHos2006derivation,erdHos2007derivation,erdHos2009rigorous,erdos2010derivation,
erdHos2007rigorous}.
\begin{theorem} \label{thm:the main theorem 2}
Assume $L(N/L)^{\beta}\to 1^{-}$ and the pair interaction $V$ is an even nonpositive smooth function compactly support on $\Om$ such that $\n{V}_{L_{z}^{\infty}L_{x}^{1}}\leq \frac{2\alpha}{C_{gn}^{4}}$ for some $\alpha\in (0,1)$. Let $\lr{\wt{\gamma}_{N,L}^{(k)}(t,\mathbf{r}_{k},\mathbf{r}_{k}')}$ be the family of marginal densities associated with the $3D$ rescaled Hamiltonian evolution $\wt{\psi}_{N,L}(t)=e^{it\wt{H}_{N,L}}\wt{\psi}_{N,L}(0)$ for $\beta\in (0,3/7)$. Suppose the initial datum $\wt{\psi}_{N,L}(0)$ is normalized asymptotically factorized and satisfies the energy condition that  \\
$(iii')$ there is a constant $C>0$ such that
$$\sup_{N,L}\lra{\wt{\psi}_{N,L}(0),(N^{-1}\wt{H}_{N,L}-1/L^{2})^{k}\wt{\psi}_{N,L}(0)}\leq C^{k},\quad \forall k\geq 1.$$
Then $\forall k\geq 1$, $t\geq 0$, we have the convergence in trace norm that
\begin{align*}
\lim_{\substack{N,1/L\to \infty\\L(N/L)^{\beta}\to 1^{-}}} Tr\bbabs{\wt{\gamma}_{N,L}^{(k)}(t,\mathbf{x}_{k},\mathbf{z}_{k};\mathbf{x}_{k}',\mathbf{z}_{k}')-\prod_{j=1}^{k}\frac{2}{\pi}\phi(t,x_{j})
\overline{\phi}(t,x_{j}')\cos(z_{j})\cos(z_{j}')}=0,
\end{align*}
where $\phi(t,x)$ solves the $2D$ periodic focusing cubic NLS with the coupling constant
$$g_{0}=\frac{4}{\pi^{2}}\int \int V(x,z_{1}-z_{2})dx|\cos(z_{1})\cos(z_{2})|^{2}dz_{1}dz_{2}<0,$$
that is
\begin{equation}\label{equ:the main theorem 2d focusing NLS}
i\pa_{t}\phi=-\Delta_{x}\phi+g_{0}|\phi|^{2}\phi,
\end{equation}
with initial condition $\phi(0,x)=\phi_{0}(x).$
\end{theorem}

For $\be<1/3$, our exact proof also works for the case in which we put $\R^{2}$ in the $x$-direction.
\begin{theorem}
Under the same condition of Theorem $\ref{thm:the main theorem}$ with $\be\in (0,1/3)$, then $\forall k\geq 1$, $t\geq 0$, we have the convergence in trace norm that
\begin{align*}
\lim_{\substack{N,1/L\to \infty\\L(N/L)^{\beta}\to 1^{-}}} Tr\bbabs{\wt{\gamma}_{N,L}^{(k)}(t,\mathbf{x}_{k},\mathbf{z}_{k};\mathbf{x}_{k}',\mathbf{z}_{k}')-\prod_{j=1}^{k}\frac{2}{\pi}\phi(t,x_{j})
\overline{\phi}(t,x_{j}')\cos(z_{j})\cos(z_{j}')}=0,
\end{align*}
where $\phi(t,x)$ solves the $2D$ focusing cubic NLS with the coupling constant
$$g_{0}=\frac{4}{\pi^{2}}\int \int V(x,z_{1}-z_{2})dx|\cos(z_{1})\cos(z_{2})|^{2}dz_{1}dz_{2}<0,$$
that is
\begin{equation}\label{equ:the main theorem 2d focusing NLS}
i\pa_{t}\phi=-\Delta_{x}\phi+g_{0}|\phi|^{2}\phi,
\end{equation}
with initial condition $\phi(0,x)=\phi_{0}(x)\in H^{1}(\R^{2}).$
\end{theorem}

We notice that Theorems $\ref{thm:the main theorem}-\ref{thm:the main theorem 2}$ carry an extra requirement $L(N/L)^{\beta}\to 1^{-}$ and a different coupling constant\footnote{This extra requirement and the coupling constant certainly give rises to a density condition for the gas. We do not compute this density as it is not our main goal here.}, if compared to the previous work, for example \cite{chen2013rigorous2dfrom3d,chen2017focusing,chen2017rigorous2dfocusing}, in which the constant is usually $\int V $ or the scattering length of $V$. It emerges from the empirical requirement for observing NLS dynamics in experiments, namely, the kinetic energy dominates the potential energy. We will certainly explain it in detail during the course of the proof. Due to the requirement, the limit of $\wt{V}_{N,L}$ defined by $(\ref{equ:introduction inteaction rescaled})$ is not a 3D $\delta$-function, though it scales like one.

There are two well-developed schemes to deal with this type of procedure. One is the Fock space method, while the other is the hierarchy approach. We take the hierarchy approach here\footnote{We believe the Fock space method will reach the same result. We just prefer a $H^{1}$ result here. In fact, some techniques we used come from the Fock space literatures \cite{Lewin2015Mean,lewin2017note}.}. The BBGKY hierarchy associated with $\wt{\psi}_{N,L}$ is
\begin{align}\label{equ:BBGKY hierarchy rescaled}
i\pa_{t}\wt{\gamma}_{N,L}^{(k)}=&\sum_{j=1}^{k}\lrc{-\Delta_{x_{j}},\wt{\gamma}_{N,L}^{(k)}}+\frac{1}{L^{2}}\sum_{j=1}^{k}\lrc{-\pa_{z_{j}}^{2},\wt{\gamma}_{N,L}^{(k)}}
\\
&+\frac{1}{N-1}\sum_{1\leq i<j\leq k}\lrc{\wt{V}_{N,L}(r_{i}-r_{j}),\wt{\gamma}_{N,L}^{(k)}}+\frac{N-k}{N-1}\sum_{j=1}^{k}Tr_{r_{k+1}}\lrc{\wt{V}_{N,L}(r_{j}-r_{k+1}),
\wt{\gamma}_{N,L}^{(k+1)}}.\notag
\end{align}

It was Erd{\"o}s, Schlein, and Yau who first rigorously derived the 3D cubic defocusing NLS from a 3D quantum many-body dynamic in their fundamental papers \cite{erdHos2006derivation,erdHos2007derivation,erdHos2009rigorous,erdos2010derivation,
erdHos2007rigorous}. They proved a-prior $L_{T}^{\infty}H_{x}^{1}$ bound to establish the compactness of BBGKY with respect to a topology on the trace class operators. Then, they showed that the limit point satisfies GP hierarchy. Finally, the proof for the uniqueness of GP hierarchy was the principal part and also surprisingly dedicate due to the fact that it is a system of infinitely many coupled equations over an unbounded number of variables. It motivated a large amount of works \cite{adami2007rigorous,chen2015unconditional,chen2010on,chen2011quintic,chen2013a,chen2014derivation,chen2014higher,
chentaliferro2014derivation,chen2012collapsing,chen2012second,chen2013rigorous,chen2013rigorous2dfrom3d,chen2019the,
kirkpatrick2011derivation,klainerman2008uniqueness,
pickl2010derivation,xie2015derivation}.

Subsequently, with imposing an additional a-prior condition on space-time norm, Klainerman and Machedon \cite{klainerman2008uniqueness} gave an another proof of the uniqueness of GP hierarchy in a different space of density matrices defined by Hilbert-Schmidt type Sobolev norms. Later, the approach of Klainerman and Machedon was used by Kirkpatrick, Schlein and Staffilani \cite{kirkpatrick2011derivation} to derived the 2D cubic defocusing NLS from the 2D quantum many-body dynamic both on $\R^{2}$ and $\T^{2}$; by T. Chen and Pavlovi{\'c} \cite{chen2011quintic} to derive the quintic NLS for $d=1$, $2$; by X. Chen \cite{chen2013rigorous} to investigate the trapping problem in 2D and 3D; and by X. Chen and Holmer \cite{chen2013rigorous2dfrom3d} to derive 2D cubic defocusing NLS from the 3D quantum many-body dynamic.

Later on, T. Chen, Hainzl, Pavlovi{\'c} and Seiringer \cite{chen2015unconditional}, using the quantum de Finetti theorem from \cite{lewin2014derivation}, provided a simplified proof of the $L_{T}^{\infty}H_{x}^{1}$-type 3D cubic uniqueness theorem in \cite{erdHos2007derivation}. This method in \cite{chen2015unconditional} inspired the study for refined uniqueness theorems, such as \cite{chen2014unconditional,hong2015unconditional,sohinger2015rigorous}.

Using Fock space methods to study the convergence rate has also been worked on by many authors, for example, see \cite{benedikter2015quantitative,chen2011rate,grillakis2013pair,grillakis2017pair,kuz2015rate,kuz2017exact,nam2019norm,rodnianski2009quantum}, and the references within.

For the focusing setting, which is a natural continuation of the defocusing problem, X. Chen and Holmer \cite{chen2016focusing} first derived the 1D focusing cubic NLS and later a 3D-to-1D reduction in \cite{chen2017focusing}. But the 2D cubic case did not see any process until \cite{lewin2016mean}, in which Lewin, Nam, and Rougerie used a quantitative version of the quantum de Finetti theorem \cite{christandl2007one} to show that the ground state energy of the 2D $N$-body was described by a NLS ground state energy. Using the finite-dimensional quantum de Finetti
theorem in \cite{lewin2016mean}, X. Chen and Holmer \cite{chen2017rigorous2dfocusing} derived 2D focusing cubic NLS from the 2D quantum many-body dynamic for $\beta\in (0,1/6)$. For higher $\beta$, Lewin, Nam, and Rougerie \cite{lewin2017note} used a bootstrapping argument to improve $\beta$, which, together with the approach in \cite{chen2017rigorous2dfocusing}, implied the convergence of the quantum many-body dynamics to the focusing NLS for $\beta\in (0,3/4)$. In the 2D focusing case, the stability of the second kind which is the energy bound when $k=1$ was improved to $\be<1$ in \cite{nam2020improved}. Besides the convergence of density matrix, the convergence rate is also of interest and extended to lower dimension in both focusing and defocusing cases by Nam and Napi\'{o}rkowski in \cite{nam2019norm} using $H^{4}$ regularity.

The derivation of 2D defocusing cubic NLS from 3D was first by X. Chen and Holmer in \cite{chen2013rigorous2dfrom3d} and then by Bo{\ss}mann in \cite{bossmann2019derivation} for the regime $\be\in (0,1]$. To our knowledge, the derivation of 2D focusing cubic NLS from 3D has not been completed before. In this paper, we follow the lead of the aforementioned focusing works \cite{chen2016focusing,chen2017focusing,chen2017rigorous2dfocusing,Lewin2015Mean,lewin2016mean,lewin2017note} and pursuit the treatment of 2D case from the 3D physical setting.

\subsection{Outline of the Proof of Theorem \ref{thm:the main theorem}}
We first establish in Section $\ref{section 2}$, under the assumption $L(N/L)^{\beta}\to 1^{-}$, that the renormalized kinetic energy controls the potential energy and hence yield an $H^{1}$ regularity bound to make the other parts of the paper work.

In section $\ref{section N,L}$, we use scaling arguments to show why we choose the uncommon mixed norm $\n{V}_{L_{z}^{\infty}L_{x}^{1}}$ and we are bounded by the extra restriction $L(N/L)^{\beta}\to 1^{-}$. In fact, a similar requirement would also show up in the harmonic well case studied by X. Chen and Holmer \cite{chen2013rigorous2dfrom3d} if one wants the renormalized kinetic energy to bound the potential energy instead of dropping it.
Subsequently we prove the energy bound with $\be<1/2$ when $k=1$, which is divided into two parts. First in Section $\ref{section k=1}$, we get to $\be<1/3$. Instead of taking the approach in X. Chen and Holmer \cite{chen2013rigorous2dfrom3d,chen2017focusing,chen2017rigorous2dfocusing}, our proof improvises from Lewin, Nam, and Rougerie \cite{lewin2017note} and Lewin \cite{Lewin2015Mean}. This proof does not use the finite-dimensional quantum de Finetti theorem and thus can be applied to the $\R^{2}$ case as well. Hence, the main Theorem $\ref{thm:the main theorem}$ works the same for $\R^{2}$ with $\be<1/3$. Second in Section $\ref{section bootstapping}$, we adapt the bootstrapping argument in \cite{lewin2017note} to reach $\beta<1/2$. Then in Section $\ref{section k>1}$, we complete Theorem $\ref{thm: Stability of matter when k>=1}$ when $k>1$, where the restriction $\be<3/7$ is required. In the defocusing case \cite{chen2013rigorous2dfrom3d},  if we require a similar requirement, the index $\be$ can be improved to $3/7$ as well.

In Section $\ref{section 3}$, we show the compactness of the BBGKY sequence. Then, we use a modified version of the approximation of identity type lemma to show that limit points satisfy the GP hierarchy with the unusual coupling constant $g_{0}$. The uniqueness for GP hierarchy on $\T^{2}$ has been well studied by Kirkpatrick, Schlein and Staffilani \cite{kirkpatrick2011derivation}, Herr and Sohinger \cite{herr2016gross,herr2019unconditional}. We use their uniqueness theorems to conclude our proof.

\section{Focusing energy estimates} \label{section 2}
In this section, we prove focusing energy estimates. Define
$$S_{j}:=(1-\Delta_{r_{j}}-1/L^{2})^{1/2},$$
and write
$$S^{(k)}=\prod_{j=1}^{k}S_{j}.$$
\begin{theorem} \label{thm: Stability of matter when k>=1}
Assume $L(N/L)^{\beta}\to 1^{-}$, $\beta<\frac{3}{7},$ and $\n{V}_{L_{z}^{\infty}L_{x}^{1}}\leq\frac{2\alpha}{C_{gn}^{4}}$ for some $\alpha\in (0,1),$ then let $c_{0}=\min\lrs{\frac{1-\alpha}{\sqrt{2}},\frac{1}{2}},$ we have $\forall k\geq 0,$ there exists an $N_{0}(k)>0$ such that
\begin{align} \label{equ:High energy estimates when k>1}
\lra{\psi_{N,L},\lrs{2+N^{-1}H_{N,L}-1/L^{2}}^{k}\psi_{N,L}}\geq c_{0}^{k}\n{S^{(k)}\psi_{N,L}}_{L^{2}}^{2}
\end{align}
for all $N>N_{0}$ and for all $\psi_{N,L}\in L_{s}^{2}(\Om_{L}^{\otimes N}).$
\end{theorem}
\begin{proof}
For smoothness of presentation, we postpone the proof of Theorem $\ref{thm: Stability of matter when k>=1}$ to Section $\ref{section k=1}-\ref{section k>1}$.
\end{proof}

Now we convert the conclusions of Theorem $\ref{thm: Stability of matter when k>=1}$ into the statement about the rescaled solution, which we will use in the remainder of the paper.

Let $\wt{P}_{0}$ denote the orthogonal projection onto the ground state of $-\pa_{z}^{2}-1$ on the region $(-\pi/2,\pi/2)$ with Dirichlet boundary condition and $\wt{P}_{\geq1}=I-\wt{P}_{0}$. We define $\wt{P}_{0}^{j}$ and $\wt{P}_{1}^{j}$ to be respectively $\wt{P}_{0}$ and $\wt{P}_{\geq 1}$ acting on the $z_{j}$-variable, and
\begin{equation} \label{equ:outline of proof projection operator}
\wt{P}_{\alpha}=\wt{P}_{\alpha_{1}}^{1}\ccc \wt{P}_{\alpha_{k}}^{k}
\end{equation}
for a $k$-tuple $\alpha=(\alpha_{1},...,\alpha_{k})$ with $\alpha_{j}\in \lr{0,1}$ and adopt the notation $|\alpha|=\alpha_{1}+\ccc+\alpha_{k}.$ Then
\begin{align}\label{equ:outline of proof projection operator I}
I=\sum_{\alpha}\wt{\cp}_{\alpha},
\end{align}
where $I:L^{2}(\Om^{k})\rightarrow L^{2}(\Om^{k}).$

\begin{corollary}\label{High Energy estimates when k>1:energy bound}

Define
\begin{align*}
\wt{S}_{j}=(1-\Delta_{x_{j}}-\pa_{z_{j}}^{2}/L^{2}-1/L^{2})^{1/2},
\end{align*}
and write
\begin{align*}
\wt{S}^{(k)}=\prod_{j=1}^{k}\wt{S}_{j},\quad \lra{\nabla}^{(k)}=\prod_{j=1}^{k}\sqrt{1-\Delta_{r_{j}}}.
\end{align*}
Assume $L(N/L)^{\beta}\to 1^{-}$. Let $\wt{\psi}_{N,L}(t)=e^{it\wt{H}_{N,L}}\wt{\psi}_{N,L}(0)$ and $\lr{\wt{\gamma}_{N,L}^{(k)}(t)}$ be the associated marginal densities. Then for all $k\geq 0,$ we have the uniform-in-time bound
\begin{align} \label{equ:High Energy estimates when k>1:energy bound}
Tr \wt{S}^{(k)}\wt{\gamma}_{N,L}^{(k)}\wt{S}^{(k)}=\n{\wt{S}^{(k)}\wt{\psi}_{N,L}(t)}_{L^{2}}^{2}\leq C^{k}.
\end{align}
Consequently,
\begin{align} \label{equ:High Energy estimates when k>1:energy bound 2}
Tr \lra{\nabla}^{(k)}\wt{\gamma}_{N,L}^{(k)}\lra{\nabla}^{(k)}=\n{\lra{\nabla}^{(k)}\wt{\psi}_{N,L}(t)}_{L^{2}}^{2}\leq C^{k},
\end{align}
and
\begin{align} \label{equ:High Energy estimates when k>1:energy bound 3}
\n{\wt{\cp}_{\alpha}\wt{\psi}_{N,L}}_{L^{2}}\leq C^{k}L^{|\alpha|},\quad |Tr\wt{\cp}_{\alpha}\wt{\gamma}_{N,L}^{(k)}\wt{\cp}_{\beta}|\leq C^{k}L^{|\alpha|+|\beta|},
\end{align}
where $\wt{\cp}_{\alpha}$ and $\wt{\cp}_{\beta}$ are defined as in $(\ref{equ:outline of proof projection operator}).$
\end{corollary}
\begin{proof}

We notice that
\begin{align*}
&(\wt{S}_{j}^{2}\wt{\psi}_{N,L})(t,\mathbf{x}_{N},\mathbf{z}_{N})
=L^{N/2}(S_{j}^{2}\psi_{N,L})(t,\mathbf{x}_{N},L\mathbf{z}_{N}),\\
&\lrs{\wt{H}_{N,L}\wt{\psi}_{N,L}}(t,\bfx_{N},\bfz_{N})=L^{N/2}\lrs{H_{N,L}\psi_{N,L}}(t,\bfx_{N},L\bfz_{N}),
\end{align*}
where $\wt{\psi}_{N,L}$ is defined by $(\ref{eq:introduction rescaled solution})$. Thus, we have
\begin{align*}
&\n{\wt{S}^{(k)}\wt{\psi}_{N,L}}_{L^{2}}^{2}=\n{S^{(k)}\psi_{N,L}}_{L^{2}}^{2},\\
&\lra{\wt{\psi}_{N,L},(2+N^{-1}\wt{H}_{N,L}-1/L^{2})^{k}\wt{\psi}_{N,L}}=
\lra{\psi_{N,L},\lrs{2+N^{-1}H_{N,L}-1/L^{2}}^{k}\psi_{N,L}}.
\end{align*}
From estimate $(\ref{equ:High energy estimates when k>1})$ in Theorem $\ref{thm: Stability of matter when k>=1}$, we obtain
\begin{align*}
\n{\wt{S}^{(k)}\wt{\psi}_{N,L}(t)}_{L^{2}}^{2}\leq C^{k}\lra{\wt{\psi}_{N,L}(t),(2+N^{-1}\wt{H}_{N,L}-1/L^{2})^{k}\wt{\psi}_{N,L}(t)}.
\end{align*}
The term on the right-hand side is conserved, so
\begin{align*}
\n{\wt{S}^{(k)}\wt{\psi}_{N,L}(t)}_{L^{2}}^{2}\leq C^{k}\lra{\wt{\psi}_{N,L}(0),(2+N^{-1}\wt{H}_{N,L}-1/L^{2})^{k}\wt{\psi}_{N,L}(0)}.
\end{align*}
Applying the binomial theorem twice,
\begin{align*}
\n{\wt{S}^{(k)}\wt{\psi}_{N,L}(t)}_{L^{2}}^{2}\leq& C^{k}\sum_{j=0}^{k}\binom{k}{j}2^{j}\lra{\wt{\psi}_{N,L}(0),(N^{-1}\wt{H}_{N,L}-1/L^{2})^{k-j}
\wt{\psi}_{N,L}(0)}\\
\leq& C^{k}\sum_{j=0}^{k}\binom{k}{j}2^{j}C^{k-j}\\
=&C^{k}(2+C)^{k}\leq \wt{C}^{k},
\end{align*}
where we used initial condition in the second-to-last line. So we have established $(\ref{equ:High Energy estimates when k>1:energy bound})$. Combining $(\ref{equ:High Energy estimates when k>1:energy bound})$ and $(\ref{equ:Appendix operator inequality 3})$, estimate $(\ref{equ:High Energy estimates when k>1:energy bound 2})$ then follows. By $(\ref{equ:High Energy estimates when k>1:energy bound})$ and $(\ref{equ:Appendix operator inequality 2})$, we obtain the first inequality of $(\ref{equ:High Energy estimates when k>1:energy bound 3})$. By Lemma $\ref{lemma:Appendix kernel and trace}$,
\begin{align*}
Tr\wt{\cp}_{\alpha}\wt{\gamma}_{N,L}^{(k)}\wt{\cp}_{\beta}=\lra{\wt{\cp}_{\alpha}\wt{\psi}_{N,L},\wt{\cp}_{\beta}\wt{\psi}_{N,L}},
\end{align*}
so the second inequality of $(\ref{equ:High Energy estimates when k>1:energy bound 3})$ follows by Cauchy-Schwarz inequality.

\end{proof}

\subsection{Explainations on the assumptions} \label{section N,L}
We will explain the idea that we choose the mixed norm $\n{V}_{L_{z}^{\infty}L_{x}^{1}}$ and the relationship $L(N/L)^{\beta}\to 1^{-}$, both of which are different from the previous work, such as \cite{chen2013rigorous2dfrom3d,chen2017focusing,chen2017rigorous2dfocusing}.
In fact, to derive the 2D focusing NLS equations, the key point is that the interaction energy can be controlled by the kinetic energy, which is described by Theorem $\ref{thm: Stability of matter when k>=1}$ when $k=1$. By a scaling, we can see that the mixed norm $\n{V}_{L_{z}^{\infty}L_{x}^{1}}$ is reasonable and $L(N/L)^{\beta}$ should be bounded.
\subsubsection{}
We begin by setting up some notations for simplicity. Let
$$H_{ij}=S_{i}^{2}+S_{j}^{2}+H_{Iij},$$
and
\begin{align*}
&H_{Iij}=LV_{N,L}(r_{i}-r_{j}),
\end{align*}
where the subscript $I$ represents the interaction energy. Then, we can rewrite
\begin{align} \label{equ:stability of matter when k=1 hamiltonian operator formula}
1+N^{-1}H_{N,L}-1/L^{2}=\frac{1}{N(N-1)}\sum_{1\leq i<j\leq N}H_{ij}.
\end{align}
When we take $\psi_{N,L}=\phi_{L}^{\otimes N}$ with $\n{\phi_{L}}_{L^{2}}=1$, we find
\begin{align}
\lra{\psi_{N,L},\lrs{1+N^{-1}H_{N,L}-1/L^{2}}\psi_{N,L}}=\frac{1}{2}\lra{\phi_{L}^{\otimes 2},H_{12}\phi_{L}^{\otimes 2}}.
\end{align}
If $H_{12} \geq 0,$ we can deduce that
\begin{align}\label{equ:section N,L kinetic control interaction}
-L\int_{\Om_{L}^{2}}V_{N,L}(r_{1}-r_{2})|\phi_{L}(r_{1})\phi_{L}(r_{2})|^{2}dr_{1}dr_{2}\leq & C(V)\lra{S_{1}^{2}\phi_{L}(r_{1})\phi_{L}(r_{2}),
\phi_{L}(r_{1})\phi_{L}(r_{2})}\\
\leq& C(V)\lra{(1-\Delta_{r_{1}})\phi_{L}(r_{1})\phi_{L}(r_{2}),
\phi_{L}(r_{1})\phi_{L}(r_{2})},\notag
\end{align}
where $C(V)$ depends on $V$. Moreover, if we assume $C(V)=C\n{V}_{X}$ where $\n{\cdot}_{X}$ is a norm, by a scaling argument, it should satisfy
$$\n{V}_{X}\leq \lambda^{-2}\n{V(\cdot/\lambda)}_{X} ,\quad \forall \lambda \in(0,1) .$$
Indeed, we take $V^{\lambda}(\cdot)=V(\cdot/\lambda)$ and $\phi_{L}^{\lambda}(\cdot)=\lambda^{-3/2}\phi_{L}(\cdot/\lambda)$ to replace $V$ and $\phi_{L}$ respectively.
Since we take the periodic condition in the $x$-direction, a scaling argument can only used for the function supported in the interior of the domain. Thus, we consider the test function $\phi_{L}\in C_{c}^{\infty}(\Om_{L})$, the space of smooth functions compactly supported in $(-\pi,\pi)^{2}\times (-L\pi/2,L\pi/2)$. For every $\phi_{L}\in C_{c}^{\infty}(\Om_{L})$, we have
\begin{align}\label{equ:section N,L interaction scaling}
-L\int_{\Om_{L}^{2}} V_{N,L}(r_{1}-r_{2})|\phi_{L}(r_{1})\phi_{L}(r_{2})|^{2}dr_{1}dr_{2}
=-L\int_{\Om_{L}^{2}} V_{N,L}^{\lambda}(r_{1}-r_{2})|\phi_{L}^{\lambda}(r_{1})\phi_{L}^{\lambda}(r_{2})|^{2}dr_{1}dr_{2},
\end{align}
and
\begin{align}\label{equ:section N,L kinetic scaling}
C\n{V^{\lambda}}_{X}\lra{(1-\Delta_{r_{1}})\phi_{L}^{\lambda}(r_{1})\phi_{L}^{\lambda}(r_{2}),\phi_{L}^{\lambda}(r_{1})\phi_{L}^{\lambda}(r_{2})}\leq C\n{V^{\lambda}}_{X}\lambda^{-2}
\lra{(1-\Delta_{r_{1}})\phi_{L}(r_{1})\phi_{L}(r_{2}),\phi_{L}(r_{1})\phi_{L}(r_{2})}.
\end{align}
If there exists a $\lambda_{0}\in (0,1)$ such that $\lambda_{0}^{-2}\n{V(\cdot/\lambda_{0})}_{X}\leq q_{0}\n{V}_{X}$ for some $q_{0}\in (0,1)$, we take $\lambda=\lambda_{0}$. Putting $(\ref{equ:section N,L kinetic control interaction})$ $(\ref{equ:section N,L interaction scaling})$ and $(\ref{equ:section N,L kinetic scaling})$ together, we get
$$-L\int_{\Om_{L}^{2}} V_{N,L}(r_{1}-r_{2})|\phi_{L}(r_{1})\phi_{L}(r_{2})|^{2}dr_{1}dr_{2}\leq q_{0}C\n{V}_{X}\lra{(1-\Delta_{r_{1}})\phi_{L}(r_{1})\phi_{L}(r_{2}),\phi_{L}(r_{1})\phi_{L}(r_{2})},$$
for all $\phi_{L}\in C_{c}^{\infty}(\Om_{L})$. Iterating the process, it will lead to a contradiction for $q_{0}<1$.

On the one hand, the common norm $\n{\cdot}_{L^{1}}$ cannot satisfy the above requirement, since
$$\n{V}_{L^{1}}=\lambda^{-3}\n{V^{\lambda}}_{L^{1}}\geq \lambda^{-2}\n{V^{\lambda}}_{L^{1}},\quad \forall \lambda\in (0,1).$$
 On the other hand, we note that
$$\n{V}_{L_{z}^{\infty}L_{x}^{1}}=\lambda^{-2}\n{V(\cdot/\lambda)}_{L_{z}^{\infty}L_{x}^{1}},\quad \forall \lambda\in (0,1).$$
That is, the mixed norm $L_{z}^{\infty}L_{x}^{1}$ satisfies the requirement. Indeed, we can establish a general Lemma $\ref{lemma:Stability of matter when k=1 interaction estimate}$ in Section $\ref{section k=1}$.

\subsubsection{}
 To derive the relationship between $N$ and $L$, the key point is also that the interaction energy can be controlled by the kinetic energy. More precisely, let us consider the rescaled system. We take the test function $\wt{\phi}(r)=f(x)\wt{g}(z)$, where $f(x)\in C_{c}^{\infty}((-\pi,\pi)^{2})$, $\wt{g}(z)\in C_{c}^{\infty}(-\pi/2,\pi/2)$ and
 $$\n{f}_{L^{2}}=\n{\wt{g}}_{L^{2}}=1.$$
 Define
 \begin{align}
 &f_{\varepsilon}(x)=\frac{1}{\varepsilon}f(x/\varepsilon)\in C_{c}^{\infty}((-\pi,\pi)^{2}),\quad \varepsilon \in (0,1),\\
 &\wt{g}_{\lambda}(z)=\frac{1}{\sqrt{\lambda}}\wt{g}(z/\lambda)\in C_{c}^{\infty}(-\pi/2,\pi/2),\quad \lambda\in (0,1),\\
&\wt{\phi}_{\varepsilon,\lambda}(r)=f_{\varepsilon}(x)\wt{g}_{\lambda}(z),\\
&\wt{\psi}_{\varepsilon,\lambda}(r_{1},r_{2})=\wt{\phi}_{\varepsilon,\lambda}(r_{1})\wt{\phi}_{\varepsilon,\lambda}(r_{2}).
 \end{align}
where $f_{\varepsilon}(x)$ should be considered as a periodic extension and $f_{\varepsilon}(x)\in C^{\infty}(\T^{2}).$

The interaction energy is
\begin{align}\label{equ:introduction the interaction energy rescaled}
&\int \wt{V}_{N,L}(r_{1}-r_{2})|f_{\varepsilon}(x_{1})\wt{g}_{\lambda}(z_{1})f_{\varepsilon}(x_{2})\wt{g}_{\lambda}(z_{2})|^{2}dr_{1}dr_{2}\\
=&\int \wt{V}_{N,L}(\varepsilon(x_{1}-x_{2}),\lambda(z_{1}-z_{2}))|f(x_{1})\wt{g}(z_{1})f(x_{2})\wt{g}(z_{2})|^{2}dr_{1}dr_{2}.\notag
\end{align}

The kinetic energy is
\begin{align}\label{equ:introduction the kinetic energy rescaled}
&\lra{\wt{S}_{1}^{2}\wt{\psi}_{\varepsilon,\lambda},\wt{\psi}_{\varepsilon,\lambda}}\\
=&\n{f_{\varepsilon}}_{L^{2}}^{2}\n{\wt{g}_{\lambda}}_{L^{2}}^{2}\lrs{\lra{(1-\Delta_{x})f_{\varepsilon}(x)\wt{g}_{\lambda}(z),f_{\varepsilon}(x)\wt{g}_{\lambda}(z)}
+\frac{1}{L^{2}}\lra{f_{\varepsilon}(x)(-\pa_{z}^{2}-1)\wt{g}_{\lambda}(z),f_{\varepsilon}(x)\wt{g}_{\lambda}(z)}}\notag \\
=&\n{f_{\varepsilon}}_{L^{2}}^{2}\n{\wt{g}_{\lambda}}_{L^{2}}^{2}\lrs{
\n{\nabla_{x}f_{\varepsilon}}_{L^{2}}^{2}\n{\wt{g}_{\lambda}}_{L^{2}}^{2}+
\n{f_{\varepsilon}}_{L^{2}}^{2}\n{\wt{g}_{\lambda}}_{L^{2}}^{2}
+\frac{1}{L^{2}}\n{f_{\varepsilon}}_{L^{2}}^{2}\n{\pa_{z}\wt{g}_{\lambda}}_{L^{2}}^{2}
-\frac{1}{L^{2}}\n{f_{\varepsilon}}_{L^{2}}^{2}\n{\wt{g}_{\lambda}}_{L^{2}}^{2}}\notag\\
=&\n{f}_{L^{2}}^{2}\n{\wt{g}}_{L^{2}}^{4}\lrs{\frac{\n{\nabla_{x} f}_{L^{2}}^{2}}{\varepsilon^{2}}+\n{f}_{L^{2}}^{2}}+\frac{1}{L^{2}}\n{f}_{L^{2}}^{4}\n{\wt{g}}_{L^{2}}^{2}\lrs{\frac{\n{\pa_{z} \wt{g}}_{L^{2}}^{2}}{\lambda^{2}}-\n{\wt{g}}_{L^{2}}^{2}}.\notag
\end{align}

When we take $\lambda^{-1}=L(N/L)^{\beta}$, $\varepsilon^{-1}=(N/L)^{\beta}$, the interaction energy is equal to
$$L(N/L)^{3\beta}\int V(x_{1}-x_{2},z_{1}-z_{2})|f(x_{1})\wt{g}(z_{1})f(x_{2})\wt{g}(z_{2})|^{2}dr_{1}dr_{2}$$
and the kinetic energy is controlled by
$$(N/L)^{2\beta}+\frac{L^{2}(N/L)^{2\beta}}{L^{2}}=2(N/L)^{2\beta}.$$
Since $L(N/L)^{\beta}\to \infty,$ the interaction energy cannot be controlled by the kinetic energy. Therefore, it implies that we should consider the case $L(N/L)^{\beta} \leq C.$ On the other hand, to make the limit of $\wt{V}_{N,L}$ exist, we should take $L(N/L)^{\beta}$ to be a constant or tend to $0$. For the case $L(N/L)^{\beta}\to 0,$ we note that the limit of $\wt{V}_{N,L}$ equals to $0$, which is not sufficient to derive the cubic NLS equation. Hence, we only consider the case $L(N/L)^{\beta}\to 1$ and it works the same for $L(N/L)^{\beta}\to R_{0}.$
\begin{remark}
 A similar argument can also apply to \cite{chen2013rigorous2dfrom3d} in the focusing setting. To control the interaction energy instead of dropping it like the defocusing case, it also needs an extra condition $(N\sqrt{\om})^{\beta}\om^{-1/2} \leq C.$
\end{remark}

\subsection{Focusing energy estimates when $k=1$} \label{section k=1}
\begin{theorem} \label{thm:Stability of matter when k=1, beta 1/2}
Assume $L(N/L)^{\beta}\to 1^{-}$, $\beta<\frac{1}{2},$ and $\n{V}_{L_{z}^{\infty}L_{x}^{1}}\leq \frac{2\alpha}{C_{gn}^{4}}$ for some $\alpha\in (0,1)$, then let $c_{0}=\min\lrs{\frac{1-\alpha}{\sqrt{2}},\frac{1}{2}},$ we have $\forall C_{0}>0$, there exists an $N_{0}>0$ such that
\begin{align}\label{equ:stability of matter the main estimate}
\lra{\psi_{N,L},(C_{0}+1+N^{-1}H_{N,L}-1/L^{2})\psi_{N,L}}\geq c_{0}\n{S_{1}\psi_{N,L}}_{L^{2}}^{2},
\end{align}
for all $N>N_{0}$ and for all $\psi_{N,L}\in L_{s}^{2}(\Om_{L}^{\otimes N}).$
\end{theorem}

We divide the proof of Theorem $\ref{thm:Stability of matter when k=1, beta 1/2}$ into two parts. The first is Theorem $\ref{thm:Stability of matter when k=1}$ where the energy bound holds for $\be<1/3$ and the second is Theorem $\ref{thm:Stability of matter when k=1, beta 1/2, bootstrap}$ which implies Theorem $\ref{thm:Stability of matter when k=1, beta 1/2}$. In this section, we prove Theorem $\ref{thm:Stability of matter when k=1}$ with $\be<1/3$.
\begin{theorem} \label{thm:Stability of matter when k=1}
Assume $L(N/L)^{\beta}\to 1^{-}$, $\beta<\frac{1}{3},$ and $\n{V}_{L_{z}^{\infty}L_{x}^{1}}\leq \frac{2\alpha}{C_{gn}^{4}}$ for some $\alpha\in (0,1)$, then $\forall C_{0}>0$, there exists an $N_{0}>0$ such that
\begin{align}\label{equ:stability of matter the main estimate}
\lra{\psi_{N,L},(C_{0}+1+N^{-1}H_{N,L}-1/L^{2})\psi_{N,L}}\geq (1-\alpha)\n{S_{1}\psi_{N,L}}_{L^{2}}^{2},
\end{align}
for all $N>N_{0}$ and for all $\psi_{N,L}\in L_{s}^{2}(\Om_{L}^{\otimes N}).$
\end{theorem}

The key of the proof of Theorem $\ref{thm:Stability of matter when k=1}$ is the following theorem.

\begin{theorem}\label{thm:Stability of matter when k=1 key theorem}
Assume $L(N/L)^{\beta}\to 1^{-}$, $\beta<\frac{1}{3},$ and $\n{V}_{L_{z}^{\infty}L_{x}^{1}}\leq \frac{2\alpha}{C_{gn}^{4}}$ for some $\alpha\in (0,1)$, define the operator
\begin{align}
H_{ij,\alpha}=\alpha S_{i}^{2}+\alpha S_{j}^{2}+H_{Iij}.
\end{align}
Then $\forall C_{0}>0$, there exists an $N_{0}>0$ such that
$$\lra{\psi_{N,L},(2C_{0}+H_{12,\alpha})\psi_{N,L}}\geq 0,$$
for all $N>N_{0}$ and for all $\psi_{N,L}\in L_{s}^{2}(\Om_{L}^{\otimes N}).$
\end{theorem}
\begin{proof}[\textbf{Proof of Theorem $\ref{thm:Stability of matter when k=1}$ assuming Theorem $\ref{thm:Stability of matter when k=1 key theorem}$}] Using formula $(\ref{equ:stability of matter when k=1 hamiltonian operator formula})$ and the symmetry of $\psi_{N,L}$, we have
\begin{align*}
&\lra{\psi_{N,L},\lrs{C_{0}+1+N^{-1}H_{N,L}-1/L^{2}}\psi_{N,L}}\\
=&\frac{1}{N(N-1)}\sum_{1\leq i<j\leq N}\lra{\psi_{N,L},(2C_{0}+H_{ij})\psi_{N,L}}\\
=&\frac{1}{2}\lra{\psi_{N,L},(2C_{0}+H_{12})\psi_{N,L}}\\
=&\frac{1}{2}\lra{\psi_{N,L},(2C_{0}+H_{12,\alpha})\psi_{N,L}}+\frac{1-\alpha}{2}
\lra{\psi_{N,L},(S_{1}^{2}+S_{2}^{2})\psi_{N,L}}\\
\geq& (1-\alpha)\n{S_{1}\psi_{N,L}}_{L^{2}}^{2}.
\end{align*}

\end{proof}

Next, we turn our attention onto the proof of Theorem $\ref{thm:Stability of matter when k=1 key theorem}$. Under the assumption $L(N/L)^{\beta}\to 1^{-}$, the renormalized kinetic energy can control the potential energy.

\begin{lemma}\label{lemma:Stability of matter when k=1 interaction estimate}
Assume $L(N/L)^{\beta}\to 1^{-}$, $M\geq 1$, then for all $\psi_{M,L}\in L_{s}^{2}(\Om_{L}^{\otimes M})$ with
$\n{\psi_{M,L}}_{L^{2}}=1,$ we have
\begin{align}\label{equ:product estimate equation}
L\int_{\Om_{L}^{2}} |V_{N,L}(r_{1}-r_{2})|\rho_{M,L}(r_{1})\rho_{M,L}(r_{2}) dr_{1}dr_{2}\leq C_{gn}^{4}\n{V}_{L_{z}^{\infty}L_{x}^{1}}\lra{S_{1}^{2}\psi_{M,L},\psi_{M,L}},
\end{align}
where density function $\rho_{M,L}(r_{1}):=\int\ccc\int |\psi_{M,L}|^{2}(r_{1},...,r_{M})dr_{2}\ccc dr_{M}$. Especially, if $\psi_{M,L}=\phi_{L}^{\otimes M}$ with $\n{\phi_{L}}_{L^{2}}=1$, then
\begin{align}\label{equ:product estimate equation 2}
L\int_{\Om_{L}^{2}} |V_{N,L}(r_{1}-r_{2})|\phi_{L}(r_{1})\phi_{L}(r_{2})|^{2} dr_{1}dr_{2}\leq C_{gn}^{4}{\n{V}_{L_{z}^{\infty}L_{x}^{1}}}\lra{S_{1}^{2}\phi_{L},\phi_{L}}.
\end{align}

\end{lemma}
\begin{proof}
Using Cauchy-Schwarz inequality and Young's convolution inequality,  we get
\begin{align}\label{eq:Stability of matter when k=1 interaction estimate  young inequality}
L\int_{\Om_{L}^{2}} |V_{N,L}(r_{1}-r_{2})|\rho_{M,L}(r_{1})\rho_{M,L}(r_{2})  dr_{1}dr_{2}\leq &
L\n{V_{N,L}*\rho_{M,L}}_{L_{z}^{\infty}L_{x}^{2}}\n{\rho_{M,L}}_{L_{z}^{1}L_{x}^{2}}\\
\leq& L\n{V_{N,L}}_{L_{z}^{\infty}L_{x}^{1}}\n{\rho_{M,L}}_{L_{z}^{1}L_{x}^{2}}^{2}.\notag
\end{align}
For $\n{\rho_{M,L}}_{L_{z}^{1}L_{x}^{2}}^{2}$ on the right-hand side of $(\ref{eq:Stability of matter when k=1 interaction estimate  young inequality})$, we use 2D inhomogeneous Gagliardo-Nirenberg inequality $(\ref{equ: 2d inhomogeneous GN inequality})$,
\begin{align} \label{eq:Stability of matter when k=1 interaction estimate}
\n{\rho_{M,L}}_{L_{z}^{1}L_{x}^{2}}^{2}=&\n{\sqrt{\rho_{M,L}}}_{L_{z}^{2}L_{x}^{4}}^{4}\leq C_{gn}^{4} \n{\sqrt{\rho_{M,L}}}_{L_{z}^{2}L_{x}^{2}}^{2}\n{\sqrt{1-\Delta_{x}}\sqrt{\rho_{M,L}}}_{L_{z}^{2}L_{x}^{2}}^{2}\\
=&C_{gn}^{4} \n{\sqrt{\rho_{M,L}}}_{L_{z}^{2}L_{x}^{2}}^{2}\lrs{\n{\sqrt{\rho_{M,L}}}_{L_{z}^{2}L_{x}^{2}}^{2}+\n{\nabla_{x}\sqrt{\rho_{M,L}}}_{L_{z}^{2}L_{x}^{2}}^{2}}.\notag
\end{align}
By the Hoffman-Ostenhof inequality (\ref{equ:Hoffman-Ostenhof inequality appendix}), we have
\begin{align} \label{eq:Stability of matter when k=1 interaction estimate hoffman-ostenhof inequality}
\n{\nabla_{x}\sqrt{\rho_{M,L}}}_{L^{2}}^{2}\leq \lra{-\Delta_{x_{1}}\psi_{M,L},\psi_{M,L}}.
\end{align}
Now, puttting $(\ref{eq:Stability of matter when k=1 interaction estimate hoffman-ostenhof inequality})$ into $(\ref{eq:Stability of matter when k=1 interaction estimate})$ and $(\ref{eq:Stability of matter when k=1 interaction estimate})$ into $(\ref{eq:Stability of matter when k=1 interaction estimate young inequality})$, we have
\begin{align}
L\int_{\Om_{L}^{2}} |V_{N,L}(r_{1}-r_{2})|\rho_{M,L}(r_{1})\rho_{M,L}(r_{2}) dr_{1}dr_{2}\leq C_{gn}^{4}L{\n{V_{N,L}}_{L_{z}^{\infty}L_{x}^{1}}}\lra{(1-\Delta_{x_{1}})\psi_{M,L},\psi_{M,L}}.
\end{align}

Noting that $\n{V_{N,L}}_{L_{z}^{\infty}L_{x}^{1}}=(N/L)^{\beta}\n{V}_{L_{z}^{\infty}L_{x}^{1}}$, with the assumption $L(N/L)^{\beta}\to 1^{-},$ we conclude that
\begin{align} \label{eq:Stability of matter when k=1 interaction estimate conclusion}
L\int_{\Om_{L}^{2}} |V_{N,L}(r_{1}-r_{2})|\rho_{M,L}(r_{1})\rho_{M,L}(r_{2}) dr_{1}dr_{2}\leq C_{gn}^{4}{\n{V}_{L_{z}^{\infty}L_{x}^{1}}}\lra{(1-\Delta_{x_{1}})\psi_{M,L},\psi_{M,L}}.
\end{align}
Since $S_{1}^{2}\geq (1-\Delta_{x_{1}})$, we arrive at the estimate $(\ref{equ:product estimate equation})$ from estimate $(\ref{eq:Stability of matter when k=1 interaction estimate conclusion})$.

\end{proof}

The following lemma is used to estimate the two-body interaction energy by a one-body term.
\begin{lemma}\label{lemma:Estimating the two-body interaction}
If $V\in C^{\infty}_{c}(\T^{2}\times \R)$ and has a positive Fourier transform $\widehat{V}\geq 0$, then for all real function $\eta \in L^{1}(\T^{2}\times \R)$
\begin{align}\label{equ:Estimating the two-body interaction}
\sum_{1\leq j<k\leq N}V(r_{j}-r_{k})\geq\frac{1}{(2\pi)^{3}}\sum_{j=1}^{N}\eta *V(r_{j})-\frac{1}{2(2\pi)^{6}}\int_{\T^{2}\times \R}\int_{\T^{2}\times \R}V(r_{1}-r_{2})\eta(r_{1})\eta(r_{2})dr_{1}dr_{2}-\frac{N}{2}V(0).
\end{align}
\end{lemma}
\begin{proof}
With the Fourier inversion formula,
\begin{align}
&\sum_{1\leq j<k\leq N}V(r_{j}-r_{k})\\
=&\int \sum_{1\leq j<k\leq N}e^{i\xi\cdot (r_{j}-r_{k})}\widehat{V}(\xi)d\xi \notag\\
=&\frac{1}{2}\int \bbabs{\sum_{j=1}^{N}e^{i\xi \cdot r_{j}}-\widehat{\eta}(\xi)}^{2}\widehat{V}(\xi)d\xi+\frac{1}{(2\pi)^{3}}\sum_{j=1}^{N}\eta *V(r_{j}) \notag\\
&-\frac{1}{2(2\pi)^{6}}\int_{\T^{2}\times \R}\int_{\T^{2}\times \R}V(r_{1}-r_{2})\eta(r_{1})\eta(r_{2})dr_{1}dr_{2}-\frac{N}{2}V(0) \notag\\
\geq&\frac{1}{(2\pi)^{3}}\sum_{j=1}^{N}\eta *V(r_{j})-\frac{1}{2(2\pi)^{6}}\int_{\T^{2}\times \R}\int_{\T^{2}\times \R}V(r_{1}-r_{2})\eta(r_{1})\eta(r_{2})dr_{1}dr_{2}-\frac{N}{2}V(0), \notag
\end{align}
where $\xi=(n_{1},n_{2},\tau)$ and $\int \cdot d\xi$ is short for $\int_{\R} \sum_{n_{1},n_{2}} \cdot\ d\tau$.
\end{proof}
\begin{remark}
In our setting, the integral region is $\Om_{L}$. To use Lemma $\ref{lemma:Estimating the two-body interaction}$, $V_{N,L}$ should be understood as the periodic extension in the $x$-direction and zero extension in the $z$-direction of the rescaled $V$ which is compactly supported on $\Om.$ That is, $V_{N,L}\in C_{c}^{\infty}(\T^{2}\times \R)$. Similarly, $\rho_{N,L}(r)$ and $\psi_{N,L}$ should be seen as $\rho_{N,L}(r)1_{\Om_{L}}(r) \in L^{1}(\T^{2}\times \R)$ and $\psi_{N,L}1_{\Om_{L}^{\otimes N}}$ respectively.

In case $\widehat{V}_{N,L}\geq0$, if we take $V=V_{N,L}$ and $\eta=(2\pi)^{3}N\rho_{N,L}$, we obtain
\begin{align} \label{equ:Estimating the two-body interaction positive case}
&\sum_{1\leq j<k\leq N}\lra{\psi_{N,L},V_{N,L}(r_{j}-r_{k})\psi_{N,L}}\\
&\geq N\sum_{j=1}^{N}\int \rho_{N,L}* V_{N,L}(r_{j})|\psi_{N,L}|^{2}(\mathbf{r}_{N})d\mathbf{r}_{N}-\frac{N^{2}}{2}\lra{\rho_{N,L}*V_{N,L},\rho_{N,L}}-\frac{N}{2}V_{N,L}(0)\notag\\
&= \frac{N^{2}}{2}\int \rho_{N,L}(r_{1})\rho_{N,L}(r_{2})V_{N,L}(r_{1}-r_{2})dr_{1}dr_{2}-\frac{N(N/L)^{3\beta}V(0)}{2}\notag\\
&\geq -\frac{N(N/L)^{3\beta}V(0)}{2},\notag
\end{align}
where we have used
$$\int \rho_{N,L}(r_{1})\rho_{N,L}(r_{2})V_{N,L}(r_{1}-r_{2})dr_{1}dr_{2}=(2\pi)^{6}\int |\widehat{\rho}_{N,L}(\xi)|^{2}\widehat{V}_{N,L}(\xi)d\xi$$
and $\widehat{V}_{N,L}\geq 0$ in the last inequality.

By estimate $(\ref{equ:Estimating the two-body interaction positive case})$, we have
\begin{align*}
\lra{\psi_{N,L},(2C_{0}+H_{12,\alpha})\psi_{N,L}}\geq 2C_{0}-\frac{(N/L)^{3\beta}V(0)}{(N-1)/L}\geq 0,
\end{align*}
as long as $\beta<1/3$ and $N/L$ is large enough.

Hence, we have established Theorem $\ref{thm:Stability of matter when k=1 key theorem}$ if $\widehat{V}_{N,L}\geq 0$. Next, we will use Lemma $\ref{lemma:Estimating the two-body interaction}$ to deal with a general interaction function $V$.
\end{remark}

\begin{proof}[\textbf{Proof of Theorem $\ref{thm:Stability of matter when k=1 key theorem}$}]

For general $V$, we consider $N=2M$ particles which we split into two groups of $M$. For the case $N=2M+1$, the proof works the same if we split the system into two groups of $M$ and $M+1$. We denote the first $M$ variables by $r_{1},...,r_{M}$ and the others by $r'_{1}=r_{M+1},...,r'_{M}=r_{2M}.$ We decompose $V_{N,L}=V_{N,L}^{+}-V_{N,L}^{-}$ on the Fourier side where $\widehat{V_{N,L}^{+}}=(\widehat{V_{N,L}})_{+}\geq 0$ and $\widehat{V_{N,L}^{-}}=(\widehat{V_{N,L}})_{-}\geq 0$. By its symmetry in the $2M$ variables, we rewrite
\begin{align*}
&\frac{1}{2}\lra{\psi_{N,L},H_{I12}\psi_{N,L}}\\
=&\frac{L}{2M(2M-1)}\bblra{\psi_{2M,L},\sum_{1\leq j<k\leq 2M}V_{N,L}(r_{j}-r_{k})\psi_{2M,L}}\\
=&\frac{L}{M(M-1)}\bblra{\psi_{2M,L},\sum_{1\leq j<k\leq M}V_{N,L}^{+}(r_{j}-r_{k})\psi_{2M,L}}\\
&+\frac{L}{M(M-1)}\bblra{\psi_{2M,L},\sum_{1\leq l<m\leq M}V_{N,L}^{-}(r'_{l}-r'_{m})\psi_{2M,L}}-\frac{L}{M^{2}}
\bblra{\psi_{2M,L},\sum_{j=1}^{M}\sum_{l=1}^{M}V_{N,L}^{-}(r_{j}-r'_{l})\psi_{2M,L}}.
\end{align*}
This means that
\begin{align}
\frac{1}{2}\lra{\psi_{2M,L},H_{I12}\psi_{2M,L}}=\lra{\psi_{2M,L},I_{M,L}\psi_{2M,L}},
\end{align}
where
\begin{align}
I_{M,L}=&\frac{L}{M(M-1)}\sum_{1\leq j<k\leq M}V_{N,L}^{+}(r_{j}-r_{k})+\frac{L}{M(M-1)}\sum_{1\leq l<m\leq M}V_{N,L}^{-}(r'_{l}-r'_{m})\\
&-\frac{L}{M^{2}}\sum_{j=1}^{M}
\sum_{l=1}^{M}V_{N,L}^{-}(r_{j}-r'_{l})\notag .
\end{align}
Then we have
\begin{align}
\lra{\psi_{2M,L},(2C_{0}+H_{12,\alpha})\psi_{2M,L}}=&\lra{\psi_{2M,L},(2C_{0}+2I_{M,L}+\alpha(S_{1}^{2}+S_{2}^{2}))\psi_{2M,L}}.
\end{align}
Thus, in order to bound the interaction $2C_{0}+H_{12,\alpha}$ from below, it suffices to consider $I_{M,L}$.
We fix the variables $r'_{1},...,r'_{M}$ in the second group. For simplicity, we use the notation $\lra{\cdot,\cdot}_{\mathbf{r_{M}}}$ to denote the integral only in the variables $\mathbf{r_{M}}:=(r_{1},..,r_{M})$. We denote the one-particle density by
\begin{align}
\rho_{M,L}(r,\mathbf{r'_{M}}):=\int_{\Om_{L}}\ccc \int_{\Om_{L}}|\psi_{2M,L}|^{2}(r,r_{2},...,r_{M},\mathbf{r'_{M}})dr_{2}\ccc dr_{M}.
\end{align}
Our goal is to get
\begin{align} \label{equ:Estimating the two-body interaction key estimate}
&\lra{\psi_{2M,L},(2C_{0}+2I_{M,L}+\alpha(S_{1}^{2}+S_{2}^{2}))\psi_{2M,L}}_{\mathbf{r_{M}}}\\
\geq &\lrs{2\alpha-C_{gn}^{4}\n{V}_{L_{z}^{\infty}L_{x}^{1}}}\lra{\psi_{2M,L},S_{1}^{2}\psi_{2M,L}}_{\mathbf{r_{M}}}+\int\rho_{M,L}(r,\mathbf{r'_{M}})dr
\lrs{2C_{0}-\frac{LV_{N,L}^{+}(0)+LV_{N,L}^{-}(0)}{(M-1)}}.\notag
\end{align}

First, we may assume $ \int\rho_{M,L}(r,\mathbf{r'_{M}})dr>0.$ The case $ \int\rho_{M,L}(r,\mathbf{r'_{M}})dr=0$ is easier and will be presented later. By using Lemma $\ref{lemma:Estimating the two-body interaction}$ with
$$V=V_{N,L}^{+},\quad  \eta= \frac{(2\pi)^{3}M\rho_{M,L}(r,\mathbf{r'_{M}})}{\int\rho_{M,L}(r,\mathbf{r'_{M}})dr},$$
 we have
\begin{align}
&\lra{\psi_{2M,L},I_{M,L}\psi_{2M,L}}_{\mathbf{r_{M}}}\\
\geq &\frac{ML}{2(M-1)\int \rho_{M,L}(r,\mathbf{r'_{M}})dr}
\int V_{N,L}^{+}(r_{1}-r_{2})\rho_{M,L}(r_{1},\mathbf{r'_{M}})\rho_{M,L}(r_{2},\mathbf{r'_{M}})dr_{1}dr_{2}\notag\\
&-\frac{LV_{N,L}^{+}(0)\int  \rho_{M,L}(r,\mathbf{r'_{M}}) dr}{2(M-1)}
+\frac{L}{M(M-1)}\sum_{1\leq l<m\leq M}V_{N,L}^{-}(r'_{l}-r'_{m})\int \rho_{M,L}(r,\mathbf{r'_{M}})dr\notag\\
&-\frac{L}{M}\sum_{l=1}^{M}\rho_{M,L}*V_{N,L}^{-}(r'_{l}).\notag
\end{align}
Next we use again Lemma $\ref{lemma:Estimating the two-body interaction}$ with
$$V=V_{N,L}^{-},\quad \eta =\frac{(2\pi)^{3}(M-1)\rho_{M,L}(r,\mathbf{r'_{M}})}{\int\rho_{M,L}(r,\mathbf{r'_{M}})dr},$$
 and obtain
\begin{align}
&\frac{L}{M(M-1)}\sum_{1\leq l<m\leq M}V_{N,L}^{-}(r'_{l}-r'_{m})\int \rho_{M,L}(r,\mathbf{r'_{M}})dr-\frac{L}{M}\sum_{l=1}^{M}\rho_{M,L}*V_{N,L}^{-}(r'_{l})\\
\geq&-\frac{(M-1)}{2M}\frac{L}{\int\rho_{M,L}(r,\mathbf{r'_{M}})dr}\int V_{N,L}^{-}(r_{1}-r_{2}) \rho_{M,L}(r_{1},\mathbf{r'_{M}})  \rho_{M,L}(r_{2},\mathbf{r'_{M}}) dr_{1}dr_{2}\notag\\
&-\frac{LV_{N,L}^{-}(0)\int\rho_{M,L}(r,\mathbf{r'_{M}})dr}{2(M-1)}.\notag
\end{align}
Thus, we have
\begin{align}
&\lra{\psi_{2M,L},I_{M,L}\psi_{2M,L}}_{\mathbf{r_{M}}}\\
\geq& \int\rho_{M,L}(r,\mathbf{r'_{M}})dr \int \frac{L}{2}V_{N,L}(r_{1}-r_{2}) \frac{\rho_{M,L}(r_{1},\mathbf{r'_{M}})}{\int\rho_{M,L}(r,\mathbf{r'_{M}})dr}  \frac{\rho_{M,L}(r_{2},\mathbf{r'_{M}})}{\int\rho_{M,L}(r,\mathbf{r'_{M}})dr}  dr_{1}dr_{2}\notag \\
&-\int\rho_{M,L}(r,\mathbf{r'_{M}})dr\frac{LV_{N,L}^{+}(0)+LV_{N,L}^{-}(0)}{2(M-1)} \notag.
\end{align}
By Lemma $\ref{lemma:Stability of matter when k=1 interaction estimate} $, we obtain
\begin{align}
&\lra{\psi_{2M,L},(2C_{0}+2I_{M,L}+\alpha(S_{1}^{2}+S_{2}^{2}))\psi_{2M,L}}_{\mathbf{r_{M}}}\\
\geq &\lrs{2\alpha-C_{gn}^{4}\n{V}_{L_{z}^{\infty}L_{x}^{1}}}\lra{\psi_{2M,L},S_{1}^{2}\psi_{2M,L}}_{\mathbf{r_{M}}}+\int\rho_{M,L}(r,\mathbf{r'_{M}})dr
\lrs{2C_{0}-\frac{LV_{N,L}^{+}(0)+LV_{N,L}^{-}(0)}{(M-1)}}.\notag
\end{align}
We arrive at the estimate $(\ref{equ:Estimating the two-body interaction key estimate})$ for the case $\int\rho_{M,L}(r,\mathbf{r'_{M}})dr>0.$

Next, if $ \int\rho_{M,L}(r,\mathbf{r'_{M}})dr=0,$ we can deduce that $\rho_{M,L}(r,\mathbf{r'_{M}})=0$ due to the nonnegativity and smoothness of $\rho_{M,L}$. Then, we have
\begin{align}
&\lra{\psi_{2M,L},I_{M,L}\psi_{2M,L}}_{\mathbf{r_{M}}}\\
=&\frac{L}{M(M-1)}\sum_{1\leq j<k\leq M}\int V_{N,L}^{+}(r_{j}-r_{k})|\psi_{M,L}(\mathbf{r}_{M},\mathbf{r}_{M}')|^{2}d\mathbf{r}_{M}\notag \\
&+\frac{L}{M(M-1)}\sum_{1\leq l<m\leq M}V_{N,L}^{-}(r'_{l}-r'_{m})\int \rho_{M,L}(r,\mathbf{r}_{M}')dr \notag \\
&-\frac{L}{M^{2}}\sum_{j=1}^{M}
\sum_{l=1}^{M}\int V_{N,L}^{-}(r_{j}-r'_{l})\rho_{M,L}(r_{j},\mathbf{r}_{M}')dr_{j}\notag\\
=&I+II+III. \notag
\end{align}
Since $V_{N,L}^{+}$ has a positive Fourier transform, we have $I\geq 0.$ By the fact that $\rho_{M,L}(r,\mathbf{r'_{M}})=0$, we obtain $II=III=0$.
That is, the estimate $(\ref{equ:Estimating the two-body interaction key estimate})$ still holds.

Hence, when $\n{V}_{L_{z}^{\infty}L_{x}^{1}}\leq \frac{2\alpha}{C_{gn}^{4}}$, we have
\begin{align}
\lra{\psi_{2M,L},(2C_{0}+H_{12,\alpha})\psi_{2M,L}}=&\lra{\psi_{2M,L},(2C_{0}+2I_{M,L}+\alpha(S_{1}^{2}+S_{2}^{2}))\psi_{2M,L}}\\
\geq& 2 C_{0}-\frac{L(N/L)^{3\beta}(V^{+}(0)+V^{-}(0))}{2(M-1)}\notag.
\end{align}
Since $\beta<\frac{1}{3}$, there exists an $N_{0}>0$ such that
\begin{align}
\lra{\psi_{2M,L},(2C_{0}+H_{12,\alpha})\psi_{2M,L}}\geq 0,
\end{align}
for all $N=2M\geq N_{0}.$

\end{proof}
The proof also works for the case in which we put $\R^{2}$ in the $x$-direction and hence the main theorem works the same for $\R^{2}$ with $\beta<1/3$.
\subsection{Bootstrapping argument}\footnote{The finite dimensional quantum de Finetti theorem is only used in this section. In Section $\ref{section k=1}$, we have already reached $\beta<1/3$ without it. Hence, the main theorem works the same for $\R^{2}$ with $\beta<1/3$.} \label{section bootstapping}In \cite{lewin2017note}, Lewin, Nam, and Rougerie can improve the index $\be$ by a bootstraping
argument for 2D case. Here, such a method also works for $\be <1/2$ with a starting point $\eta_{0}\leq (22\be-4)/5$ in our setting. Let us define
\begin{align}
\label{def:many-body ground state per particle}&E_{N,L}:=\inf_{\n{\psi_{N,L}}_{L^{2}}=1}\lra{\psi_{N,L},\lrs{1+N^{-1}H_{N,L}-1/L^{2}}\psi_{N,L}},\\
\label{def:modified many-body ground state per particle}&E_{N,L,\varepsilon}:=\inf_{\n{\psi_{N,L}}_{L^{2}}=1}\lra{\psi_{N,L},\lrs{1+N^{-1}H_{N,L}-1/L^{2}-\varepsilon N^{-1}\sum_{j=1}^{N}S_{j}^{2}}\psi_{N,L}},
\end{align}
where $E_{N,L}$ denotes the many-body ground state energy per particle. From the definition $(\ref{def:modified many-body ground state per particle})$, estimate $(\ref{equ:stability of matter the main estimate})$ in Theorem $\ref{thm:Stability of matter when k=1}$, is equivalent to prove $E_{N,L,1-\alpha}\geq -C_{0}$ for $N\geq N_{0}$. Indeed, if $E_{N,L,1-\alpha}\geq -C_{0}$ for $N\geq N_{0}$, it means that
 $$C_{0}+1+N^{-1}H_{N,L}-1/L^{2}-(1-\alpha) S_{1}^{2}\geq 0,$$
 for $N\geq N_{0},$ which is the estimate $(\ref{equ:stability of matter the main estimate})$.

Thus, our goal is to bound $E_{N,L,\varepsilon}$ from below. We note that $E_{N,L,\varepsilon}=(1-\varepsilon)E_{N,L}^{\varepsilon}$, where $E_{N,L}^{\varepsilon}$ is the ground state energy with interaction function $V^{\varepsilon}=(1-\varepsilon)^{-1}V$. So we only need to deal with $E_{N,L}^{\varepsilon}$ or $E_{N,L}$. A main tool is the finite-dimensional quantum de Finetti theorem (Lemma \ref{lemma:Finite-dimensional quantum de Finetti Stability of matter when k=1}). Then we can give a lower bound on the Hamiltonian energy in Lemma $\ref{lemma:stability of matter when k=1:lower bound on the ground state energy}$, that is,
\begin{align}
\bblra{\psi_{N,L},\lrs{1+N^{-1}H_{N,L}-1/L^{2}}\psi_{N,L}}\geq C(V,N,L,\lambda_{x},\lambda_{z},\n{S_{1}\psi_{N,L}}_{L^{2}},\n{S_{1}S_{2}\psi_{N,L}}_{L^{2}}),
\end{align}
where $\lambda_{x}$ and $\lambda_{z}$ are cut-off parameters.

Subsequently in Lemma $\ref{lemma:stability of matter when k=1: moments estimate}$, we will control $\n{S_{1}\psi_{N,L}}_{L^{2}}$ and $\n{S_{1}S_{2}\psi_{N,L}}_{L^{2}}$ for the ground state $\psi_{N,L}$. More precisely,
\begin{align}
&Tr\lrs{S_{1}^{2}\gamma_{N,L}^{(1)}}\lesssim C_{up}\frac{1+|E_{N,L,\theta}|}{\theta},\\
&Tr\lrs{S_{1}^{2}S_{2}^{2}\gamma_{N,L}^{(2)}}\lesssim C_{up}^{2}\lrs{\frac{1+|E_{N,L,\theta}|}{\theta}}^{2},
\end{align}
where $C_{up}$ is an upper bound constant defined by $(\ref{eq:upper bound constant})$.

With Lemma $\ref{lemma:stability of matter when k=1:lower bound on the ground state energy}$ and $\ref{lemma:stability of matter when k=1: moments estimate}$, we arrive at a closed control relationship, namely,
$$|E_{N,L}|\leq C(V,N,L,\lambda_{x},\lambda_{z},|E_{N,L,\theta}|).$$
Thus we can use the bootstapping argument to bound $|E_{N,L,\varepsilon}|$ as long as there exists a starting point.

Now, we present the above procedure in detail. First, we take $\psi_{N,L}=\phi_{L}^{\otimes N}$ with $\n{\phi_{L}}_{L^{2}}=1$ and obtain the NLS energy functional
\begin{align}
\ce_{N,L}(\phi_{L})=&\lra{\phi_{L}^{\otimes N},\lrs{1+N^{-1}H_{N,L}-1/L^{2}}\phi_{L}^{\otimes N}} \\
=&\frac{1}{2}\lra{\phi_{L}^{\otimes 2},H_{12}\phi_{L}^{\otimes 2}}=\lra{S_{1}^{2}\phi_{L},\phi_{L}}+\frac{L}{2}\int_{\Om_{L}^{2}} V_{N,L}(r_{1}-r_{2})|\phi_{L}(r_{1})\phi_{L}(r_{2})|^{2}dr_{1}dr_{2}.\notag
\end{align}
Define
\begin{align} \label{def:ground state energy of the NLS energy functional}
e_{N,L}:=\inf_{\n{\phi_{L}}_{L^{2}}=1}\ce_{N,L}(\phi_{L}),
\end{align}
where $e_{N,L}$ stands for the ground state energy of the NLS energy functional. From the above definition, we know $e_{N,L}\geq E_{N,L}\geq E_{N,L,\varepsilon}.$
To bound $E_{N,L,\varepsilon}$ from below, it is necessary to bound $e_{N,L}$ from below. Here, we first give a sufficient condition to bound $e_{N,L}$ as follows.

 \begin{lemma}\label{lemma:stability of matter when k=1: Hartree functional estimate :lower bound}
 Assume $L(N/L)^{\beta}\to 1^{-}$ and  $\n{V}_{L_{z}^{\infty}L_{x}^{1}}\leq \frac{2}{C_{gn}^{4}}$, then $0\leq e_{N,L}\leq C_{up},$ where
 \begin{align}\label{eq:upper bound constant}
 C_{up}=1+\frac{C_{gn}^{4}\n{V}_{L_{z}^{\infty}L_{x}^{1}}}{2}.
 \end{align}
 \end{lemma}
 \begin{proof}

For the lower bound, it suffices to prove $\ce_{N,L}(\phi_{L})\geq 0$. From estimate $(\ref{equ:product estimate equation 2})$ in Lemma $\ref{lemma:Stability of matter when k=1 interaction estimate}$, we have
\begin{align}
\ce_{N,L}(\phi_{L})=&\lra{S_{1}^{2}\phi_{L},\phi_{L}}+\frac{L}{2}\int_{\Om_{L}^{2}} V_{N,L}(r_{1}-r_{2})|\phi_{L}(r_{1})\phi_{L}(r_{2})|^{2}dr_{1}dr_{2} \\
\geq &\lra{S_{1}^{2}\phi_{L},\phi_{L}}-\frac{C_{gn}^{4}\n{V}_{L_{z}^{\infty}L_{x}^{1}}}{2}\lra{S_{1}^{2}\phi_{L},\phi_{L}}\geq 0,\notag
\end{align}
as long as $\n{V}_{L_{z}^{\infty}L_{x}^{1}}\leq \frac{2}{C_{gn}^{4}}$.

For the upper bound, we use estimate $(\ref{equ:product estimate equation 2})$ again and obtain
\begin{align}
\ce_{N,L}(\phi_{L})=&\lra{S_{1}^{2}\phi_{L},\phi_{L}}+\frac{L}{2}\int_{\Om_{L}^{2}} V_{N,L}(r_{1}-r_{2})|\phi_{L}(r_{1})\phi_{L}(r_{2})|^{2}dr_{1}dr_{2}\\
\leq& \lra{S_{1}^{2}\phi_{L},\phi_{L}}+\frac{C_{gn}^{4}\n{V}_{L_{z}^{\infty}L_{x}^{1}}}{2}\lra{S_{1}^{2}\phi_{L},\phi_{L}}.\notag
\end{align}
When we take $\phi_{L}(x,z)=\cos_{L}(z)$, which is the $L^{2}$ normalized ground state wave function of $S_{1}^{2}$, then $\lra{S_{1}^{2}\phi_{L},\phi_{L}}=1$. Hence, we have
\begin{align}
e_{N,L}\leq 1+\frac{C_{gn}^{4}\n{V}_{L_{z}^{\infty}L_{x}^{1}}}{2}.
\end{align}
 \end{proof}

To establish the lower bound estimate for $E_{N,L}$, we need the finite-dimensional quantum de Finetti theorem. We define the Littlewood-Paley projectors (eigenspace projector) by
\begin{align}
&P_{x_{j},m}=\chi_{\lr{m}}(-\Delta_{x_{j}}),\quad m\geq 0, \label{eq:Littlewood-Paley projectors}\\
&P_{x_{j},\leq \lambda}=\chi_{[0,\lambda^{2}]}\lrs{-\Delta_{x_{j}}},\quad P_{z_{j},\leq \lambda}=\chi_{[1/L^{2},\lambda^{2}/L^{2}]}(-\pa_{z_{j}}^{2}),\\
&P_{x_{j},>\lambda}=I-P_{x_{j},\leq \lambda},\quad P_{z_{j},>\lambda}=I-P_{z_{j},\leq \lambda}.
\end{align}

From the definition, we notice that $P_{z_{j},\leq \lambda}$ and $P_{z_{j},>\lambda}$ are $L$ dependent. However, we omit it for simplicity.

\begin{lemma}[Finite-dimensional quantum de Finetti\cite{lewin2016mean}]\label{lemma:Finite-dimensional quantum de Finetti Stability of matter when k=1}
Assume $\lr{\gamma_{N,L}^{(k)}}_{k=1}^{N}$ is the marginal density generated by an N-body wave function $\psi_{N,L}\in L_{s}^{2}(\Om_{L}^{N})$ and $P_{L}$ be a finite-rank orthogonal projector with
$$dim(P_{L}(L^{2}(\Om_{L})))=d<\infty,$$
where $d$ can be independent of $L$. Then, there is a positive Borel measure $d\mu_{N,L}$ supported on the unit sphere $\cs P_{L}(L^{2}(\Om_{L}))$ such that
\begin{align}
Tr\bbabs{P_{L}^{\otimes 2}\gamma_{N,L}^{(2)}P_{L}^{\otimes 2}-\int _{\cs(P_{L}(L^{2}(\Om_{L})))}|\phi_{L}^{\otimes 2}\rangle \langle \phi_{L}^{\otimes 2}|d\mu_{N,L}(\phi_{L})}\leq \frac{8d}{N}.
\end{align}
\end{lemma}

Moreover, we will need operator inequalities for two-body interaction as follows.
\begin{lemma}\label{lemma:Stability of matter when k=1 basic operator estimate}
Assume $L(N/L)^{\beta}\to 1^{-}$. For $\delta\in (0,1)$, the multiplication operator $V_{N,L}(r_{1}-r_{2})$ on $L^{2}(\Om_{L}^{\otimes 2})$ satisfies
\begin{align}
&L|V_{N,L}(r_{1}-r_{2})|\leq C_{\delta}(N/L)^{\delta}\n{V}_{L_{z}^{\infty}L_{x}^{1+\delta}}(1-\Delta_{x_{1}}),\label{equ:Stability of matter when k=1 basic operator estimate: one body estimate}\\
&L|V_{N,L}(r_{1}-r_{2})|\leq C_{\delta}\n{V}_{L_{z}^{\infty}L_{x}^{1}}(1-\Delta_{x_{1}})^{1/2+\delta}(1-\Delta_{x_{2}})^{1/2+\delta},\label{equ:Stability of matter when k=1 basic operator estimate}\\
&\n{\lra{\nabla_{x_{1}}}^{-1}\lra{\nabla_{x_{2}}}^{-1}V(r_{1}-r_{2})\lra{\nabla_{x_{1}}}^{-1}\lra{\nabla_{x_{2}}}^{-1}}_{op}\leq C\n{V}_{L_{z}^{\infty}L_{x}^{1}},\label{equ:operator norm estimate}\\
&S_{1}^{2}LV_{N,L}(r_{1}-r_{2})+LV_{N,L}(r_{1}-r_{2})S_{1}^{2}\geq -C_{\delta}(V)(N/L)^{\beta+\delta}S_{1}^{2}S_{2}^{2}\label{equ:Stability of matter when k=1 basic operator estimate: commutator estimate},
\end{align}
where $C_{\delta}(V)$ is dependent on $V$.

\end{lemma}
\begin{proof}
For smoothness of presentation, we put the proof in the Appendix.

\end{proof}

Now, along with Lemma $\ref{lemma:Finite-dimensional quantum de Finetti Stability of matter when k=1}$, we can establish the following lower bound estimate.
\begin{lemma}\label{lemma:stability of matter when k=1:lower bound on the ground state energy}
Assume $L(N/L)^{\beta}\to 1^{-}$ and $\n{V}_{L_{z}^{\infty}L_{x}^{1}}\leq \frac{2}{C_{gn}^{4}}$. Then for every $\delta\in (0,1/2),$ there exists a constant $C_{\delta}>0$ such that for all $N\geq 2$, $\lambda_{z}\geq 1$, $\lambda_{x}\geq 0$ and for all wave functions $\psi_{N,L}$
\begin{align*}
&\bblra{\psi_{N,L},\lrs{1+N^{-1}H_{N,L}-1/L^{2}}\psi_{N,L}}\\
\gtrsim
&-C_{\delta}\n{V}_{L_{z}^{\infty}L_{x}^{1}}(1+\lambda_{x}^{2})^{1+2\delta}\frac{d}{N}\\
&-\frac{C_{\delta}\n{V}_{L_{z}^{\infty}L_{x}^{1}}}{\min\lr{1+\lambda_{x}^{2},1+(\lambda_{z}^{2}-1)/L^{2}}^{1/4-\delta/2}}\lrs{Tr\lrs{S_{1}^{2}\gamma_{N,L}^{(1)}}}^{1/4-\delta/2}
\lrs{Tr\lrs{S_{1}^{2}S_{2}^{2}\gamma_{N,L}^{(2)}}}^{1/2+\delta},
\end{align*}
where $d=dim(P_{x,\leq \lambda_{x}}P_{z,\leq \lambda_{z}}L^{2}(\Om_{L}))$.
\end{lemma}
\begin{proof}
For simplicity, we adopt the notation
\begin{align*}
&P:=P_{x,\leq \lambda_{x}}P_{z,\leq \lambda_{z}},\quad Q=1-P,\\
&P^{j}:=P_{x_{j},\leq \lambda_{x}}P_{z_{j},\leq \lambda_{z}},\quad Q^{j}:=1-P^{j},\\
&P^{(k)}:=\prod_{j=1}^{k}P_{x_{j},\leq \lambda_{x}},\quad P_{z_{j},\leq \lambda_{z}}=\prod_{j=1}^{k}P^{j}.
\end{align*}
Recall $H_{12}=S_{1}^{2}+S_{2}^{2}+LV_{N,L}(r_{1}-r_{2})$ and $S_{1}^{2}=1-\Delta_{r_{1}}-1/L^{2}$. Using Lemma $\ref{lemma:Finite-dimensional quantum de Finetti Stability of matter when k=1}$ , we write
\begin{align}
\bblra{\psi_{N,L},\lrs{1+N^{-1}H_{N,L}-1/L^{2}}\psi_{N,L}}
=&\frac{1}{2}\lra{\psi_{N,L},H_{12}\psi_{N,L}}\notag\\
\geq& \lra{\psi_{N,L},(1-\Delta_{x_{1}}+\frac{1}{2}LV_{N,L}(r_{1}-r_{2}))\psi_{N,L}}\notag\\
=&I+II+III,\notag
\end{align}
where
\begin{align}
I=&\int \lra{\phi_{L}^{\otimes 2},(1-\Delta_{x_{1}}+\frac{1}{2}LV_{N,L}(r_{1}-r_{2}))\phi_{L}^{\otimes 2}}d\mu_{N,L}(\phi_{L})\\
II=&Tr \lrs{(1-\Delta_{x_{1}}+\frac{1}{2}LV_{N,L}(r_{1}-r_{2}))\lrc{\gamma_{N,L}^{(2)}-P^{(2)}\ga_{N,L}^{(2)}P^{(2)}}}\\
III=&Tr\lrs{(1-\Delta_{x_{1}}+\frac{1}{2}LV_{N,L}(r_{1}-r_{2}))\lrc{P^{(2)}\gamma_{N,L}^{(2)}P^{(2)}-\int_{\mathcal{S}(P(L^{2}(\Om_{L})))}|\phi_{L}^{\otimes 2}\rangle \langle \phi_{L}^{\otimes 2}|d\mu_{N,L}(\phi_{L})}}
\end{align}
We only need to bound these terms from below.

We first handle $I$. By estimate $(\ref{eq:Stability of matter when k=1 interaction estimate conclusion})$ with $\psi_{M,L}=\phi_{L}^{\otimes 2}$, we have
\begin{align*}\lra{\phi_{L}^{\otimes 2},(1-\Delta_{x_{1}}+\frac{1}{2}LV_{N,L}(r_{1}-r_{2}))\phi_{L}^{\otimes 2}}\geq 0.
\end{align*}
Noting that $\mu_{N,L}$ is a positive measure, we obtain $I\geq 0.$ Hence we can discard it.

We now deal with $II$ and rewrite
\begin{align}
II=&Tr \lrs{(1-\Delta_{x_{1}})\lrc{\gamma_{N,L}^{(2)}-P^{(2)}\gamma_{N,L}^{(2)}P^{(2)}}}
+\frac{1}{2}Tr \lrs{LV_{N,L}(r_{1}-r_{2})\lrc{\gamma_{N,L}^{(2)}-P^{(2)}\gamma_{N,L}^{(2)}P^{(2)}}}\\
=&II_{K}+II_{V}.\notag
\end{align}

Since $(1-\Delta_{x_{1}})\geq (1-\Delta_{x_{1}})P^{(2)}$, with Lemma $\ref{lemma:Appendix kernel and trace}$, we obtain
\begin{align}
II_{K}=&Tr\lrs{(1-\Delta_{x_{1}})\lrc{\gamma_{N,L}^{(2)}-P^{(2)}\gamma_{N,L}^{(2)}P^{(2)}}}\\
=&\lra{\lrc{(1-\Delta_{x_{1}})-(1-\Delta_{x_{1}})P^{(2)}}\psi_{N,L},\psi_{N,L}}\geq0.\notag
\end{align}

For $II_{V}$, we expand
\begin{align}
&2\lrs{LV_{N,L}(r_{1}-r_{2})-P^{(2)}LV_{N,L}(r_{1}-r_{2})P^{(2)}}\\
=&LV_{N,L}(r_{1}-r_{2})(1-P^{(2)})+(1-P^{(2)})LV_{N,L}(r_{1}-r_{2})\notag\\
&+(1-P^{(2)})LV_{N,L}(r_{1}-r_{2})P^{(2)}+ P^{(2)}LV_{N,L}(r_{1}-r_{2})(1-P^{(2)}).\notag
\end{align}
By Lemma $\ref{lemma:Appendix kernel and trace}$ and Cauchy-Schwarz inequality,
\begin{align}\label{equ:lower bound on the ground state energy: error term II interaction 1}
&\frac{1}{2}\bbabs{Tr \lrs{\lrs{LV_{N,L}(r_{1}-r_{2})(1-P^{(2)})+(1-P^{(2)})LV_{N,L}(r_{1}-r_{2})}\gamma_{N,L}^{(2)}}}\\
=&\frac{1}{2}\bbabs{\lra{LV_{N,L}(r_{1}-r_{2})(1-P^{(2)})\psi_{N,L},\psi_{N,L}}+\lra{LV_{N,L}(r_{1}-r_{2})\psi_{N,L},(1-P^{(2)})\psi_{N,L}}}\notag\\
\leq& L\n{|V_{N,L}(r_{1}-r_{2})|^{1/2}\psi_{N,L}}_{L^{2}}\n{|V_{N,L}(r_{1}-r_{2})|^{1/2}(1-P^{(2)})\psi_{N,L}}_{L^{2}}.\notag
\end{align}
Computing in the same way, we have
\begin{align}\label{equ:lower bound on the ground state energy: error term II interaction 2}
&\frac{1}{2}\bbabs{Tr \lrs{(1-P^{(2)})LV_{N,L}(r_{1}-r_{2})P^{(2)}+ P^{(2)}LV_{N,L}(r_{1}-r_{2})(1-P^{(2)})}\gamma_{N,L}^{(2)}}\\
=&\frac{1}{2}\bbabs{\lra{LV_{N,L}(r_{1}-r_{2})P^{(2)}\psi_{N,L},(1-P^{(2)})\psi_{N,L}}+\lra{LV_{N,L}(r_{1}-r_{2})(1-P^{(2)})\psi_{N,L},P^{(2)}\psi_{N,L}}}\notag\\
\leq& L\n{|V_{N,L}(r_{1}-r_{2})|^{1/2}P^{(2)}\psi_{N,L}}_{L^{2}}\n{|V_{N,L}(r_{1}-r_{2})|^{1/2}(1-P^{(2)})\psi_{N,L}}_{L^{2}}.\notag
\end{align}
Combining estimates $(\ref{equ:lower bound on the ground state energy: error term II interaction 1})$ and $(\ref{equ:lower bound on the ground state energy: error term II interaction 2})$, we obtain
\begin{align} \label{equ:lower bound on the ground state energy: error term II interaction estimate}
|II_{V}|\leq& L \lrs{\n{|V_{N,L}(r_{1}-r_{2})|^{1/2}\psi_{N,L}}_{L^{2}}+
\n{|V_{N,L}(r_{1}-r_{2})|^{1/2}P^{(2)}\psi_{N,L}}_{L^{2}}}\\
&\times\n{|V_{N,L}(r_{1}-r_{2})|^{1/2}(1-P^{(2)})\psi_{N,L}}_{L^{2}}.\notag
\end{align}

Next, we need to bound the right side terms. From estimate $(\ref{equ:Stability of matter when k=1 basic operator estimate})$, we obtain
\begin{align*}
|LV_{N,L}(r_{1}-r_{2})|\leq &C_{\delta}\n{V}_{L_{z}^{\infty}L_{x}^{1}}\lrs{(1-\Delta_{x_{1}})(1-\Delta_{x_{2}})}^{1/2+\delta}\\
\leq& C_{\delta}\n{V}_{L_{z}^{\infty}L_{x}^{1}}\lrs{\eta^{-1}(1-\Delta_{x_{1}})(1-\Delta_{x_{2}})+\eta^{\frac{1+2\delta}{1-2\delta}}},
\end{align*}
where we have used the interpolation inequality for fractional powers in Lemma $\ref{lemma:appendix interpolation inequality}$ in the last line. By optimizing over $\eta>0,$ we have
\begin{align} \label{equ:lower bound on the ground state energy: error term II interaction P 1}
L\n{|V_{N,L}(r_{1}-r_{2})|^{1/2}\psi_{N,L}}_{L^{2}}^{2}=&\lra{L|V_{N,L}(r_{1}-r_{2})|\psi_{N,L},\psi_{N,L}}\\
\leq &C_{\delta}\n{V}_{L_{z}^{\infty}L_{x}^{1}}\lrs{\lra{(1-\Delta_{x_{1}})(1-\Delta_{x_{2}})\psi_{N,L},\psi_{N,L}}}^{1/2+\delta}.\notag
\end{align}
Similarly,
\begin{align} \label{equ:lower bound on the ground state energy: error term II interaction P 2}
L\n{|V_{N,L}(r_{1}-r_{2})|^{1/2}P^{(2)}\psi_{N,L}}_{L^{2}}^{2}\leq
C_{\delta}\n{V}_{L_{z}^{\infty}L_{x}^{1}}\lrs{\lra{(1-\Delta_{x_{1}})(1-\Delta_{x_{2}})\psi_{N,L},\psi_{N,L}}}^{1/2+\delta}.
\end{align}

Using estimate $(\ref{equ:Stability of matter when k=1 basic operator estimate})$ again, we get
\begin{align}\label{equ:lower bound on the ground state energy: error term II interaction 1-P 1}
&(1-P^{(2)})|LV_{N,L}(r_{1}-r_{2})|(1-P^{(2)})\\
\leq& C_{\delta}\n{V}_{L_{z}^{\infty}L_{x}^{1}}(1-P^{(2)})\lrs{(1-\Delta_{x_{1}})(1-\Delta_{x_{2}})}^{1/2+\delta}(1-P^{(2)})\notag\\
=& C_{\delta}\n{V}_{L_{z}^{\infty}L_{x}^{1}}(1-P^{(2)})\lrs{(1-\Delta_{x_{1}})(1-\Delta_{x_{2}})}^{1/2+\delta},\notag
\end{align}
where we used $(1-\Delta_{x_{j}})P^{(2)}=P^{(2)}(1-\Delta_{x_{j}})$ in the last equality.

Since $1-P^{(2)}\leq Q_{1}+Q_{2}$, we use Property $3$ in Lemma $\ref{lemma:appendix standard operator inequalities}$ to get
\begin{align}\label{equ:lower bound on the ground state energy: error term II interaction 1-P 2}
(1-P^{(2)})\lrs{(1-\Delta_{x_{1}})(1-\Delta_{x_{2}})}^{1/2+\delta}\leq (Q_{1}+Q_{2})\lrs{(1-\Delta_{x_{1}})(1-\Delta_{x_{2}})}^{1/2+\delta}.
\end{align}
By $\min\lr{1+\lambda_{x}^{2},1+(\lambda_{z}^{2}-1)/L^{2}}  Q\leq S^{2}Q,$ we obtain
\begin{align} \label{equ:lower bound on the ground state energy: error term II interaction 1-P 3}
&(Q_{1}+Q_{2})\lrs{(1-\Delta_{x_{1}})(1-\Delta_{x_{2}})}^{1/2+\delta}\\
\leq &(Q_{1}+Q_{2})(S_{1}^{2}S_{2}^{2})^{1/2+\delta} \notag \\
\leq &\frac{1}{\min\lr{1+\lambda_{x}^{2},1+(\lambda_{z}^{2}-1)/L^{2}}^{1/2-\delta}} \lrs{S_{1}^{2}\lrs{S_{2}^{2}}^{1/2+\delta}+\lrs{S_{1}^{2}}^{1/2+\delta}S_{2}^{2}}\notag\\
\leq & \frac{1}{\min\lr{1+\lambda_{x}^{2},1+(\lambda_{z}^{2}-1)/L^{2}}^{1/2-\delta}}\lrc{\eta^{-1}S_{1}^{2}S_{2}^{2}+\eta^{\frac{1+2\delta}
{1-2\delta}}\lrs{S_{1}^{2}+S_{2}^{2}}}.\notag
\end{align}
where we have used the interpolation inequality for fractional powers in Lemma $\ref{lemma:appendix interpolation inequality}$ in the last line.

Putting estimates $(\ref{equ:lower bound on the ground state energy: error term II interaction 1-P 1})$ ($\ref{equ:lower bound on the ground state energy: error term II interaction 1-P 2}$) and $(\ref{equ:lower bound on the ground state energy: error term II interaction 1-P 3})$ together, we have
\begin{align}
&(1-P^{(2)})|LV_{N,L}(r_{1}-r_{2})|(1-P^{(2)})\\
\leq & \frac{C_{\delta}\n{V}_{L_{z}^{\infty}L_{x}^{1}}}{\min\lr{1+\lambda_{x}^{2},1+(\lambda_{z}^{2}-1)/L^{2}}^{1/2-\delta}}\lrc{\eta^{-1}S_{1}^{2}S_{2}^{2}+\eta^{\frac{1+2\delta}
{1-2\delta}}\lrs{S_{1}^{2}+S_{2}^{2}}}.\notag
\end{align}
By optimizing over $\eta$, we deduce that
\begin{align}\label{equ:lower bound on the ground state energy: error term II interaction 1-P estimate}
&L\n{|V_{N,L}(r_{1}-r_{2})|^{1/2}(1-P^{(2)})\psi_{N,L}}_{L^{2}}^{2} \\
=&\lra{L|V_{N,L}(r_{1}-r_{2})|(1-P^{(2)})\psi_{N,L},(1-P^{(2)})\psi_{N,L}}\notag \\
\leq&\frac{C_{\delta}\n{V}_{L_{z}^{\infty}L_{x}^{1}}}{\min\lr{1+\lambda_{x}^{2},1+(\lambda_{z}^{2}-1)/L^{2}}^{1/2-\delta}}\lrs{Tr\lrs{S_{1}^{2}\gamma_{N,L}^{(1)}}}^{1/2-\delta}
\lrs{Tr\lrs{S_{1}^{2}S_{2}^{2}\gamma_{N,L}^{(2)}}}^{1/2+\delta}.\notag
\end{align}

Combining estimates ($\ref{equ:lower bound on the ground state energy: error term II interaction estimate}$) ($\ref{equ:lower bound on the ground state energy: error term II interaction P 1}$) ($\ref{equ:lower bound on the ground state energy: error term II interaction P 2}$) and ($\ref{equ:lower bound on the ground state energy: error term II interaction 1-P estimate}$), we get
\begin{align}\label{equ:lower bound on the ground state energy: error term II}
II\geq& -\frac{C_{\delta}\n{V}_{L_{z}^{\infty}L_{x}^{1}}}{\min\lr{1+\lambda_{x}^{2},1+(\lambda_{z}^{2}-1)/L^{2}}^{1/4-\delta/2}}\lrs{Tr\lrs{S_{1}^{2}\gamma_{N,L}^{(1)}}}^{1/4-\delta/2}
\lrs{Tr\lrs{S_{1}^{2}S_{2}^{2}\gamma_{N,L}^{(2)}}}^{1/2+\delta}.
\end{align}

For $III$, we rewrite
\begin{align}
III=&Tr(1-\Delta_{x_{1}})\lrc{P^{(2)}\gamma_{N,L}^{(2)}P^{(2)}-\int_{\cs P(L^{2}(\Om_{L}))}|\phi_{L}^{\otimes 2}\rangle \langle \phi_{L}^{\otimes 2}|d\mu_{N,L}(\phi_{L})}\\
&+\frac{1}{2}Tr\lrs{LV_{N,L}(r_{1}-r_{2})\lrc{P^{(2)}\gamma_{N,L}^{(2)}P^{(2)}-\int|\phi_{L}^{\otimes 2}\rangle \langle \phi_{L}^{\otimes 2}|d\mu_{N,L}(\phi_{L})}}\notag\\
=&III_{K}+III_{V}.\notag
\end{align}

For the first term $III_{K}$, we can use the inequality $|Tr AB|\leq \n{A}_{op}Tr |B|$ to get
\begin{align}\label{eq:term III kinetic estimate stability of matter when k=1:lower bound on the ground state energy}
|III_{K}| \leq \n{(1-\Delta_{x})P}_{op}Tr\bbabs{P^{(2)}\gamma_{N,L}^{(2)}P^{(2)}-\int_{\cs P(L^{2}(\Om_{L}))}|\phi_{L}^{\otimes 2}\rangle \langle \phi_{L}^{\otimes 2}|d\mu_{N,L}(\phi_{L})}.
\end{align}
 Using $(1-\Delta_{x})P\leq (1+\la_{x}^{2})P$ and Lemma $\ref{lemma:Finite-dimensional quantum de Finetti Stability of matter when k=1}$, we have
\begin{align} \label{equ:lower bound on the ground state energy: error term IIIK}
|III_{K}|\lesssim &(1+\lambda_{x}^{2})\frac{d}{N}.
\end{align}

For the second term $III_{V}$, we use the inequality $|Tr AB|\leq \n{A}_{op}Tr |B|$ again to get
\begin{align}
|III_{V}|\leq \n{P^{(2)}|LV_{N,L}(r_{1}-r_{2})|P^{(2)}}_{op}Tr\bbabs{P^{(2)}\gamma_{N,L}^{(2)}P^{(2)}-\int_{\cs P(L^{2}(\Om_{L}))}|\phi_{L}^{\otimes 2}\rangle \langle \phi_{L}^{\otimes 2}|d\mu_{N,L}(\phi_{L})}.
\end{align}
By estimate $(\ref{equ:Stability of matter when k=1 basic operator estimate})$ and $P^{j}(1-\Delta_{x_{j}})\leq (1+\lambda_{x}^{2})P^{j}$, we have
\begin{align}
P^{(2)}|LV_{N,L}(r_{1}-r_{2})|P^{(2)}\leq& C_{\delta}\n{V}_{L_{z}^{\infty}L_{x}^{1}}\lrs{\lrs{P^{1}(1-\Delta_{x_{1}})}\otimes \lrs{P^{2}(1-\Delta_{x_{2}})}}^{1/2+\delta}\\
\leq &C_{\delta}\n{V}_{L_{z}^{\infty}L_{x}^{1}} (1+\lambda_{x}^{2})^{1+2\delta}P^{(2)}.\notag
\end{align}
Noting that $[P^{(2)}|LV_{N,L}(r_{1}-r_{2})|P^{(2)},P^{(2)}]=0,$ by Property $2$ in Lemma $\ref{lemma:appendix standard operator inequalities}$, we deduce that
$$\n{P^{(2)}|LV_{N,L}(r_{1}-r_{2})|P^{(2)}}_{op}\leq C_{\delta}\n{V}_{L_{z}^{\infty}L_{x}^{1}} (1+\lambda_{x}^{2})^{1+2\delta}.$$

Therefore, using Lemma $\ref{lemma:Finite-dimensional quantum de Finetti Stability of matter when k=1}$ again, we get
\begin{align}\label{equ:lower bound on the ground state energy: error term IIIV}
&\bbabs{Tr\lrs{LV_{N,L}(r_{1}-r_{2})\lrc{P^{(2)}\gamma_{N,L}^{(2)}P^{(2)}-\int_{\cs P(L^{2}(\Om_{L}))}|\phi_{L}^{\otimes 2}\rangle \langle \phi_{L}^{\otimes 2}|d\mu_{N,L}(\phi_{L})}}}\\
\lesssim &C_{\delta}\n{V}_{L_{z}^{\infty}L_{x}^{1}}(1+\lambda_{x}^{2})^{1+2\delta}\frac{d}{N}.\notag
\end{align}

Thus, for all $\delta>0,$ combining estimates $(\ref{equ:lower bound on the ground state energy: error term IIIK})$ and $(\ref{equ:lower bound on the ground state energy: error term IIIV})$, we obtain
\begin{align}\label{equ:lower bound on the ground state energy: error term III}
III\gtrsim -C_{\delta}\n{V}_{L_{z}^{\infty}L_{x}^{1}}(1+\lambda_{x}^{2})^{1+2\delta}\frac{d}{N}.
\end{align}

Combining estimates $(\ref{equ:lower bound on the ground state energy: error term II})$ and $(\ref{equ:lower bound on the ground state energy: error term III})$, we get the desired lower bound, that is,
\begin{align*}
&\bblra{\psi_{N,L},\lrs{1+N^{-1}H_{N,L}-1/L^{2}}\psi_{N,L}}\\
\gtrsim
&-C_{\delta}\n{V}_{L_{z}^{\infty}L_{x}^{1}}(1+\lambda_{x}^{2})^{1+2\delta}\frac{d}{N}\\
&-\frac{C_{\delta}\n{V}_{L_{z}^{\infty}L_{x}^{1}}}{\min\lr{1+\lambda_{x}^{2},1+(\lambda_{z}^{2}-1)/L^{2}}^{1/4-\delta/2}}\lrs{Tr\lrs{S_{1}^{2}\gamma_{N,L}^{(1)}}}^{1/4-\delta/2}
\lrs{Tr\lrs{S_{1}^{2}S_{2}^{2}\gamma_{N,L}^{(2)}}}^{1/2+\delta}.
\end{align*}
\end{proof}

From Lemma $\ref{lemma:stability of matter when k=1:lower bound on the ground state energy}$, to get the lower bound estimate for $E_{N,L}$, we are left to control $Tr(S_{1}^{2}\gamma_{N,L}^{(1)})$ and $Tr(S_{1}^{2}S_{2}^{2}\gamma_{N,L}^{(2)})$ for the ground state $\psi_{N,L}.$

\begin{lemma}\label{lemma:stability of matter when k=1: moments estimate}
Assume $L(N/L)^{\beta}\to 1^{-}$, $\beta<\frac{1}{2}$. Let $\psi_{N,L}$ be a ground state of $1+H_{N,L}-1/L^{2}$. For all $\theta \in (0,1)$, we have
\begin{align}
&Tr\lrs{S_{1}^{2}\gamma_{N,L}^{(1)}}\lesssim C_{up}\frac{1+|E_{N,L,\theta}|}{\theta}, \label{lemma:stability of matter when k=1: moments estimate first order}\\
&Tr\lrs{S_{1}^{2}S_{2}^{2}\gamma_{N,L}^{(2)}}\lesssim C_{up}^{2}\lrs{\frac{1+|E_{N,L,\theta}|}{\theta}}^{2}.\label{lemma:stability of matter when k=1: moments estimate second order}
\end{align}
\end{lemma}
\begin{proof}
For estimate $(\ref{lemma:stability of matter when k=1: moments estimate first order})$, from the definition of $E_{N,L,\theta}$, we have
\begin{align}\label{equ:Stability of matter when k=1 trace estimate,one body bound}
1+N^{-1}H_{N,L}-1/L^{2}-\theta N^{-1}\sum_{j=1}^{N}S_{j}^{2}\geq E_{N,L,\theta}.
\end{align}
Let $\psi_{N,L}$ be a ground state of $1+N^{-1}H_{N,L}-1/L^{2}$. By taking the expectation against $\psi_{N,L}$ on both sides of $(\ref{equ:Stability of matter when k=1 trace estimate,one body bound})$, we obtain
\begin{align*}
Tr\lrs{S_{1}^{2}\gamma_{N,L}^{(1)}}\leq& \frac{1}{\theta}\lrs{\bblra{\psi_{N,L},\lrs{1+N^{-1}H_{N,L}-1/L^{2}}\psi_{N,L}}- E_{N,L,\theta}}
=\frac{E_{N,L}-E_{N,L,\theta}}{\theta}.
\end{align*}
From Lemma $\ref{lemma:stability of matter when k=1: Hartree functional estimate :lower bound}$, we see $C_{up}\geq e_{N,L}\geq E_{N,L}\geq E_{N,L,\theta}$. Therefore, $|E_{N,L}|\leq C_{up}+|E_{N,L,\theta}|$ and
\begin{align} \label{equ:stability of matter when k=1: moments estimate one order}
Tr\lrs{S_{1}^{2}\gamma_{N,L}^{(1)}}\leq \frac{2C_{up}(1+|E_{N,L,\theta}|)}{\theta}.
\end{align}

For estimate $(\ref{lemma:stability of matter when k=1: moments estimate second order})$, we notice that
\begin{align}\label{eq:stability of matter when k=1: moments estimate second order control inequality}
Tr \lrs{S_{1}^{2}S_{2}^{2}\gamma_{N,L}^{(2)}}\leq \frac{2}{N^{2}}\bblra{\psi_{N,L},\lrs{\sum_{j=1}^{N}S_{j}^{2}}^{2}\psi_{N,L}}.
\end{align}
Thus, it needs only to control the right term of $(\ref{eq:stability of matter when k=1: moments estimate second order control inequality})$. We rewrite
\begin{align} \label{eq:stability of matter when k=1: moments estimate second order rewrite}
\frac{2}{N^{2}}\lrs{\sum_{j=1}^{N}S_{j}^{2}}^{2}
=&\lrs{1+N^{-1}H_{N,L}-1/L^{2}}\frac{1}{N}\sum_{j=1}^{N}S_{j}^{2}+\frac{1}{N}\sum_{j=1}^{N}S_{j}^{2}\lrs{1+N^{-1}H_{N,L}-1/L^{2}}\\
&-\frac{1}{N^{2}(N-1)}\sum_{i=1}^{N}\sum_{j<k}\lrs{S_{i}^{2}LV_{N,L}(r_{j}-r_{k})+LV_{N,L}(r_{j}-r_{k})S_{i}^{2}}.\notag
\end{align}
We need to control the right side terms of $(\ref{eq:stability of matter when k=1: moments estimate second order rewrite})$.

For the first and second terms, by using the ground state $\psi_{N,L}$ and estimate $(\ref{equ:stability of matter when k=1: moments estimate one order})$, we have
\begin{align} \label{eq:Stability of matter when k=1 moment estimate first and second terms}
&\bblra{\psi_{N,L},\lrs{\lrs{1+N^{-1}H_{N,L}-1/L^{2}}\frac{1}{N}\sum_{j=1}^{N}S_{j}^{2}+\frac{1}{N}\sum_{j=1}^{N}S_{j}^{2}\lrs{1+N^{-1}H_{N,L}-1/L^{2}}}\psi_{N,L}}\\
=&\frac{2E_{N,L}}{N}\bblra{\psi_{N,L},\sum_{j=1}^{N}S_{j}^{2}\psi_{N,L}}\leq \frac{4C_{up}^{2}(1+|E_{N,L,\theta}|)^{2}}{\theta}.\notag
\end{align}

For the third term, we decompose it into two cases. On the one hand, we consider the case $j\neq i$ and $k\neq i.$ By estimate $(\ref{equ:Stability of matter when k=1 basic operator estimate: one body estimate})$, we have
\begin{align}
&\frac{1}{N(N-1)}\sum_{i\neq j<k\neq i}LV_{N,L}(r_{j}-r_{k})\\
=&1+N^{-1}H_{N,L}-1/L^{2}-\theta N^{-1}\sum_{j=1}^{N}S_{j}^{2}-(1-\theta)N^{-1}\sum_{j=1}^{N}S_{j}^{2}
-\frac{1}{N(N-1)}\sum_{j\neq i}LV_{N,L}(r_{i}-r_{j})\notag\\
\geq &E_{N,L,\theta}-\lrs{\frac{1-\theta}{N}+\frac{C_{\delta}(N/L)^{\delta}\n{V}_{L_{z}^{\infty}L_{x}^{1+\delta}}}{N^{2}}}\sum_{j=1}^{N}S_{j}^{2}.\notag
\end{align}
Since $[V_{N,L}(r_{j}-r_{k}),S_{i}^{2}]=0,$ by summing over $i$, we obtain
\begin{align}\label{equ:stability of matter when k=1: moments estimate two order main term}
&\frac{1}{N^{2}(N-1)}\sum_{i=1}^{N}\sum_{i\neq j<k\neq i}\lrs{S_{i}^{2}LV_{N,L}(r_{j}-r_{k})+LV_{N,L}(r_{j}-r_{k})S_{i}^{2}}\\
\geq & \frac{2}{N}\sum_{i=1}^{N}S_{i}^{2}\lrs{E_{N,L,\theta}-\lrs{\frac{1-\theta}{N}+\frac{C_{\delta}(N/L)^{\delta}
\n{V}_{L_{z}^{\infty}L_{x}^{1+\delta}}}{N^{2}}}
\sum_{j=1}^{N}S_{j}^{2}}\notag\\
= &\frac{2E_{N,L,\theta}}{N}\sum_{j=1}^{N}S_{j}^{2}-\frac{2}{N}\lrs{\frac{1-\theta}{N}+
\frac{C_{\delta}(N/L)^{\delta}\n{V}_{L_{z}^{\infty}L_{x}^{1+\delta}}}{N^{2}}}\lrs{\sum_{j=1}^{N}S_{j}^{2}}^{2}.\notag
\end{align}

On the other hand, when $j=i$ or $k=i$, by estimate $(\ref{equ:Stability of matter when k=1 basic operator estimate: commutator estimate})$, we have
\begin{align*}
S_{j}^{2}LV_{N,L}(r_{j}-r_{k})+LV_{N,L}(r_{j}-r_{k})S_{j}^{2}\geq -C_{\delta}(V)(N/L)^{\beta+\delta}S_{j}^{2}S_{k}^{2}.
\end{align*}
Therefore,
\begin{align} \label{equ:stability of matter when k=1: moments estimate two order cross term}
&\frac{1}{N^{2}(N-1)}\sum_{j\neq k}\lrs{S_{j}^{2}LV_{N,L}(r_{j}-r_{k})+LV_{N,L}(r_{j}-r_{k})S_{j}^{2}}\\
\geq & -C_{\delta}(V)\frac{(N/L)^{\beta+\delta}}{N^{3}}\lrs{\sum_{j=1}^{N}S_{j}^{2}}^{2}.\notag
\end{align}
Putting estimates ($\ref{equ:stability of matter when k=1: moments estimate two order main term}$) and ($\ref{equ:stability of matter when k=1: moments estimate two order cross term}$) together, we obtain
\begin{align} \label{eq:Stability of matter when k=1 moment estimate third term}
&\frac{1}{N^{2}(N-1)}\sum_{i=1}^{N}\sum_{j<k}\lrs{S_{i}^{2}LV_{N,L}(r_{j}-r_{k})
+LV_{N,L}(r_{j}-r_{k})S_{i}^{2}}\\
\geq &-\frac{2}{N^{2}}\lrs{1-\theta+\frac{C_{\delta}(V)(N/L)^{\beta+\delta}}{N}}\lrs{\sum_{j=1}^{N}S_{j}^{2}}^{2}
+\frac{2E_{N,L,\theta}}{N}\sum_{j=1}^{N}S_{j}^{2}.\notag
\end{align}

Taking the expectation against the ground state $\psi_{N,L}$ on both sides of $(\ref{eq:Stability of matter when k=1 moment estimate third term})$, we obtain
\begin{align} \label{eq:Stability of matter when k=1 moment estimate third term for ground state}
&\frac{1}{N^{2}(N-1)}\sum_{i=1}^{N}\sum_{j<k}\lra{\psi_{N,L},\lrs{S_{i}^{2}LV_{N,L}(r_{j}-r_{k})
+LV_{N,L}(r_{j}-r_{k})S_{i}^{2}}\psi_{N,L}}\\
\geq&
-\frac{2}{N^{2}}\lrs{1-\theta+\frac{C_{\delta}(V)(N/L)^{\beta+\delta}}{N}}\bblra{\psi_{N,L},\lrs{\sum_{j=1}^{N}S_{j}^{2}}^{2}\psi_{N,L}}
+\frac{2E_{N,L,\theta}}{N}\sum_{j=1}^{N}\lra{\psi_{N,L},S_{j}^{2}\psi_{N,L}}\notag\\
\geq&-\frac{2}{N^{2}}\lrs{1-\theta+\frac{C_{\delta}(V)(N/L)^{\beta+\delta}}{N}}\bblra{\psi_{N,L},\lrs{\sum_{j=1}^{N}S_{j}^{2}}^{2}\psi_{N,L}}
-\frac{2C_{up}(1+|E_{N,L,\theta}|)^{2}}{\theta},\notag
\end{align}
where we used estimate $(\ref{equ:stability of matter when k=1: moments estimate one order})$ in the last line.

Taking the expectation against the ground state $\psi_{N,L}$ on both sides of $(\ref{eq:stability of matter when k=1: moments estimate second order rewrite})$, we use estimates $(\ref{eq:Stability of matter when k=1 moment estimate first and second terms})$ and $(\ref{eq:Stability of matter when k=1 moment estimate third term for ground state})$ to obtain
\begin{align}
&\frac{2}{N^{2}}\bblra{\psi_{N,L},\lrs{\sum_{j=1}^{N}S_{j}^{2}}^{2}\psi_{N,L}}\\
\leq& \frac{8C_{up}^{2}(1+|E_{N,L,\theta}|)^{2}}{\theta}+
\frac{2}{N^{2}}\lrs{1-\theta+\frac{C_{\delta}(V)(N/L)^{\beta+\delta}}{N}}\bblra{\psi_{N,L},\lrs{\sum_{j=1}^{N}S_{j}^{2}}^{2}\psi_{N,L}}.\notag
\end{align}
Equivalently, we have
\begin{align} \label{eq:Stability of matter when k=1 moment estimate conclude}
\frac{2}{N^{2}}\lrs{\theta-\frac{C_{\delta}(V)(N/L)^{\beta+\delta}}{N}}\bblra{\psi_{N,L},\lrs{\sum_{j=1}^{N}S_{j}^{2}}^{2}\psi_{N,L}}\leq
\frac{8C_{up}^{2}(1+|E_{N,L,\theta}|)^{2}}{\theta}.
\end{align}
Since $\beta<1/2,$ we can take $\delta$ such that $2\beta+\delta<1$. Then, with $L(N/L)^{\beta}\to 1^{-}$, we deduce that
$$\theta-\frac{C_{\delta}(V)(N/L)^{\beta+\delta}}{N}\geq \theta-\frac{C_{\delta}(V)(N/L)^{2\beta+\delta}}{N/L}\geq \theta/2,$$
for large $N$ and $1/L$. With estimate $(\ref{eq:Stability of matter when k=1 moment estimate conclude})$, we conclude
\begin{align*}
Tr \lrs{S_{1}^{2}S_{2}^{2}\gamma_{N,L}^{(2)}}\leq \frac{2}{N^{2}}\bblra{\lrs{\sum_{j=1}^{N}S_{j}^{2}}^{2}\psi_{N,L},\psi_{N,L}}\lesssim \frac{C_{up}^{2}(1+|E_{N,L,\theta}|)^{2}}{\theta^{2}}.
\end{align*}
\end{proof}

Then, we prove the following theorem with a bootstrapping argument.
\begin{theorem} \label{thm:Stability of matter when k=1, beta 1/2, bootstrap}
Assume $L(N/L)^{\beta}\to 1^{-}$, $\beta<\frac{1}{2},$ and $\n{V}_{L_{z}^{\infty}L_{x}^{1}}\leq \frac{2\alpha}{C_{gn}^{4}}$ for some $\alpha\in (0,1)$, then we have
\begin{align*}
\liminf_{\substack{N,1/L\to \infty\\L(N/L)^{\beta}\to 1^{-}}} E_{N,L,\varepsilon}\geq 0
\end{align*}
for $\ve\in [0,1-\al).$
\end{theorem}
\begin{proof}
From Lemma $\ref{lemma:stability of matter when k=1:lower bound on the ground state energy}$ and Lemma $\ref{lemma:stability of matter when k=1: moments estimate}$, we have
\begin{align}\label{equ:stability of matter when k=1:stability of the second}
 E_{N,L}\gtrsim &-C_{\delta}\n{V}_{L_{z}^{\infty}L_{x}^{1}}(1+\lambda_{x}^{2})^{1+2\delta}\frac{d}{N}\\
&-\frac{C_{\delta}\n{V}_{L_{z}^{\infty}L_{x}^{1}}}{\min\lr{1+\lambda_{x}^{2},1+(\lambda_{z}^{2}-1)/L^{2}}^{1/4-\delta/2}}\lrs{\frac{C_{up}(1+|E_{N,L,\theta}|)}{\theta}}^{5/4+3\delta/2}.\notag
\end{align}
Similarly, we use $e_{N,L}^{\varepsilon}$, $E_{N,L}^{\varepsilon}$, $E^{\varepsilon}_{N,L,\theta}$ to denote the ground state energy and $C_{up}^{\varepsilon}$ to denote the upper bound in Lemma $\ref{lemma:stability of matter when k=1: Hartree functional estimate :lower bound}$ with interaction function $V^{\varepsilon}=(1-\varepsilon)^{-1}V$. Then, we obtain
\begin{align}
 E_{N,L}^{\varepsilon}\gtrsim &C_{\delta}\n{V^{\varepsilon}}_{L_{z}^{\infty}L_{x}^{1}}(1+\lambda_{x}^{2})^{1+2\delta}\frac{d}{N}\\
&-\frac{C_{\delta}\n{V^{\varepsilon}}_{L_{z}^{\infty}L_{x}^{1}}}{\min\lr{1+\lambda_{x}^{2},1+(\lambda_{z}^{2}-1)/L^{2}}^{1/4-\delta/2}}
\lrs{\frac{C_{up}^{\varepsilon}(1+|E_{N,L,\theta}^{\varepsilon}|)}{\theta}}^{5/4+3\delta/2},\notag
\end{align}
Noting that $E_{N,L,\varepsilon}=(1-\varepsilon)E^{\varepsilon}_{N,L}$ and $d\lesssim \la_{x}^{2}\la_{z}$, we take $\theta=(\varepsilon'-\varepsilon)/(1-\varepsilon)$ to obtain
\begin{align}\label{equ:stability of matter when k=1:stability of the second, closed relation}
 E_{N,L,\varepsilon}\geq& -C(V,\delta,\varepsilon,\varepsilon') \lrs{ (1+\lambda_{x}^{2})^{1+2\delta}\frac{\la_{x}^{2}\la_{z}}{N}+\frac{|E_{N,L,\varepsilon'}|^{5/4+3\delta/2}}
{\min\lr{1+\lambda_{x}^{2},1+(\lambda_{z}^{2}-1)/L^{2}}^{1/4-\delta/2}}}.
\end{align}
for all $0<\varepsilon<\varepsilon'<1/2$ and $\delta\in (0,1/2)$, $N\geq N(\delta,\varepsilon,\varepsilon')$.

We make the induction hypothesis (labeled $I_{\eta}$)
\begin{align}
\limsup_{\substack{N,1/L\to \infty\\L(N/L)^{\beta}\to 1^{-}}} \frac{|E_{N,L,\varepsilon}|}{1+(N/L)^{\eta}}<\infty\ \text{for all } 0<\varepsilon<1-\alpha.
\end{align}
We may assume $I_{\eta}$ holds for $\eta=\eta_{0}$ as a start point.
Under the assumption $L(N/L)^{\be}\to 1^{-}$, we have
$$N^{-1}\sim (N/L)^{\be-1}.$$
Then, by taking $\la_{z}=2$ and $\lambda_{x}=(N/L)^{\tau}$ with $\tau\leq \be,$ we deduce that $I_{\eta'}$ holds provided that
\begin{align}
\eta'>\max\lr{4\tau+\be-1,5\eta_{0}/4-\tau/2}.
\end{align}
With the optimal choice $\tau=5\eta_{0}/18+2(1-\be)/9$, we get
\begin{align}
\eta'>\eta_{0}+\frac{\eta_{0}+\beta-1}{9}.
\end{align}
Noting that $\tau\leq \be,$ we have to require
\begin{align}\label{equ:stability of matter when k=1:stability of the second starting point}
\eta_{0}\leq \frac{22\be-4}{5}.
\end{align}
By Theorem $\ref{thm:Stability of matter when k=1}$, we can choose $\eta_{0}=3\be-1$, which satisfies $(\ref{equ:stability of matter when k=1:stability of the second starting point})$.
With $\be<1/2$, we choose a constant $c$ such that
$$0<c<\frac{2-4\beta}{9}$$
and $I_{\eta'}$ holds with $\eta'=\eta_{0}-c$. Repeating the process, we finally deduce that $I_{0}$ holds. It means that $|E_{N,L,\varepsilon}|\leq C_{\ve}$ for $N\geq N_{\ve}$. Then  by taking $\la_{z}=2$ and $\lambda_{x}=(N/L)^{\tau}$ with $\tau$ small enough, we use estimate $(\ref{equ:stability of matter when k=1:stability of the second, closed relation})$ to obtain
 \begin{align*}
\liminf_{\substack{N,1/L\to \infty\\L(N/L)^{\beta}\to 1^{-}}} E_{N,L,\varepsilon}\geq 0.
\end{align*}

\end{proof}
By taking $\ve=\frac{1-\al}{\sqrt{2}}$ in Theorem $\ref{thm:Stability of matter when k=1, beta 1/2, bootstrap}$, we obtain Theorem $\ref{thm:Stability of matter when k=1, beta 1/2}$.

\subsection{High Energy estimates when $k>1$} \label{section k>1}
Assuming $(\ref{equ:High energy estimates when k>1})$ holds for $k$, we now prove it for $k+2$. By the induction hypothesis, we have
\begin{align*}
&\frac{1}{c_{0}^{k+2}}\lra{\psi_{N,L},\lrs{2+N^{-1}H_{N,L}-1/L^{2}}^{k+2}\psi_{N,L}}\\
\geq& \frac{1}{c_{0}^{2}}\n{S^{(k)}(2+N^{-1}H_{N,L}-1/L^{2})\psi_{N,L}}_{L^{2}}^{2}\\
=& MS+E_{C}+E_{P}
\end{align*}
where the main sum $MS$ is
\begin{align*}
MS=\frac{1}{c_{0}^{2}N^{2}(N-1)^{2}}\sum_{\substack{1\leq i_{1}<j_{1}\leq N\\ 1\leq i_{2}<j_{2}\leq N\\i_{1}>k,i_{2}>k}}
\lra{S^{(k)}(2+H_{i_{1}j_{1}})\psi_{N,L},S^{(k)}(2+H_{i_{2}j_{2}})\psi_{N,L}},
\end{align*}
the cross error term $E_{C}$ is
\begin{align*}
E_{C}=\frac{1}{c_{0}^{2}N^{2}(N-1)^{2}}\sum_{\substack{1\leq i_{1}<j_{1}\leq N\\ 1\leq i_{2}<j_{2}\leq N\\i_{1}\leq k,i_{2}>k}}
2Re\lra{S^{(k)}(2+H_{i_{1}j_{1}})\psi_{N,L},S^{(k)}(2+H_{i_{2}j_{2}})\psi_{N,L}},
\end{align*}
and the nonnegative error term $E_{P}$ is
\begin{align*}
E_{P}=&\frac{1}{c_{0}^{2}N^{2}(N-1)^{2}}\sum_{\substack{1\leq i_{1}<j_{1}\leq N\\ 1\leq i_{2}<j_{2}\leq N\\i_{1}\leq k,i_{2}\leq k}}
\lra{S^{(k)}(2+H_{i_{1}j_{1}})\psi_{N,L},S^{(k)}(2+H_{i_{2}j_{2}})\psi_{N,L}}\\
=&\frac{1}{c_{0}^{2}N^{2}(N-1)^{2}}\bblra{\sum_{\substack{1\leq i<j\leq N\\ i\leq k}}S^{(k)}(2+H_{ij})\psi_{N,L},\sum_{\substack{1\leq i<j\leq N\\ i\leq k}}S^{(k)}(2+H_{ij})\psi_{N,L}}\\
\geq& 0.
\end{align*}

\subsubsection{\textbf{Handling the main sum}}
Commuting $(1+H_{i_{1}j_{1}})$ and $(1+H_{i_{2}j_{2}})$ with $S^{(k)}$ in $MS$, we get
$$MS=M_{1}+M_{2}+M_{3},$$
where $M_{1}$ consists of the terms with
$$\lr{i_{1},j_{1}}\cap \lr{i_{2},j_{2}}=\emptyset,$$
$M_{2}$ consists of the terms with
$$|\lr{i_{1},j_{1}}\cap \lr{i_{2},j_{2}}|=1,$$
and $M_{3}$ consists of the terms with
$$|\lr{i_{1},j_{1}}\cap \lr{i_{2},j_{2}}|=2.$$
By symmetric of $\psi_{N,L}$,we have
\begin{align*}
&M_{1}=  \frac{1}{4c_{0}^{2}}\lra{S^{(k)}(2+H_{(k+1)(k+2)})\psi_{N,L},S^{(k)}(2+H_{(k+3)(k+4)})\psi_{N,L}},\\
&M_{2}= \frac{1}{2c_{0}^{2}}N^{-1}\lra{S^{(k)}(2+H_{(k+1)(k+2)})\psi_{N,L},S^{(k)}(2+H_{(k+2)(k+3)})\psi_{N,L}},\\
&M_{3}= \frac{1}{2c_{0}^{2}}N^{-2}\lra{S^{(k)}(2+H_{(k+1)(k+2)})\psi_{N,L},S^{(k)}(2+H_{(k+1)(k+2)})\psi_{N,L}},
\end{align*}
up to an unimportant combination number.

Since $M_{3}\geq 0$, we drop it. By the fact that
\begin{align*}
[2+H_{(k+1)(k+2)},2+H_{(k+3)(k+4)}]=0,
\end{align*}
we have
\begin{align*}
M_{1}\geq\frac{4(1-\alpha)^{2}}{4c_{0}^{2}}\lra{S^{(k)}\psi_{N,L},S_{k+1}^{2}S_{k+2}^{2}S^{(k)}\psi_{N,L}}
\end{align*}
using Theorem $\ref{thm:Stability of matter when k=1}$ and Lemma $\ref{lemma:appendix standard operator inequalities}$. Recall $c_{0}=\min\lrs{\frac{1-\alpha}{\sqrt{2}},\frac{1}{2}}$, hence
\begin{align} \label{equ:Stability of matter when k>1 M1 estimate}
M_{1}\geq 2\lra{S^{(k+2)}\psi_{N,L},S^{(k+2)}\psi_{N,L}}=2\n{S^{(k+2)}\psi_{N,L}}_{L^{2}}^{2}.
\end{align}

To deal with $M_{2},$ we expand
$$M_{2}=M_{21}+M_{22}+M_{23},$$
where
\begin{align*}
&M_{21}=\frac{N^{-1}}{2c_{0}^{2}}\lra{\lrs{2+S_{k+1}^{2}+S_{k+2}^{2}}S^{(k)}\psi_{N,L},(2+S_{k+2}^{2}+S_{k+3}^{2})S^{(k)}\psi_{N,L}},\\
&M_{22}=\frac{N^{-1}}{c_{0}^{2}}Re\lra{(2+S_{k+1}^{2}+S_{k+2}^{2})S^{(k)}\psi_{N,L},LV_{N,L}(r_{k+2}-r_{k+3})S^{(k)}\psi_{N,L}},\\
&M_{23}=\frac{N^{-1}}{2c_{0}^{2}}\lra{LV_{N,L}(r_{k+1}-r_{k+2})S^{(k)}\psi_{N,L},LV_{N,L}(r_{k+2}-r_{k+3})S^{(k)}\psi_{N,L}}.
\end{align*}

We keep only the $S_{k+2}^{4}$ terms inside $M_{21}$, which is the main contribution. That is,
\begin{align}
M_{21}\geq& \frac{N^{-1}}{2c_{0}^{2}}\lra{S_{k+2}^{2}S_{k+2}^{2}S^{(k)}\psi_{N,L},S^{(k)}\psi_{N,L}}\geq 2N^{-1}\lra{S_{k+1}^{4}S^{(k)}\psi_{N},S^{(k)}\psi_{N,L}}\\
=&2N^{-1}\n{S_{1}S^{(k+1)}\psi_{N,L}}_{L^{2}}^{2}\notag.
\end{align}

For $M_{22}$, we expand
\begin{align*}
M_{22}=& \frac{2N^{-1}}{c_{0}^{2}}\lra{S^{(k)}\psi_{N,L},LV_{N,L}(r_{k+2}-r_{k+3})S^{(k)}\psi_{N,L}}\\
&+\frac{N^{-1}}{c_{0}^{2}}\lra{S^{(k+1)}\psi_{N,L},LV_{N,L}(r_{k+2}-r_{k+3})S^{(k+1)}\psi_{N,L}}\\
&+\frac{N^{-1}}{c_{0}^{2}}\lra{S^{(k)}\psi_{N,L},S_{k+2}^{2}LV_{N,L}(r_{k+2}-r_{k+3})S^{(k)}\psi_{N,L}}\\
=&M_{221}+M_{222}+M_{223}.
\end{align*}

By estimate $(\ref{equ:Stability of matter when k=1 basic operator estimate: one body estimate})$
\begin{align}
|M_{221}|+|M_{222}|\lesssim & \frac{L(N/L)^{\beta+}}{N}\lrs{\n{S^{(k+1)}\psi_{N,L}}_{L^{2}}^{2}+\n{S^{(k+2)}\psi_{N,L}}_{L^{2}}^{2}}.
\end{align}

By estimate $(\ref{equ:Stability of matter when k=1 basic operator estimate: commutator estimate})$,
\begin{align}
|M_{223}|\lesssim & \frac{(N/L)^{\beta+}}{N}\n{S^{(k+2)}\psi_{N,L}}_{L^{2}}^{2}.
\end{align}
This requires $\beta<\frac{1}{2}.$

For $M_{23}$, with H\"{o}lder inequality, we have
\begin{align} \label{equ:Stability of matter when k>1 M23 estimate}
&|M_{23}|\\
\leq&N^{-1}\n{LV_{N,L}(r_{k+1}-r_{k+2})}_{L_{z_{k+1}}^{\infty}L_{x_{k+1}}^{1+}}\n{LV_{N,L}(r_{k+2}-r_{k+3})}_{L_{z_{k+3}}^{\infty}L_{x_{k+3}}^{1+}}\n{S^{(k)}
\psi_{N,L}}_{L^{2} L_{x_{k+1}}^{\infty-}L_{x_{k+3}}^{\infty-}}^{2} \notag \\
\lesssim& N^{-1}L^{2}(N/L)^{2\beta+2\delta}\n{S^{(k+2)}\psi_{N,L}}_{L^{2}}^{2}.\notag
\end{align}

Putting $(\ref{equ:Stability of matter when k>1 M1 estimate})$--$(\ref{equ:Stability of matter when k>1 M23 estimate})$ together, with the assumption $L(N/L)^{\beta}\to 1^{-}$, we arrive at the following estimate for $MS$:
\begin{equation} \label{equ:Stability of matter when k>1 MS estimate}
MS\geq (2-C(N/L)^{2\beta-1+})\lrs{\n{S^{(k+2)}\psi}_{L^{2}}^{2}+N^{-1}\n{S_{1}S^{(k+1)}\psi}_{L^{2}}^{2}}.
\end{equation}

\subsubsection{\textbf{Handling the cross error term}}
Next, we turn our attention to estimate $E_{C}$. We will prove that
\begin{align}\label{equ:Stability of matter when k>1 Ec estimate}
E_{C}\geq -C\max\lrs{N^{-\frac{3}{2}}(N/L)^{2\beta+},N^{-1}(N/L)^{\beta+}}(\n{S^{(k+2)}\psi_{N,L}}_{L^{2}}^{2}+N^{-1}\n{S_{1}S^{(k+1)}\psi_{N,L}}_{L^{2}}^{2}).
\end{align}
Since $L(N/L)^{\beta}\to 1^{-}$, $(\ref{equ:Stability of matter when k>1 Ec estimate})$ is equivalent to
\begin{align}\label{equ:Stability of matter when k>1 Ec estimate equivalent expression}
E_{C}\geq -C(N/L)^{(7\beta-3)/2}(\n{S^{(k+2)}\psi_{N,L}}_{L^{2}}^{2}+N^{-1}\n{S_{1}S^{(k+1)}\psi_{N,L}}_{L^{2}}^{2}).
\end{align}
That is, when $\beta<3/7$, $E_{C}$ can be absorbable if added into $(\ref{equ:Stability of matter when k>1 MS estimate})$.

We assume $k\geq 1$, since $E_{C}=0$ when $k=0.$ We decompose the sum into three parts
$$E_{C}=E_{1}+E_{2}+E_{3},$$
where $E_{1}$ contains the terms with $j_{1}\leq k,$ $E_{2}$ contains the terms with $j_{1}>k$ and $j_{1}\in \lr{i_{2},j_{2}}$ and $E_{3}$ contains those term with $j_{1}>k$, $j_{1}\neq j_{2}$.

Since $H_{ij}=H_{ji},$ by symmetry of $\psi_{N,L}$, we have
\begin{align*}
&E_{1}=N^{-2}\lra{S^{(k)}\lrs{2+H_{12}}\psi_{N,L},S^{(k)}\lrs{2+H_{(k+1)(k+2)}}\psi_{N,L}},\\
&E_{2}=N^{-2}\lra{S^{(k)}\lrs{2+H_{1(k+1)}}\psi_{N,L},S^{(k)}\lrs{2+H_{(k+1)(k+2)}}\psi_{N,L}},\\
&E_{3}=N^{-1}\lra{S^{(k)}\lrs{2+H_{1(k+1)}}\psi_{N,L},S^{(k)}\lrs{2+H_{(k+2)(k+3)}}\psi_{N,L}},
\end{align*}
up to an unimportant combination number.

when $k=1$, $E_{1}=0$. Therefore, we address $E_{1}$ for $k\geq 2$.
\begin{align*}
E_{1}=E_{11}+E_{12}+E_{13}+E_{14},
\end{align*}
where
\begin{align*}
&E_{11}=N^{-2}\lra{S^{(k)}(2+S_{1}^{2}+S_{2}^{2})\psi_{N,L},S^{(k)}(2+S_{k+1}^{2}+S_{k+2}^{2})\psi_{N,L}},\\
&E_{12}=N^{-2}\lra{S^{(k)}(2+S_{1}^{2}+S_{2}^{2})\psi_{N,L},S^{(k)}LV_{N,L}(r_{k+1}-r_{k+2})\psi_{N,L}},\\
&E_{13}=N^{-2}\lra{S^{(k)}LV_{N,L}(r_{1}-r_{2})\psi_{N,L},S^{(k)}(2+S_{k+1}^{2}+S_{k+2}^{2})\psi_{N,L}},\\
&E_{14}=N^{-2}\lra{S^{(k)}LV_{N,L}(r_{1}-r_{2})\psi_{N,L},S^{(k)}LV_{N,L}(r_{k+1}-r_{k+2})\psi_{N,L}}.
\end{align*}

Since $E_{11}\geq0$, we discard it. For $E_{12}$, by symmetry of $\psi_{N,L}$, we need to only consider
$$N^{-2}\lra{S^{(k)}S_{1}^{2}\psi_{N,L},S^{(k)}LV_{N,L}(r_{k+1}-r_{k+2})\psi_{N,L}}.$$
Since $[V_{N,L}(r_{k+1}-r_{k+2}),S_{1}]=0$, we use estimate $(\ref{equ:Stability of matter when k=1 basic operator estimate: one body estimate})$ to obtain
\begin{align*}
&N^{-2}\lra{S^{(k)}S_{1}^{2}\psi_{N,L},S^{(k)}LV_{N,L}(r_{k+1}-r_{k+2})\psi_{N,L}}\\
=&N^{-2}\lra{S^{(k)}S_{1}\psi_{N,L},LV_{N,L}(r_{k+1}-r_{k+2})S^{(k)}S_{1}\psi_{N,L}}\\
\lesssim & N^{-2}L(N/L)^{\beta+}\n{S_{1}S^{(k+1)}\psi_{N,L}}_{L^{2}}^{2}.
\end{align*}
Hence
\begin{align}\label{equ:Stability of matter when k>1 estimate E12}
|E_{12}|\lesssim & N^{-2}L(N/L)^{\beta+}\n{S_{1}S^{(k+1)}\psi_{N,L}}_{L^{2}}^{2}.
\end{align}

For $E_{13}$, we decompose
\begin{align*}
E_{13}=E_{131}+E_{132}+E_{133},
\end{align*}
where
\begin{align*}
&E_{131}=N^{-2}\lra{S^{(k)}LV_{N,L}(r_{1}-r_{2})\psi_{N,L},S^{(k)}S_{k+1}^{2}\psi_{N,L}},\\
&E_{132}=N^{-2}\lra{S^{(k)}LV_{N,L}(r_{1}-r_{2})\psi_{N,L},S^{(k)}S_{k+2}^{2}\psi_{N,L}},\\
&E_{133}=N^{-2}\lra{S^{(k)}LV_{N,L}(r_{1}-r_{2})\psi_{N,L},2S^{(k)}\psi_{N,L}}.
\end{align*}

For $E_{131}$, we expand
\begin{align*}
E_{131}=&N^{-2}\lra{LV_{N,L}(r_{1}-r_{2})\prod_{j=3}^{k+1}S_{j}\psi_{N,L},S_{1}^{2}S_{2}^{2}\prod_{j=3}^{k+1}S_{j}P_{z_{1},>1}\psi_{N,L}}\\
=&N^{-2}\lra{LV_{N,L}(r_{1}-r_{2})\prod_{j=3}^{k+1}S_{j}\psi_{N,L},(1-\Delta_{r_{1}}-1/L^{2})S_{2}^{2}\prod_{j=3}^{k+1}S_{j}P_{z_{1},>1}\psi_{N,L}}\\
=&E_{1311}+E_{1312},
\end{align*}
where
\begin{align*}
&E_{1311}=N^{-2}\lra{LV_{N,L}(r_{1}-r_{2})\prod_{j=3}^{k+1}S_{j}\psi_{N,L},-\Delta_{r_{1}}S_{2}^{2}\prod_{j=3}^{k+1}S_{j}P_{z_{1},>1}\psi_{N,L}},\\
&E_{1312}=N^{-2}\lra{LV_{N,L}(r_{1}-r_{2})\prod_{j=3}^{k+1}S_{j}\psi_{N,L},(1-1/L^{2})S_{2}^{2}\prod_{j=3}^{k+1}S_{j}P_{z_{1},>1}\psi_{N,L}}.
\end{align*}
Using integration by parts for $E_{1311}$,
\begin{align*}
E_{1311}=&N^{-2}\lra{\nabla_{r_{1}}(LV_{N,L}(r_{1}-r_{2})\prod_{j=3}^{k+1}S_{j}\psi_{N,L}),\nabla_{r_{1}}S_{2}^{2}\prod_{j=3}^{k+1}S_{j}P_{z_{1},>1}\psi_{N,L}}\\
=&N^{-2}L(N/L)^{\beta}\lra{(\nabla_{r_{1}}V)_{N,L}(r_{1}-r_{2})\prod_{j=3}^{k+1}S_{j}\psi_{N,L}),\nabla_{r_{1}}S_{2}^{2}\prod_{j=3}^{k+1}S_{j}P_{z_{1},>1}\psi_{N,L}}\\
&+N^{-2}L\lra{V_{N,L}(r_{1}-r_{2})\nabla_{r_{1}}\prod_{j=3}^{k+1}S_{j}\psi_{N,L},\nabla_{r_{1}}S_{2}^{2}\prod_{j=3}^{k+1}S_{j}P_{z_{1},>1}\psi_{N,L}}.
\end{align*}
Using H\"{o}lder and Sobolev inequality,
\begin{align*}
|E_{1311}|\leq&N^{-2}L(N/L)^{\be}\n{(\nabla V)_{N,L}}_{L_{z}^{\infty}L_{x}^{2+}}\n{\prod_{j=3}^{k+1}S_{j}\psi_{N,L}}_{L^{2}L_{x_{2}}^{\infty-}}\n{\nabla_{r_{1}}S_{2}^{2}\prod_{j=3}^{k+1}S_{j}P_{z_{1},>1}\psi_{N,L}}_{L^{2}}\\
&+N^{-2}L\n{V_{N,L}}_{L_{z}^{\infty}L_{x}^{2+}}\n{\nabla_{r_{1}}\prod_{j=3}^{k+1}S_{j}\psi_{N,L}}_{L^{2}L_{x_{2}}^{\infty-}}
\n{\nabla_{r_{1}}S_{2}^{2}\prod_{j=3}^{k+1}S_{j}P_{z_{1},>1}\psi_{N,L}}_{L^{2}}\\
\lesssim & N^{-2}L(N/L)^{3\beta+}\n{\prod_{j=2}^{k+1}S_{j}\psi_{N,L}}_{L^{2}}\n{\nabla_{r_{1}}S_{2}^{2}\prod_{j=3}^{k+1}S_{j}P_{z_{1},>1}\psi_{N,L}}_{L^{2}}\\
&+N^{-2}L(N/L)^{2\beta+}\n{\nabla_{r_{1}}\prod_{j=2}^{k+1}S_{j}\psi_{N,L}}_{L^{2}}
\n{\nabla_{r_{1}}S_{2}^{2}\prod_{j=2}^{k+1}S_{j}P_{z_{1},>1}\psi_{N,L}}_{L^{2}}
\end{align*}
By estimates $(\ref{equ:Appendix operator inequality 0})$ and $(\ref{equ:Appendix operator inequality 1})$, with $L(N/L)^{\be}\to 1^{-}$, we have
\begin{align}\label{equ:Stability of matter when k>1 E1311 estimate}
|E_{1311}|\lesssim& N^{-2}(N/L)^{2\beta+}\n{\prod_{j=2}^{k+1}S_{j}\psi_{N,L}}_{L^{2}}\n{S_{2}S^{(k+1)}\psi_{N,L}}_{L^{2}}\\
&+N^{-2}(N/L)^{2\beta+}\n{S^{(k+1)}\psi_{N,L}}_{L^{2}}\n{S_{2}S^{(k+1)}\psi_{N,L}}_{L^{2}}\notag\\
\leq& N^{-\frac{3}{2}}(N/L)^{2\beta+}(\n{S^{(k+1)}\psi_{N,L}}_{L^{2}}^{2}+N^{-1}\n{S_{1}S^{(k+1)}\psi_{N,L}}_{L^{2}}^{2}).\notag
\end{align}

For $E_{1312}$, with H\"{o}lder and Sobolev inequality, we obtain
\begin{align*}
|E_{1312}|\leq& N^{-2}L^{-1}\n{V_{N,L}}_{L_{z}^{\infty}L_{x}^{2+}}\n{\prod_{j=3}^{k+1}S_{j}\psi_{N,L}}_{L^{2}L_{x_{1}}^{\infty-}}
\n{S_{2}^{2}\prod_{j=3}^{k+1}S_{j}P_{z_{1},>1}\psi_{N,L}}_{L^{2}}\\
\lesssim&N^{-2}L^{-1}(N/L)^{2\beta+}\n{S_{1}\prod_{j=3}^{k+1}S_{j}\psi_{N,L}}_{L^{2}}\n{S_{2}^{2}\prod_{j=3}^{k+1}S_{j}P_{z_{1},>1}\psi_{N,L}}
_{L^{2}}
\end{align*}
By $L^{-2}P_{z_{1},>1}\leq S_{1}^{2}P_{z_{1},>1}$, we get
\begin{align}\label{equ:Stability of matter when k>1 E1312 estimate}
|E_{1312}|\lesssim& N^{-2}(N/L)^{2\beta+}\n{S^{(k)}\psi_{N,L}}_{L^{2}}\n{S_{1}S^{(k+1)}\psi_{N,L}}_{L^{2}} \\
\leq & N^{-\frac{3}{2}}(N/L)^{2\beta+}(\n{S^{(k)}\psi_{N,L}}_{L^{2}}^{2}+N^{-1}\n{S_{1}S^{(k+1)}\psi_{N,L}}_{L^{2}}^{2}).\notag
\end{align}

Estimated in the same way as $E_{131}$,
\begin{align}
\label{equ:Stability of matter when k>1 E132 estimate}|E_{132}|\lesssim N^{-\frac{3}{2}}(N/L)^{2\beta+}(\n{S^{(k)}\psi_{N,L}}_{L^{2}}^{2}+N^{-1}\n{S_{1}S^{(k+1)}\psi_{N,L}}_{L^{2}}^{2}),\\
\label{equ:Stability of matter when k>1 E133 estimate} |E_{133}|\lesssim N^{-\frac{3}{2}}(N/L)^{2\beta+}(\n{S^{(k)}\psi_{N,L}}_{L^{2}}^{2}+N^{-1}\n{S_{1}S^{(k+1)}\psi_{N,L}}_{L^{2}}^{2}).
\end{align}

For $E_{14}$, we decompose
\begin{align*}
E_{14}=&N^{-2}\lra{LV_{N,L}(r_{1}-r_{2})\prod_{j=3}^{k}S_{j}\psi_{N,L},S_{1}^{2}S_{2}^{2}LV_{N,L}(r_{k+1}-r_{k+2})\prod_{j=3}^{k}S_{j}P_{z_{1},>1}\psi_{N,L}}\\
=&E_{141}+E_{142},
\end{align*}
where
\begin{align*}
&E_{141}=N^{-2}\lra{LV_{N,L}(r_{1}-r_{2})\prod_{j=3}^{k}S_{j}\psi_{N,L},-\Delta_{r_{1}}LV_{N,L}(r_{k+1}-r_{k+2})S_{2}^{2}\prod_{j=3}^{k}S_{j}P_{z_{1},>1}\psi_{N,L}},\\
&E_{142}=N^{-2}\lra{LV_{N,L}(r_{1}-r_{2})\prod_{j=3}^{k}S_{j}\psi_{N,L},(1-1/L^{2})LV_{N,L}(r_{k+1}-r_{k+2})S_{2}^{2}\prod_{j=3}^{k}S_{j}P_{z_{1},>1}\psi_{N,L}}.
\end{align*}

For $E_{141}$, using integration by parts, we obtain
\begin{align*}
E_{141}=&N^{-2}L(N/L)^{\be}\lra{\nabla_{r_{1}}(LV_{N,L}(r_{1}-r_{2})\prod_{j=3}^{k}S_{j}\psi_{N,L}),\nabla_{r_{1}}LV_{N,L}(r_{k+1}-r_{k+2})S_{2}^{2}\prod_{j=3}^{k}S_{j}P_{z_{1},>1}\psi_{N,L}},\\
=&N^{-2}L(N/L)^{\be}\lra{(\nabla_{r_{1}}V)_{N,L}(r_{1}-r_{2})\prod_{j=3}^{k}S_{j}\psi_{N,L},\nabla_{r_{1}}LV_{N,L}(r_{k+1}-r_{k+2})S_{2}^{2}\prod_{j=3}^{k}S_{j}P_{z_{1},>1}\psi_{N,L}}\\
&+N^{-2}L\lra{V_{N,L}(r_{1}-r_{2})\nabla_{r_{1}}\prod_{j=3}^{k}S_{j}\psi_{N,L},\nabla_{r_{1}}LV_{N,L}(r_{k+1}-r_{k+2})S_{2}^{2}\prod_{j=3}^{k}S_{j}P_{z_{1},>1}\psi_{N,L}}.
\end{align*}
Using H\"{o}lder and Sobolev inequality,
\begin{align*}
&|E_{141}|\\
\leq& N^{-2}L^{2}(N/L)^{\beta}\n{(\nabla_{r}V)_{N,L}}_{L_{z}^{\infty}L_{x}^{2+}}\n{\prod_{j=3}^{k}S_{j}\psi_{N,L}}_{L^{2}L_{x_{2}}^{\infty-}L_{x_{k+1}}^{\infty-}}
\n{V_{N,L}}_{L_{z}^{\infty}L_{x}^{1+}}\n{\nabla_{r_{1}}S_{2}^{2}\prod_{j=3}^{k}S_{j}P_{z_{1},>1}\psi_{N,L}}_{L^{2}L_{x_{k+1}}^{\infty-}}\\
&+N^{-2}L^{2}\n{V_{N,L}}_{L_{z}^{\infty}L_{x}^{2+}}
\n{\nabla_{r_{1}}\prod_{j=3}^{k}S_{j}\psi_{N,L}}_{L^{2}L_{x_{2}}^{\infty-}L_{x_{k+1}}^{\infty-}}\n{V_{N,L}}_{L_{z}^{\infty}L_{x}^{1+}}
\n{\nabla_{r_{1}}S_{2}^{2}\prod_{j=3}^{k}S_{j}P_{z_{1},>1}\psi_{N,L}}_{L^{2}L_{x_{k+1}}^{\infty-}}\\
\lesssim&N^{-2}L^{2}(N/L)^{4\beta+}\n{\prod_{j=2}^{k+1}S_{j}\psi_{N,L}}_{L^{2}}
\n{\nabla_{r_{1}}S_{2}^{2}\prod_{j=3}^{k+1}S_{j}P_{z_{1},>1}\psi_{N,L}}_{L^{2}}\\
&+N^{-2}L^{2}(N/L)^{3\beta+}\n{\nabla_{r_{1}}\prod_{j=2}^{k+1}S_{j}\psi_{N,L}}_{L^{2}}
\n{\nabla_{r_{1}}S_{2}^{2}\prod_{j=3}^{k+1}S_{j}P_{z_{1},>1}\psi_{N,L}}_{L^{2}}.
\end{align*}
By estimates $(\ref{equ:Appendix operator inequality 0})$ and $(\ref{equ:Appendix operator inequality 1})$, with $L(N/L)^{\beta}\to 1^{-}$, we have
\begin{align}\label{Stability of matter when k>1 E141 estimate}
|E_{141}|\lesssim& N^{-2}(N/L)^{2\beta+}\n{\prod_{j=2}^{k+1}S_{j}\psi_{N,L}}_{L^{2}}\n{S_{1}S^{(k+1)}\psi_{N,L}}_{L^{2}}\\
&+N^{-2}(N/L)^{2\beta+}\n{S^{(k+1)}\psi_{N,L}}_{L^{2}}\n{S_{1}S^{(k+1)}\psi_{N,L}}_{L^{2}}\notag\\
\leq& N^{-\frac{3}{2}}(N/L)^{2\beta+}(\n{S^{(k+1)}\psi_{N,L}}_{L^{2}}^{2}+N^{-1}\n{S_{1}S^{(k+1)}\psi_{N,L}}_{L^{2}}^{2}).\notag
\end{align}

For $E_{142}$, with H\"{o}lder and Sobolev inequality, we obtain
\begin{align}\label{Stability of matter when k>1 E142 estimate}
|E_{142}|\lesssim& N^{-2}\n{V_{N,L}}_{L_{z}^{\infty}L_{x}^{1+}}\n{\prod_{j=3}^{k}S_{j}\psi_{N,L}}_{L^{2}L_{x_{1}}^{\infty-}L_{x_{k+1}}^{\infty-}}\n{V_{N,L}}_{L_{z}^{\infty}L_{x}^{1+}}
\n{S_{2}^{2}\prod_{j=3}^{k}S_{j}P_{z_{1},>1}\psi_{N,L}}_{L^{2}L_{x_{1}}^{\infty-}L_{x_{k+1}}^{\infty-}}\\
\lesssim& N^{-2}(N/L)^{2\beta+}\n{S^{(k)}\psi_{N,L}}_{L^{2}}\n{S_{1}S^{(k+1)}\psi_{N,L}}_{L^{2}}\notag\\
\leq&N^{-\frac{3}{2}}(N/L)^{2\beta+}(\n{S^{(k+1)}\psi_{N,L}}_{L^{2}}^{2}+N^{-1}\n{S_{1}S^{(k+1)}\psi_{N,L}}_{L^{2}}^{2}).\notag
\end{align}
Hence, we obtain
\begin{align}\label{equ:Stability of matter when k>1 E1 estimate}
E_{1}\geq -C N^{-\frac{3}{2}}(N/L)^{2\beta+}(\n{S^{(k+1)}\psi_{N,L}}_{L^{2}}^{2}+N^{-1}\n{S_{1}S^{(k+1)}\psi_{N,L}}_{L^{2}}^{2}).
\end{align}

Next, we deal with $E_{2}.$ We write
\begin{align*}
E_{2}=E_{21}+E_{22}+E_{23}+E_{24},
\end{align*}
where
\begin{align*}
&E_{21}=N^{-2}\lra{S^{(k)}(2+S_{1}^{2}+S_{k+1}^{2})\psi_{N,L},S^{(k)}(2+S_{k+1}^{2}+S_{k+2}^{2})\psi_{N,L}},\\
&E_{22}=N^{-2}\lra{S^{(k)}(2+S_{1}^{2}+S_{k+1}^{2})\psi_{N,L},S^{(k)}LV_{N,L}(r_{k+1}-r_{k+2})\psi_{N,L}},\\
&E_{23}=N^{-2}\lra{S^{(k)}LV_{N,L}(r_{1}-r_{k+1})\psi_{N,L},S^{(k)}(2+S_{k+1}^{2}+S_{k+2}^{2})\psi_{N,L}},\\
&E_{24}=N^{-2}\lra{S^{(k)}LV_{N,L}(r_{1}-r_{k+1})\psi_{N,L},S^{(k)}LV_{N,L}(r_{k+1}-r_{k+2})\psi_{N,L}}.
\end{align*}

Since $E_{21}\geq 0$, we can discard it. For $E_{22}$, we decompose
\begin{align*}
E_{22}=E_{221}+E_{222}+E_{223},
\end{align*}
where
\begin{align*}
&E_{221}=2N^{-2}\lra{S^{(k)}\psi_{N,L},S^{(k)}LV_{N,L}(r_{k+1}-r_{k+2})\psi_{N,L}},\\
&E_{222}=N^{-2}\lra{S^{(k)}S_{1}^{2}\psi_{N,L},S^{(k)}LV_{N,L}(r_{k+1}-r_{k+2})\psi_{N,L}},\\
&E_{223}=N^{-2}\lra{S^{(k)}S_{k+1}^{2}\psi_{N,L},S^{(k)}LV_{N,L}(r_{k+1}-r_{k+2})\psi_{N,L}}.
\end{align*}
By estimate ($\ref{equ:Stability of matter when k=1 basic operator estimate: one body estimate}$), we obtain
\begin{align}
&\label{equ:Stability of matter when k>1 E221 estimate}|E_{221}|\lesssim N^{-2}L(N/L)^{\beta+}\n{S^{(k+1)}\psi_{N,L}}_{L^{2}}^{2},\\
&\label{equ:Stability of matter when k>1 E222 estimate}|E_{222}|\lesssim N^{-2}L(N/L)^{\beta+}\n{S_{1}S^{(k+1)}\psi_{N,L}}_{L^{2}}^{2}.
\end{align}

For $E_{223}$, by H\"{o}lder and Sobolev inequality, we have
\begin{align}\label{equ:Stability of matter when k>1 E223 estimate}
|E_{223}|\leq& N^{-2}\n{S_{1}S^{(k+1)}\psi_{N,L}}_{L^{2}}\n{LV_{N,L}}_{L_{z}^{\infty}L_{x}^{2+}}\n{S^{(k)}\psi_{N,L}}_{L^{2}L_{x_{k+1}}^{\infty-}}\\
\lesssim& N^{-2}L(N/L)^{2\beta+}\n{S_{1}S^{(k+1)}\psi_{N,L}}_{L^{2}}^{2}.\notag
\end{align}

For $E_{23},$ we expand
\begin{align*}
E_{23}=E_{231}+E_{232}+E_{233},
\end{align*}
where
\begin{align*}
&E_{231}=N^{-2}\lra{S^{(k)}LV_{N,L}(r_{1}-r_{k+1})\psi_{N,L},S^{(k)}S_{k+2}^{2}\psi_{N,L}},\\
&E_{232}=N^{-2}\lra{S^{(k)}LV_{N,L}(r_{1}-r_{k+1})\psi_{N,L},S^{(k)}S_{k+1}^{2}\psi_{N,L}},\\
&E_{233}=N^{-2}\lra{S^{(k)}LV_{N,L}(r_{1}-r_{k+1})\psi_{N,L},2S^{(k)}\psi_{N,L}}.
\end{align*}

For $E_{231}$, we expand
\begin{align*}
E_{231}=&N^{-2}\lra{\prod_{j=2}^{k}S_{j}S_{k+2}LV_{N,L}(r_{1}-r_{k+1})\psi_{N,L},S_{1}^{2}\prod_{j=2}^{k}S_{j}S_{k+2}P_{z_{1},>1}\psi_{N,L}}\\
=&E_{2311}+E_{2312},
\end{align*}
where
\begin{align*}
&E_{2311}=N^{-2}\lra{\prod_{j=2}^{k}S_{j}S_{k+2}LV_{N,L}(r_{1}-r_{k+1})\psi_{N,L},-\Delta_{r_{1}}\prod_{j=2}^{k}S_{j}S_{k+2}P_{z_{1},>1}\psi_{N,L}},\\
&E_{2312}=N^{-2}\lra{\prod_{j=2}^{k}S_{j}S_{k+2}LV_{N,L}(r_{1}-r_{k+1})\psi_{N,L},(1-1/L^{2})\prod_{j=2}^{k}S_{j}S_{k+2}P_{z_{1},>1}\psi_{N,L}}.
\end{align*}

For  $E_{2311}$, using integration by parts, we have
\begin{align*}
E_{2311}=&N^{-2}\lra{\nabla_{r_{1}}(LV_{N,L}(r_{1}-r_{k+1})\prod_{j=2}^{k}S_{j}S_{k+2}\psi_{N,L}),\nabla_{r_{1}}\prod_{j=2}^{k}S_{j}S_{k+2}P_{z_{1},>1}\psi_{N,L}}\\
=&N^{-2}L(N/L)^{\beta}\lra{(\nabla V)_{N,L}(r_{1}-r_{k+1})\prod_{j=2}^{k}S_{j}S_{k+2}\psi_{N,L},\nabla_{r_{1}}\prod_{j=2}^{k}S_{j}S_{k+2}P_{z_{1},>1}\psi_{N,L}}\\
&+N^{-2}L\lra{ V_{N,L}(r_{1}-r_{k+1})\nabla_{r_{1}}\prod_{j=2}^{k}S_{j}S_{k+2}\psi_{N,L},\nabla_{r_{1}}\prod_{j=2}^{k}S_{j}S_{k+2}P_{z_{1},>1}\psi_{N,L}}.
\end{align*}
Using H\"{o}lder and Sobolev inequality, we obtain
\begin{align*}
|E_{2311}|\leq& N^{-2}L(N/L)^{\beta}\n{(\nabla_{r_{1}}V)_{N,L}}_{L_{z}^{\infty}L_{x}^{1+}}
\n{\prod_{j=2}^{k}S_{j}S_{k+2}\psi_{N,L}}_{L^{2}L_{x_{k+1}}^{\infty-}}
\n{\nabla_{r_{1}}\prod_{j=2}^{k}S_{j}S_{k+2}P_{z_{1},>1}\psi_{N,L}}_{L^{2}L_{x_{k+1}}^{\infty-}}\\
&+N^{-2}L\n{V_{N,L}}_{L_{z}^{\infty}L_{x}^{1+}}
\n{\nabla_{r_{1}}\prod_{j=2}^{k}S_{j}S_{k+2}\psi_{N,L}}_{L^{2}L_{x_{k+1}}^{\infty-}}
\n{\nabla_{r_{1}}\prod_{j=2}^{k}S_{j}S_{k+2}P_{z_{1},>1}\psi_{N,L}}_{L^{2}L_{x_{k+1}}^{\infty-}}\\
\lesssim& N^{-2}L(N/L)^{2\beta+}\n{\prod_{j=2}^{k+2}S_{j}\psi_{N,L}}_{L^{2}}
\n{\nabla_{r_{1}}\prod_{j=2}^{k+2}S_{j}P_{z_{1},>1}\psi_{N,L}}_{L^{2}}\\
&+N^{-2}L(N/L)^{\beta+}\n{\nabla_{r_{1}}\prod_{j=2}^{k+2}S_{j}\psi_{N,L}}_{L^{2}}
\n{\nabla_{r_{1}}\prod_{j=2}^{k+2}S_{j}P_{z_{1},>1}\psi_{N,L}}_{L^{2}}.
\end{align*}
By estimates $(\ref{equ:Appendix operator inequality 0})$ and $(\ref{equ:Appendix operator inequality 1})$, with $L(N/L)^{\beta}\to 1^{-}$, we have
\begin{align}\label{equ:Stability of matter when k>1 E2311 estimate}
|E_{2311}|\lesssim N^{-2}(N/L)^{\beta+}\n{S^{(k+2)}\psi_{N,L}}_{L^{2}}^{2}.
\end{align}

For $E_{2312}$, with H\"{o}lder and Sobolev inequality, we have
\begin{align*}
|E_{2312}| \lesssim &N^{-2}L^{-2}\n{LV_{N,L}}_{L_{z}^{\infty}L_{x}^{1+}}\n{\prod_{j=2}^{k}S_{j}S_{k+2}\psi_{N,L}}_{L^{2}L_{x_{k+1}}^{\infty-}}
\n{\prod_{j=2}^{k}S_{j}S_{k+2}P_{z_{1},>1}\psi_{N,L}}_{L^{2}L_{x_{k+1}}^{\infty-}}\\
\lesssim& N^{-2}L^{-1}(N/L)^{\beta+}\n{\prod_{j=2}^{k+2}S_{j}\psi_{N,L}}_{L^{2}}\n{\prod_{j=2}^{k+2}S_{j}P_{z_{1},>1}\psi_{N,L}}_{L^{2}}.\notag
\end{align*}
By $L^{-2}P_{z_{1},>1}\leq S_{1}^{2}P_{z_{1},>1}$, we get
\begin{align}\label{equ:Stability of matter when k>1 E2312 estimate}
|E_{2312}|\lesssim N^{-2}(N/L)^{\beta+}\n{S^{(k+1)}\psi_{N,L}}_{L^{2}}\n{S^{(k+2)}\psi_{N,L}}_{L^{2}}
\end{align}

Estimated in the same way as $E_{131}$,
\begin{align}
\label{equ:Stability of matter when k>1 E232 estimate}|E_{232}|\lesssim& N^{-\frac{3}{2}}(N/L)^{2\beta+}(\n{S^{(k+1)}\psi_{N,L}}_{L^{2}}^{2}+N^{-1}\n{S_{1}S^{(k+1)}\psi_{N,L}}_{L^{2}}^{2}).
\end{align}

Estimated in the same way as $E_{231}$,
\begin{align}
\label{equ:Stability of matter when k>1 E233 estimate}|E_{233}|\lesssim& N^{-2}(N/L)^{\beta+}\n{S^{(k+2)}\psi_{N,L}}_{L^{2}}^{2}.
\end{align}

For $E_{24}$, we decompose
\begin{align*}
E_{24}=&N^{-2}\lra{LV_{N,L}(r_{1}-r_{k+1})\prod_{j=2}^{k}S_{j}\psi_{N,L},S_{1}^{2}LV_{N,L}(r_{k+1}-r_{k+2})\prod_{j=2}^{k}S_{j}P_{z_{1},>1}\psi_{N,L}}\\
=&E_{241}+E_{242},
\end{align*}
where
\begin{align*}
&E_{241}=N^{-2}\lra{LV_{N,L}(r_{1}-r_{k+1})\prod_{j=2}^{k}S_{j}\psi_{N,L},-\Delta_{r_{1}}LV_{N,L}(r_{k+1}-r_{k+2})\prod_{j=2}^{k}S_{j}P_{z_{1},>1}\psi_{N,L}},\\
&E_{242}=N^{-2}\lra{LV_{N,L}(r_{1}-r_{k+1})\prod_{j=2}^{k}S_{j}\psi_{N,L},(1-1/L^{2})LV_{N,L}(r_{k+1}-r_{k+2})\prod_{j=2}^{k}S_{j}P_{z_{1},>1}\psi_{N,L}}.
\end{align*}

For $E_{241},$ with H\"{o}lder and Sobolev inequality, we obtain
\begin{align*}
|E_{241}|\leq & N^{-2}\n{LV_{N,L}}_{L_{z}^{\infty}L_{x_{1}}^{2+}}\n{\prod_{j=2}^{k}S_{j}\psi_{N,L}}_{L^{2}L_{x_{k+1}}^{\infty-}L_{x_{1}}^{\infty-}}\n{LV_{N,L}}_{L_{z}^{\infty}
L_{x_{k+1}}^{2+}}\n{\Delta_{r_{1}}\prod_{j=2}^{k}S_{j}P_{z_{1},>1}
\psi_{N,L}}_{L^{2}}\\
\lesssim&N^{-2}L^{2}(N/L)^{4\beta+}\n{S^{(k+1)}\psi_{N,L}}_{L^{2}}\n{\Delta_{r_{1}}\prod_{j=2}^{k}S_{j}P_{z_{1},>1}
\psi_{N,L}}_{L^{2}}.
\end{align*}
By estimate $(\ref{equ:Appendix operator inequality 1})$, with $L(N/L)^{\beta}\to 1^{-}$, we have
\begin{align}\label{equ:Stability of matter when k>1 E241 estimate}
|E_{241}|\leq&N^{-2}(N/L)^{2\beta+}\n{S^{(k+1)}\psi_{N,L}}_{L^{2}}\n{S_{1}S^{(k+1)}\psi_{N,L}}_{L^{2}}\\
\leq&N^{-\frac{3}{2}}(N/L)^{2\beta+}(\n{S^{(k+1)}\psi_{N,L}}_{L^{2}}^{2}+N^{-1}\n{S_{1}S^{(k+1)}\psi_{N,L}}_{L^{2}}^{2}).\notag
\end{align}

For $E_{242}$, with H\"{o}lder and Sobolev inequality, we have
\begin{align}\label{equ:Stability of matter when k>1 E242 estimate}
|E_{242}|\lesssim& N^{-2}\n{V_{N,L}}_{L_{z}^{\infty}L_{x_{1}}^{1+}}\n{\prod_{j=2}^{k}S_{j}\psi_{N,L}}_{L^{2}L_{x_{k+1}}^{\infty-}L_{x_{1}}^{\infty-}}\n{V_{N,L}}_{L_{z}^{\infty}
L_{x_{k+1}}^{1+}}\n{\prod_{j=2}^{k}S_{j}P_{z_{1},>1}
\psi_{N,L}}_{L^{2}L_{x_{k+1}}^{\infty-}L_{x_{1}}^{\infty-}}\\
\lesssim& N^{-2}(N/L)^{2\beta+}\n{S^{(k+1)}\psi_{N,L}}_{L^{2}}^{2}.\notag
\end{align}
Hence, we get
\begin{align}\label{equ:Stability of matter when k>1 E2 estimate}
E_{2}\gtrsim -\max \lrs{N^{-\frac{3}{2}}(N/L)^{2\beta+},N^{-2}(N/L)^{2\beta+}}(\n{S^{(k+2)}\psi_{N,L}}_{L^{2}}^{2}+N^{-1}\n{S_{1}S^{(k+1)}\psi_{N,L}}_{L^{2}}^{2}).
\end{align}

Finally, we handle $E_{3}$ and expand
\begin{align*}
E_{3}=E_{31}+E_{32}+E_{33},
\end{align*}
where
\begin{align*}
&E_{31}=N^{-1}\lra{S^{(k)}(2+S_{1}^{2}+S_{k+1}^{2})\psi_{N,L},S^{(k)}(2+H_{(k+2)(k+3)})\psi_{N,L}},\\
&E_{32}=N^{-1}\lra{S^{(k)}LV_{N,L}(r_{1}-r_{k+1})\psi_{N,L},S^{(k)}(2+S_{k+2}^{2}+S_{k+3}^{2})\psi_{N,L}},\\
&E_{33}=N^{-1}\lra{S^{(k)}LV_{N,L}(r_{1}-r_{k+1})\psi_{N,L},S^{(k)}LV_{N,L}(r_{k+2}-r_{k+3})\psi_{N,L}}.
\end{align*}

We first discard $E_{31}$, since $E_{31}\geq 0$ by Theorem $\ref{thm:Stability of matter when k=1}$ and Lemma $\ref{lemma:appendix standard operator inequalities}$. For $E_{32}$, we expand
\begin{align*}
E_{32}=E_{321}+E_{322}+E_{323},
\end{align*}
where
\begin{align*}
&E_{321}=N^{-1}\lra{S^{(k)}LV_{N,L}(r_{1}-r_{k+1})\psi_{N,L},S^{(k)}S_{k+2}^{2}\psi_{N,L}},\\
&E_{322}=N^{-1}\lra{S^{(k)}LV_{N,L}(r_{1}-r_{k+1})\psi_{N,L},S^{(k)}S_{k+3}^{2}\psi_{N,L}},\\
&E_{323}=N^{-1}\lra{S^{(k)}LV_{N,L}(r_{1}-r_{k+1})\psi_{N,L},2S^{(k)}\psi_{N,L}}.
\end{align*}
Estimated in the same way as $E_{231},$
\begin{align}
&\label{equ:Stability of matter when k>1 E321 estimate}|E_{321}|\lesssim N^{-1}(N/L)^{\beta+}\n{S^{(k+2)}\psi_{N,L}}_{L^{2}}^{2},\\
&\label{equ:Stability of matter when k>1 E322 estimate}|E_{322}|\lesssim N^{-1}(N/L)^{\beta+}\n{S^{(k+2)}\psi_{N,L}}_{L^{2}}^{2},\\
&\label{equ:Stability of matter when k>1 E323 estimate}|E_{323}|\lesssim N^{-1}(N/L)^{\beta+}\n{S^{(k+2)}\psi_{N,L}}_{L^{2}}^{2}.
\end{align}

For $E_{33}$, we expand
\begin{align*}
E_{33}=&N^{-1}\lra{LV_{N,L}(r_{1}-r_{k+1})\prod_{j=2}^{k}S_{j}\psi_{N,L},S_{1}^{2}LV_{N,L}(r_{k+2}-r_{k+3})\prod_{j=2}^{k}S_{j}\psi_{N,L}}\\
=&E_{331}+E_{332},
\end{align*}
where
\begin{align*}
&E_{331}=N^{-1}\lra{LV_{N,L}(r_{1}-r_{k+1})\prod_{j=2}^{k}S_{j}\psi_{N,L},-\Delta_{r_{1}}LV_{N,L}(r_{k+2}-r_{k+3})\prod_{j=2}^{k}S_{j}P_{z_{1},>1}\psi_{N,L}},\\
&E_{332}=N^{-1}\lra{LV_{N,L}(r_{1}-r_{k+1})\prod_{j=2}^{k}S_{j}\psi_{N,L},(1-1/L^{2})LV_{N,L}(r_{k+2}-r_{k+3})\prod_{j=2}^{k}S_{j}P_{z_{1},>1}\psi_{N,L}}.
\end{align*}
Using integration by parts for $E_{331}$,
\begin{align*}
E_{331}=&N^{-1}\lra{\nabla_{r_{1}}LV_{N,L}(r_{1}-r_{k+1})\prod_{j=2}^{k}S_{j}\psi_{N,L},\nabla_{r_{1}}LV_{N,L}(r_{k+2}-r_{k+3})\prod_{j=2}^{k}S_{j}P_{z_{1},>1}\psi_{N,L}}\\
=&N^{-1}\lra{(\nabla_{r_{1}}V)_{N,L}(r_{1}-r_{k+1})\prod_{j=2}^{k}S_{j}\psi_{N,L},\nabla_{r_{1}}LV_{N,L}(r_{k+2}-r_{k+3})\prod_{j=2}^{k}S_{j}P_{z_{1},>1}\psi_{N,L}}\\
&+N^{-1}L\lra{V_{N,L}(r_{1}-r_{k+1})\nabla_{r_{1}}\prod_{j=2}^{k}S_{j}\psi_{N,L},\nabla_{r_{1}}LV_{N,L}(r_{k+2}-r_{k+3})\prod_{j=2}^{k}S_{j}P_{z_{1},>1}\psi_{N,L}}.
\end{align*}
Using H\"{o}lder and Sobolev inequality,
\begin{align*}
&|E_{331}|\\
\leq& N^{-1}L\n{(\nabla_{r}V)_{N,L}}_{L_{z}^{\infty}L_{x}^{1+}}\n{\prod_{j=2}^{k}S_{j}\psi_{N,L}}_{L^{2}L_{x_{k+1}}^{\infty-}L_{x_{k+2}}^{\infty-}}
\n{V_{N,L}}_{L_{z}^{\infty}L_{x}^{1+}}\n{\nabla_{r_{1}}\prod_{j=2}^{k}S_{j}P_{z_{1},>1}\psi_{N,L}}_{L^{2}L_{x_{k+1}}^{\infty-}L_{x_{k+2}}^{\infty-}}\\
&+N^{-1}L^{2}\n{V_{N,L}}_{L_{z}^{\infty}L_{x}^{1+}}\n{\nabla_{r_{1}}\prod_{j=2}^{k}S_{j}\psi_{N,L}}_{L^{2}L_{x_{k+1}}^{\infty-}L_{x_{k+2}}^{\infty-}}\n{V_{N,L}}_{L_{z}^{\infty}L_{x}^{1+}}
\n{\nabla_{r_{1}}\prod_{j=2}^{k}S_{j}P_{z_{1},>1}\psi_{N,L}}_{L^{2}L_{x_{k+1}}^{\infty-}L_{x_{k+2}}^{\infty-}}\\
\lesssim &N^{-1}L(N/L)^{2\beta+}\n{\prod_{j=2}^{k+2}S_{j}\psi_{N,L}}_{L^{2}}
\n{\nabla_{r_{1}}\prod_{j=2}^{k+2}S_{j}P_{z_{1},>1}\psi_{N,L}}_{L^{2}}\\
&+N^{-1}L^{2}(N/L)^{2\beta+}\n{\nabla_{r_{1}}\prod_{j=2}^{k+2}S_{j}\psi_{N,L}}_{L^{2}}
\n{\nabla_{r_{1}}\prod_{j=2}^{k+2}S_{j}P_{z_{1},>1}\psi_{N,L}}_{L^{2}}.
\end{align*}
By estimates $(\ref{equ:Appendix operator inequality 0})$ and $(\ref{equ:Appendix operator inequality 1})$, with $L(N/L)^{\beta}\to 1^{-}$, we have
\begin{align}\label{equ:Stability of matter when k>1 E331 estimate}
|E_{331}|\lesssim& N^{-1}(N/L)^{\beta+}\n{S^{(k+1)}\psi_{N,L}}_{L^{2}}\n{S^{(k+2)}\psi_{N,L}}_{L^{2}}\\
&+N^{-1}(N/L)^{\beta+}\n{S^{(k+2)}\psi_{N,L}}_{L^{2}}\n{S^{(k+2)}\psi_{N,L}}_{L^{2}}.\notag
\end{align}

For $E_{332}$,  with H\"{o}lder and Sobolev inequality,  we obtain
\begin{align*}
|E_{332}|\lesssim& N^{-1}\n{V_{N,L}}_{L_{z}^{\infty}L_{x}^{1+}}\n{\prod_{j=2}^{k}S_{j}\psi_{N,L}}_{L^{2}L_{x_{k+1}}^{\infty-}L_{x_{k+2}}^{\infty-}}\n{V_{N,L}}_{L_{z}^{\infty}L_{x}^{1+}}
\n{\prod_{j=2}^{k}S_{j}P_{z_{1},>1}\psi_{N,L}}_{L^{2}L_{x_{k+1}}^{\infty-}L_{x_{k+2}}^{\infty-}}\\
\lesssim& N^{-1}(N/L)^{2\beta+}\n{\prod_{j=2}^{k+2}S_{j}\psi_{N,L}}_{L^{2}}\n{\prod_{j=2}^{k+2}S_{j}P_{z_{1},>1}\psi_{N,L}}_{L^{2}}.
\end{align*}
By $L^{-2}P_{z_{1},>1}\leq S_{1}^{2}P_{z_{1},>1}$, with $L(N/L)^{\beta}\to 1^{-}$, we get
\begin{align}\label{equ:Stability of matter when k>1 E332 estimate}
|E_{332}|\leq& N^{-1}(N/L)^{\beta+}\n{S^{(k+1)}\psi_{N,L}}_{L^{2}}\n{S^{(k+2)}\psi_{N,L}}_{L^{2}}.
\end{align}
That is
\begin{align}\label{equ:Stability of matter when k>1 E3 estimate}
E_{3}\gtrsim -N^{-1}(N/L)^{\beta+}\n{S^{(k+2)}\psi_{N,L}}_{L^{2}}^{2}.
\end{align}

Putting $(\ref{equ:Stability of matter when k>1 E1 estimate})$, $(\ref{equ:Stability of matter when k>1 E2 estimate})$ and $(\ref{equ:Stability of matter when k>1 E3 estimate})$ together, we obtain the estimate for the cross error term
\begin{align}\label{equ:Stability of matter when k>1 Ec estimate conclude}
E_{C}\geq -C\max\lrs{N^{-\frac{3}{2}}(N/L)^{2\beta+},N^{-1}(N/L)^{\beta+}}(\n{S^{(k+2)}\psi_{N,L}}_{L^{2}}^{2}+N^{-1}\n{S_{1}S^{(k+1)}\psi_{N,L}}_{L^{2}}^{2}).
\end{align}

Hence we have proved for all $k$ and established Theorem $\ref{thm: Stability of matter when k>=1}$.

\section{Compactness, Convergence, and Uniqueness} \label{section 3}

To work on compactness, convergence and uniqueness, we introduce an appropriate topology  on the density matrices, as was previously done in \cite{chen2011quintic,chen2012collapsing,
chen2013rigorous2dfrom3d,chen2017focusing,chen2017rigorous2dfocusing,erdHos2006derivation,erdHos2007derivation,
erdHos2009rigorous,
erdos2010derivation,erdHos2007rigorous,kirkpatrick2011derivation,sohinger2015rigorous}. Denote the spaces of compact operators and trace class operators on $L^{2}(\Om^{\otimes k})$ as $\ck_{k}$ and $\cl_{k}^{1}$, respectively. Then $(\ck_{k})'=\cl_{k}^{1}.$ By the fact that $\ck_{k}$ is separable, we select
a dense countable subset $\lr{J_{i}^{(k)}}_{i\geq 1}\subset \ck_{k}$ in the unit ball of $\ck_{k}$ (so $\n{J_{i}^{(k)}}_{op}\leq 1$ where $\n{\cdot}_{op}$ is the operator norm). For $\gamma^{(k)}$, $\wt{\gamma}^{(k)}\in \cl_{k}^{1}$, we then define a metric $d_{k}$ on $\cl_{k}^{1}$ by
\begin{align*}
d_{k}(\gamma^{(k)},\wt{\gamma}^{(k)})=\sum_{i=1}^{\infty}2^{-i}\bbabs{Tr J_{i}^{(k)}\lrs{\gamma^{(k)}-\wt{\gamma}^{(k)}}}.
\end{align*}
A uniformly bounded sequence $\wt{\gamma}^{(k)}_{N,L}\in \cl_{k}^{1}$ converges to $\wt{\gamma}^{(k)}$ with respect to the weak* topology if and only if
$$\lim_{N,1/L\to \infty}d_{k}(\wt{\gamma}_{N,L}^{(k)},\wt{\gamma}^{(k)})=0.$$
For fixed $T>0,$ let $C([0,T];\cl_{k}^{1})$ be the space of functions of $t\in [0,T]$ with values in $\cl_{k}^{1}$ that are continuous with respect to the metric $d_{k}$. On $C([0,T];\cl_{k}^{1})$, we define the metric
\begin{align}\label{eq:compactness of the BBGKY metrics}
\hat{d}_{k}(\gamma^{(k)}(\cdot),\wt{\gamma}^{(k)}(\cdot))=\sup_{t\in [0,T]}d_{k}(\gamma^{(k)}(t),\wt{\gamma}^{(k)}(t)),
\end{align}
and denote by $\tau_{prod}$ the topology on the space $\bigotimes_{k\geq 1}C([0,T];\cl_{k}^{1})$ given by the product of topologies generated by the metrics $\hat{d}_{k}$ on $C([0,T],\cl_{k}^{1})$.

\subsection{Compactness of the BBGKY sequence} \label{section 3.1}
\begin{theorem} \label{thm:compactness of BBGKY}
Assume $L(N/L)^{\beta}\to 1^{-}$. Then the sequence
\begin{align}
\lr{\Gamma_{N,L}(t)=\lr{\wt{\gamma}_{N,L}^{(k)}}_{k=1}^{N}}\subset \bigotimes_{k\geq 1}C\lrs{[0,T];\cl_{k}^{1}},
\end{align}
which satisfies the BBGKY hierarchy, is compact with respect to the product topology $\tau_{prod}.$ For any limit point $\Gamma(t)=\lr{\wt{\gamma}^{(k)}(t)}_{k=1}^{\infty},$ we have $\wt{\gamma}^{(k)}$ is a symmetric nonnegative trace class operator with trace bounded by 1.
\end{theorem}
\begin{proof}
By the standard diagonalization argument, it suffices to show the compactness of $\wt{\gamma}_{N,L}^{(k)}$ for fixed $k$ with respect to the metric $\hat{d}_{k}.$ By the Arzel$\grave{a}$-Ascoli theorem, this is equivalent to the equicontinuity of $\gamma_{N,L}^{(k)}$, and by, this is equivalent to the statement that for every observable $J^{(k)}$ from a dense subset of $\ck_{k}$ and for every $\varepsilon>0$, there exists $\delta(J^{(k)},\varepsilon)$ such that for all $t_{1}$, $t_{2}\in [0,T]$ with $|t_{1}-t_{2}|\leq \delta,$ we have
\begin{align} \label{equ:compactness of the BBGKY sequence equicontinuity}
\sup_{N,L}\bbabs{Tr J^{(k)}\wt{\gamma}_{N,L}^{(k)}(t_{1})-Tr J^{(k)}\wt{\gamma}_{N,L}^{(k)}(t_{2})}\leq \varepsilon.
\end{align}

We assume that compact operators $J^{(k)}$ have been cutoff in Lemma $\ref{lemma:Appendix cut-off operator}$. Since the observable $J^{(k)}$ can be written as a sum of a self-adjoint operator and an anti-self-adjoint operator, we may assume $J^{(k)}$ is self-adjoint. Inserting the decomposition ($\ref{equ:outline of proof projection operator I}$) on the left and right sides of $\wt{\gamma}_{N,L}^{(k)}$, we obtain
$$\wt{\gamma}_{N,L}^{(k)}=\sum_{\alpha,\beta}\wt{\cp}_{\alpha}\wt{\gamma}_{N,L}^{(k)} \wt{\cp}_{\beta},$$
where the sum is taken over all $k$-tuples $\alpha$ and $\beta$.

To establish $(\ref{equ:compactness of the BBGKY sequence equicontinuity})$, it suffices to prove that, for each $\alpha$ and $\beta$, we have
\begin{align} \label{equ:compactness of the BBGKY sequence equicontinuity alpha and beta}
\sup_{N,L}\bbabs{Tr J^{(k)}\wt{\cp}_{\alpha}\wt{\gamma}_{N,L}^{(k)}\wt{\cp}_{\beta}(t_{1})-Tr J^{(k)}\wt{\cp}_{\alpha}\wt{\gamma}_{N,L}^{(k)}\wt{\cp}_{\beta}(t_{2})}\leq \varepsilon.
\end{align}
To this end, we establish the estimate
\begin{align}
&\babs{Tr J^{(k)}\wt{\cp}_{\alpha}\wt{\gamma}_{N,L}^{(k)}\wt{\cp}_{\beta}(t_{1})-Tr J^{(k)}\wt{\cp}_{\alpha}\wt{\gamma}_{N,L}^{(k)}\wt{\cp}_{\beta}(t_{2})}
\label{equ:compactness of the BBGKY sequence equicontinuity the key estimate}\\
\leq& C|t_{2}-t_{1}|\lrs{1_{\lr{\alpha=0,\beta=0}}+\max\lr{1,L^{|\alpha|+|\beta|-2}}1_{\lr{\alpha\neq 0\ or\ \beta\neq 0}}}\notag.
\end{align}
By $(\ref{equ:compactness of the BBGKY sequence equicontinuity the key estimate})$, we can directly establish ($\ref{equ:compactness of the BBGKY sequence equicontinuity alpha and beta}$) except for the case $|\alpha|+|\beta|=1.$ However, from Corollary $\ref{High Energy estimates when k>1:energy bound}$, we can also get a bound
\begin{align}\label{equ:compactness of the BBGKY sequence the upper bound}
 &\babs{Tr J^{(k)}\wt{\cp}_{\alpha}\wt{\gamma}_{N,L}^{(k)}\wt{\cp}_{\beta}(t_{1})-Tr J^{(k)}\wt{\cp}_{\alpha}\wt{\gamma}_{N,L}^{(k)}\wt{\cp}_{\beta}(t_{2})}\\
\leq &2\sup_{t}|\lra{J^{(k)}\wt{\cp}_{\alpha}\wt{\psi}_{N,L}(t),\wt{\cp}_{\beta}\wt{\psi}_{N,L}(t)}|\notag\\
\leq&2\n{J^{(k)}}_{op}\n{\wt{\cp}_{\alpha}\wt{\psi}_{N,L}(t)}_{L^{2}}\n{\wt{\cp}_{\beta}\wt{\psi}_{N,L}(t)}_{L^{2}}\notag\\
\lesssim& L^{|\alpha|+|\beta|}.\notag
\end{align}
By averaging $(\ref{equ:compactness of the BBGKY sequence equicontinuity the key estimate})$ and ($\ref{equ:compactness of the BBGKY sequence the upper bound}$) in the case $|\alpha|+|\beta|=1$, we obtain
\begin{align*}
\babs{Tr J^{(k)}\wt{\cp}_{\alpha}\wt{\gamma}_{N,L}^{(k)}\wt{\cp}_{\beta}(t_{1})-Tr J^{(k)}\wt{\cp}_{\alpha}\wt{\gamma}_{N,L}^{(k)}\wt{\cp}_{\beta}(t_{2})}\lesssim |t_{2}-t_{1}|^{1/2},
\end{align*}
which suffices to establish ($\ref{equ:compactness of the BBGKY sequence equicontinuity alpha and beta}$).

Thus, we are left to prove $(\ref{equ:compactness of the BBGKY sequence equicontinuity the key estimate})$
The BBGKY hierarchy $(\ref{equ:BBGKY hierarchy rescaled})$ yields
\begin{align} \label{equ:compactness of the BBGKY sequence differential form}
\pa_{t}Tr J^{(k)}\wt{\cp}_{\alpha}\wt{\gamma}_{N,L}^{(k)}\wt{\cp}_{\beta}=I+II+III+IV,
\end{align}
where
\begin{align*}
&I=-i\sum_{j=1}^{k}Tr J^{(k)}\lrc{-\Delta_{x_{j}},\wt{\cp}_{\alpha}\wt{\gamma}_{N,L}^{(k)}\wt{\cp}_{\beta}},\\
&II=-i\frac{1}{L^{2}}\sum_{j=1}^{k}Tr J^{(k)}\lrc{-\pa_{z_{j}}^{2},\wt{\cp}_{\alpha}\wt{\gamma}_{N,L}^{(k)}\wt{\cp}_{\beta}},\\
&III=-i\frac{1}{N-1}\sum_{1\leq i<j\leq k}Tr J^{(k)}\wt{\cp}_{\alpha}\lrc{\wt{V}_{N,L}(r_{i}-r_{j}),\wt{\gamma}_{N,L}^{(k)}}\wt{\cp}_{\beta},\\
&IV=-i\frac{N-k}{N-1}\sum_{j=1}^{k}Tr J^{(k)}Tr_{r_{k+1}}\wt{\cp}_{\alpha}\lrc{\wt{V}_{N,L}(r_{j}-r_{k+1}),
\wt{\gamma}_{N,L}^{(k+1)}}\wt{\cp}_{\beta}.
\end{align*}

First, we handle $I$. By Lemma $\ref{lemma:Appendix kernel and trace}$ and integration by parts, we have
\begin{align*}
I=&i\sum_{j=1}^{k}\lrs{\lra{J^{(k)}\Delta_{x_{j}}\wt{\cp}_{\alpha}\wt{\psi}_{N,L},\wt{\cp}_{\beta}\wt{\psi}_{N,L}}-
\lra{J^{(k)}\wt{\cp}_{\alpha}\wt{\psi}_{N,L},\wt{\cp}_{\beta}\Delta_{x_{j}}\wt{\psi}_{N,L}}}\\
=&i\sum_{j=1}^{k}\lrs{\lra{J^{(k)}\Delta_{x_{j}}\wt{\cp}_{\alpha}\wt{\psi}_{N,L},\wt{\cp}_{\beta}\wt{\psi}_{N,L}}-
\lra{\Delta_{x_{j}}J^{(k)}\wt{\cp}_{\alpha}\wt{\psi}_{N,L},\wt{\cp}_{\beta}\wt{\psi}_{N,L}}}.
\end{align*}
Hence
\begin{align} \label{equ:compactness of the BBGKY sequence estimate I}
|I| \leq\sum_{j=1}^{k}\lrs{\n{J^{(k)}\Delta_{x_{j}}}_{op}+\n{\Delta_{x_{j}}J^{(k)}}_{op}}
\n{\wt{\cp}_{\alpha}\wt{\psi}_{N,L}}_{L^{2}}\n{\wt{\cp}_{\beta}\wt{\psi}_{N,L}}_{L^{2}}\leq C_{k,J^{(k)}},
\end{align}
where in the last step we used the energy estimate.

Next, we consider $II$. When $\alpha=\beta=0,$ we have
\begin{align*}
II=-i\frac{1}{L^{2}}\sum_{j=1}^{k}Tr J^{(k)}\lrc{-\pa_{z_{j}}^{2}-1,\wt{\cp}_{\mathbf{0}}\wt{\gamma}_{N,L}^{(k)}\wt{\cp}_{\mathbf{0}}}=0,
\end{align*}
where we used $[1,\wt{\cp}_{\alpha}\wt{\gamma}_{N,L}^{(k)}\wt{\cp}_{\beta}]=0$ in the first equality.

When $|\alpha|+|\beta|\geq 1$, applying Lemma $\ref{lemma:Appendix kernel and trace}$ and integration by parts again, we have
\begin{align*}
II= &i\frac{1}{L^{2}}\sum_{j=1}^{k}\lrs{\lra{J^{(k)}\pa_{z_{j}}^{2}\wt{\cp}_{\alpha}\wt{\psi}_{N,L},\wt{\cp}_{\beta}\wt{\psi}_{N,L}}-
\lra{J^{(k)}\wt{\cp}_{\alpha}\wt{\psi}_{N,L},\pa_{z_{j}}^{2}\wt{\cp}_{\beta}\wt{\psi}_{N,L}}}\\
=&i\frac{1}{L^{2}}\sum_{j=1}^{k}\lrs{\lra{J^{(k)}\pa_{z_{j}}^{2}\wt{\cp}_{\alpha}\wt{\psi}_{N,L},\wt{\cp}_{\beta}\wt{\psi}_{N,L}}-
\lra{\pa_{z_{j}}^{2}J^{(k)}\wt{\cp}_{\alpha}\wt{\psi}_{N,L},\wt{\cp}_{\beta}\wt{\psi}_{N,L}}}.
\end{align*}
Hence
\begin{align*}
|II|\leq&\frac{1}{L^{2}}\sum_{j=1}^{k}\lrs{\n{J^{(k)}\pa_{z_{j}}^{2}}_{op}+\n{\pa_{z_{j}}^{2}J^{(k)}}_{op}}\n{\wt{\cp}_{\alpha}\wt{\psi}_{N,L}}_{L^{2}}
\n{\wt{\cp}_{\beta}\wt{\psi}_{N,L}}_{L^{2}}.
\end{align*}

By the energy estimate $(\ref{equ:High Energy estimates when k>1:energy bound 3})$,
\begin{align}
&|II|=0, &|\alpha|+|\beta|=0, \label{equ:compactness of the BBGKY sequence estimate II1}\\
&|II|\lesssim C_{k,J^{(k)}} L^{|\alpha|+|\beta|-2},&|\alpha|+|\beta|\geq 1. \label{equ:compactness of the BBGKY sequence estimate II2}
\end{align}

Next, we consider $III.$ Similarly,
\begin{align*}
III=& \frac{-i}{N-1}\sum_{1\leq i<j\leq k}\lra{J^{(k)}\wt{\cp}_{\alpha}\wt{V}_{N,L}(r_{i}-r_{j})\wt{\psi}_{N,L},\wt{\cp}_{\beta}\wt{\psi}_{N,L}}\\
&+\frac{i}{N-1}\sum_{1\leq i<j\leq k}\lra{J^{(k)}\wt{\cp}_{\alpha}\wt{\psi}_{N,L},\wt{\cp}_{\beta}\wt{V}_{N,L}(r_{i}-r_{j})\wt{\psi}_{N,L}}\\
=& \frac{-i}{N-1}\sum_{1\leq i<j\leq k}\lra{J^{(k)}\wt{\cp}_{\alpha}\wt{V}_{N,L}(r_{i}-r_{j})\wt{\psi}_{N,L},\wt{\cp}_{\beta}\wt{\psi}_{N,L}}\\
&+\frac{i}{N-1}\sum_{1\leq i<j\leq k}\lra{\wt{\cp}_{\alpha}\wt{\psi}_{N,L},J^{(k)}\wt{\cp}_{\beta}\wt{V}_{N,L}(r_{i}-r_{j})\wt{\psi}_{N,L}}.
\end{align*}
Let
$$W_{ij}=\lra{\nabla_{r_{i}}}^{-1}\lra{\nabla_{r_{j}}}^{-1}\wt{V}_{N,L}(r_{i}-r_{j})\lra{\nabla_{r_{i}}}^{-1}\lra{\nabla_{r_{j}}}^{-1}.$$
Hence
\begin{align*}
|III|\lesssim& N^{-1}\sum_{1\leq i<j\leq k}\n{J^{(k)}\lra{\nabla_{r_{i}}}\lra{\nabla_{r_{j}}}}_{op}\n{W_{ij}}_{op}\n{\lra{\nabla_{r_{i}}}\lra{\nabla_{r_{j}}}\wt{\psi}_{N,L}}_{L^{2}}\n{\wt{\cp}_{\beta}\wt{\psi}_{N,L}}_{L^{2}}\\
+&N^{-1}\sum_{1\leq i<j\leq k}\n{\lra{\nabla_{r_{i}}}\lra{\nabla_{r_{j}}}J^{(k)}}_{op}\n{W_{ij}}_{op}\n{\lra{\nabla_{r_{i}}}\lra{\nabla_{r_{j}}}\wt{\psi}_{N,L}}_{L^{2}}\n{\wt{\cp}_{\alpha}\wt{\psi}_{N,L}}_{L^{2}}.
\end{align*}
Since $\n{W_{ij}}_{op}\lesssim \n{\wt{V}_{N,L}}_{L_{z}^{\infty}L_{x}^{1}}\leq \n{V}_{L_{z}^{\infty}L_{x}^{1}}$ by $(\ref{equ:operator norm estimate})$, the energy estimates $(\ref{equ:High Energy estimates when k>1:energy bound 2})$ $(\ref{equ:High Energy estimates when k>1:energy bound 3})$ imply that
\begin{align} \label{equ:compactness of the BBGKY sequence estimate III}
|III|\lesssim \frac{C_{k,J^{(k)}}}{N}.
\end{align}

For $IV$, we have
\begin{align*}
IV=&-i\frac{N-k}{N-1}\sum_{j=1}^{k}\lra{J^{(k)}\wt{\cp}_{\alpha}\wt{V}_{N,L}(r_{j}-r_{k+1})\psi_{N,L},\wt{\cp}_{\beta}\psi_{N,L}}\\
&+i\frac{N-k}{N-1}\sum_{j=1}^{k}\lra{J^{(k)}\wt{\cp}_{\alpha}\wt{V}_{N,L}(r_{j}-r_{k+1})\psi_{N,L},\wt{\cp}_{\beta}\psi_{N,L}}.
\end{align*}
Then, since $J^{(k)}\lra{\nabla_{r_{k+1}}}=\lra{\nabla_{r_{k+1}}}J^{(k)},$
\begin{align*}
IV=&-i\frac{N-k}{N-1}\sum_{j=1}^{k}\lra{J^{(k)}\wt{\cp}_{\alpha}\lra{\nabla_{r_{j}}}W_{j(k+1)}\lra{\nabla_{r_{j}}}\lra{\nabla_{r_{k+1}}}\wt{\psi}_{N,L},\wt{\cp}_{\beta}\lra{\nabla_{r_{k+1}}}\wt{\psi}_{N,L}}\\
&+ i\frac{N-k}{N-1}\sum_{j=1}^{k}\lra{\lra{\nabla_{r_{j}}}J^{(k)}\wt{\cp}_{\alpha}\lra{\nabla_{r_{k+1}}}\wt{\psi}_{N,L},\wt{\cp}_{\beta}W_{j(k+1)}\lra{\nabla_{r_{j}}}\lra{\nabla_{r_{k+1}}}\wt{\psi}_{N,L}}.
\end{align*}
Hence
\begin{align*}
|IV|\lesssim& \sum_{j=1}^{k}\lrs{\n{J^{(k)}\lra{\nabla_{r_{j}}}}_{op}+\n{\lra{\nabla_{r_{j}}}J^{(k)}}_{op}}\n{W_{j(k+1)}}_{op}\n{\lra{\nabla_{r_{j}}}\lra{\nabla_{r_{k+1}}}\wt{\psi}_{N,L}}_{L^{2}}
\n{\lra{\nabla_{r_{k+1}}}\wt{\psi}_{N,L}}_{L^{2}}.
\end{align*}
By energy estimate ($\ref{equ:High Energy estimates when k>1:energy bound 2}$),
\begin{align} \label{equ:compactness of the BBGKY sequence estimate IV}
|IV|\lesssim C_{k,J^{(k)}}.
\end{align}
Integrating ($\ref{equ:compactness of the BBGKY sequence differential form}$) from $t_{1}$ to $t_{2}$ and putting ($\ref{equ:compactness of the BBGKY sequence estimate I}$), ($\ref{equ:compactness of the BBGKY sequence estimate II1}$), ($\ref{equ:compactness of the BBGKY sequence estimate II2}$), ($\ref{equ:compactness of the BBGKY sequence estimate III}$) and ($\ref{equ:compactness of the BBGKY sequence estimate IV}$) together, we obtain ($\ref{equ:compactness of the BBGKY sequence equicontinuity the key estimate}$).

\end{proof}

\begin{corollary} \label{cor:Compactness of the BBGKY sequence,limit points variable-separated form}
Let $\Gamma(t)=\lr{\wt{\gamma}^{(k)}}_{k=1}^{\infty}$ be a limit point of $\lr{\Gamma_{N,L}(t)=\lr{\wt{\gamma}_{N,L}^{(k)}}_{k=1}^{N}}$, with respect to the product topology $\tau_{prod}$. Then $\wt{\gamma}^{(k)}$ satisfies the priori bound
\begin{align} \label{equ:Compactness of the BBGKY sequence,limit points variable-separated form energy estimate}
Tr \lra{\nabla}^{(k)}\wt{\gamma}^{(k)}\lra{\nabla}^{(k)}\leq C^{k},
\end{align}
and takes the structure
\begin{align}\label{equ:Compactness of the BBGKY sequence,limit points variable-separated form}
\wt{\gamma}^{(k)}(t,(\mathbf{x}_{k},\mathbf{z}_{k});(\hx_{k}',\hz_{k}'))=\wt{\gamma}_{x}^{(k)}(t,\hx_{k};\hx_{k}')
\lrs{\prod_{j=1}^{k}\frac{2}{\pi}\cos(z_{j})\cos(z_{j}')},
\end{align}
where $\wt{\gamma}_{x}^{(k)}=Tr_{z}\wt{\gamma}^{(k)}.$
\end{corollary}
\begin{proof}
The estimate $(\ref{equ:Compactness of the BBGKY sequence,limit points variable-separated form energy estimate})$ follows by $(\ref{equ:High Energy estimates when k>1:energy bound 2})$ in Corollary $\ref{High Energy estimates when k>1:energy bound}$ and Theorem $\ref{thm:compactness of BBGKY}$. To establish the formula $(\ref{equ:Compactness of the BBGKY sequence,limit points variable-separated form})$, it suffices to prove $\wt{\cp}_{\alpha}\wt{\gamma}^{(k)}\wt{\cp}_{\beta}=0$ if either $\alpha\neq 0$ or $\beta\neq 0.$ This is equivalent to the statement that for any $J^{(k)}\in \ck_{k}$, $Tr J^{(k)}\wt{\cp}_{\alpha}\wt{\gamma}^{(k)}\wt{\cp}_{\beta}=0.$
By Corollary $(\ref{High Energy estimates when k>1:energy bound})$, we obtain
\begin{equation}\label{equ:Compactness of the BBGKY sequence,limit points variable-separated form limit point}
Tr J^{(k)}\wt{\cp}_{\alpha}\wt{\gamma}^{(k)}\wt{\cp}_{\beta}=\lim_{\substack{N,1/L\to \infty\\L(N/L)^{\beta}\to 1^{-}}} Tr J^{(k)}\wt{\cp}_{\alpha}\wt{\gamma}_{N,L}^{(k)}\wt{\cp}_{\beta}.
\end{equation}
By Lemma $\ref{lemma:Appendix kernel and trace}$,
\begin{align*}
Tr J^{(k)}\wt{\cp}_{\alpha}\wt{\gamma}_{N,L}^{(k)}\wt{\cp}_{\beta}=\lra{J^{(k)}\cp_{\alpha}\wt{\psi}_{N,L},\wt{\cp}_{\beta}\wt{\psi}_{N,L}},
\end{align*}
and by Cauchy-Schwarz and $(\ref{equ:High Energy estimates when k>1:energy bound 3})$,
\begin{align*}
\bbabs{Tr J^{(k)}\wt{\cp}_{\alpha}\wt{\gamma}_{N,L}\wt{\cp}_{\beta}}\leq& \n{J^{(k)}}_{op}\n{\wt{\cp}_{\alpha}\wt{\psi}_{N,L}}_{L^{2}}\n{\wt{\cp}_{\beta}\wt{\psi}_{N,L}}_{L^{2}}\\
\leq& C^{k}L^{|\alpha|+|\beta|}.
\end{align*}
Hence the right side of $(\ref{equ:Compactness of the BBGKY sequence,limit points variable-separated form limit point})$ is $0.$
\end{proof}

\begin{theorem}\label{thm:compactness of the BBGKY sequence, 2D case}
The sequence
\begin{align*}
\Gamma_{x,N,L}=\lr{\wt{\gamma}_{x,N,L}=Tr_{z}\wt{\gamma}_{N,L}^{(k)}}_{k=1}^{N}\in \bigotimes_{k\geq 1}C\lrs{[0,T];\cl_{k}^{1}(\T^{2k})}
\end{align*}
is compact with respect to the $2D$ version of the product topology $\tau_{prod}$ used in Theorem $\ref{thm:compactness of BBGKY}$.
\end{theorem}
\begin{proof}
The proof is similar to the 3D case and we omit it. Also see \cite[Theorem 5]{chen2013rigorous2dfrom3d}.
\end{proof}

\subsection{Limit points satisfy GP hierarchy} \label{section 3.2}
To prove the limit points satisfy the GP hierarchy, a technical tool we need is the approximation of identity type lemma, which is used to compare the $\delta$-function and its approximation. Since we request $L(N/L)^{\beta}\to 1^{-}$, we see that $\wt{V}_{N,L}(x,z)$ defined by $(\ref{equ:introduction inteaction rescaled})$ formally converges to $\delta(x)\int V(x,z)dx $. Thus, we need a modified version of this type lemma as follows.
\begin{lemma}\label{lemma:convergence indentity approximation}
Let $\rho\in L^{1}(\Om)$ be a function compactly supported on $\Om$ such that
$$\sup_{z}\int |\rho(x,z)||x|dx<\infty$$
and define $\rho_{\varepsilon,\lambda}(x,z)=\varepsilon^{-2}\lambda^{-1}\rho(x/\varepsilon,z/\lambda)$ and $g(z)=\int \rho(x,z)dx$. Then, for every $\kappa\in [0,1)$, there exists $C>0$, such that
\begin{align*}
&\bbabs{TrJ^{(k)}(\rho_{\varepsilon,\lambda}(r_{j}-r_{k+1})-\delta(x_{j}-x_{k+1})g(z_{j}-z_{k+1}))\gamma^{(k+1)}}\\
\leq& C\varepsilon^{\kappa}\sup_{z}\int |\rho(x,z)||x|^{\kappa}dx\lrs{\n{\lra{\nabla_{x_{j}}}^{-1}J^{(k)}\lra{\nabla_{x_{j}}}}_{op}+\n{\lra{\nabla_{x_{j}}}J^{(k)}\lra{\nabla_{x_{j}}}^{-1}}_{op}}
Tr \lra{\nabla_{x_{j}}}^{2}\lra{\nabla_{x_{k+1}}}^{2}\gamma^{(k+1)}\\
&+C_{J}\n{\rho_{1,\lambda}-\rho_{1,1}}_{L^{1}}Tr \lra{\nabla_{r_{j}}}^{2}\lra{\nabla_{r_{k+1}}}^{2}\gamma^{(k+1)}
\end{align*}
for all nonnegative $\gamma^{(k+1)}\in \cl_{k+1}^{1}.$
\end{lemma}
\begin{proof}
We will give a proof for Lemma $\ref{lemma:convergence indentity approximation}$ in the Appendix.
\end{proof}
\begin{theorem} \label{thm:limit points satisfy GP hierarchy, 2d coupled focusing GP}
Assume $L(N/L)^{\beta}\to 1^{-},$ and let $\Gamma(t)=\lr{\wt{\gamma}^{(k)}}_{k=1}^{\infty}$ be a limit point of
$$\lr{\Gamma_{N,L}(t)=\lr{\wt{\gamma}_{N,L}^{(k)}(t)}_{k=1}^{N}}$$
 with respect to the product topology $\tau_{prod}.$ Then $\lr{\wt{\gamma}_{x}^{(k)}=Tr_{z}\wt{\gamma}^{(k)}}_{k=1}^{\infty}$ is a solution to the coupled focusing Gross-Pitaevskii hierarchy subject to initial data $\wt{\gamma}_{x}^{(k)}(0)=|\phi_{0}\rangle \langle \phi_{0}|^{\otimes k}$, which, rewritten in integral form, is
\begin{align} \label{equ:limit points satisfy GP hierarchy, 2d coupled focusing GP}
\wt{\gamma}_{x}^{(k)}=&U^{(k)}(t)\wt{\gamma}_{x}^{(k)}(0) -i\sum_{j=1}^{k}\int_{0}^{t}U^{(k)}(t-s)Tr_{x_{k+1}}Tr_{z}\lrc{\delta(x_{j}-x_{k+1})\int V(x,z_{j}-z_{k+1})dx,\wt{\gamma}^{(k+1)}(s)}ds,
\end{align}
where
$U^{(k)}(t)=\prod_{j=1}^{k}e^{it \Delta_{x_{j}}}e^{-it \Delta_{x_{j}'}}.$
\end{theorem}

\begin{proof}
Passing to subsequences if necessary, we have
\begin{align}\label{equ:limit point GP hierarchy}
\lim_{\substack{N,1/L\to \infty\\L(N/L)^{\beta}\to 1^{-}}}\sup_{t}Tr J^{(k)}\lrs{\wt{\gamma}_{N,L}^{(k)}-\wt{\gamma}^{(k)}(t)}=0,\quad
\forall J^{(k)}\in \ck(L^{2}(\Om^{k}))
\end{align}
\begin{align}\label{equ:limit point GP hierarchy trace z}
\lim_{\substack{N,1/L\to \infty\\L(N/L)^{\beta}\to 1^{-}}}\sup_{t}Tr J_{x}^{(k)}\lrs{\wt{\gamma}_{x,N,L}^{(k)}-\wt{\gamma}_{x}^{(k)}(t)}=0,\quad
\forall J_{x}^{(k)}\in \ck(L^{2}(\T^{2k}))
\end{align}
from Theorem $\ref{thm:compactness of BBGKY}$ and $\ref{thm:compactness of the BBGKY sequence, 2D case}$.

It suffices to test the limit point against the test function $J_{x}^{(k)}\in \ck(L^{2}(\T^{2k})).$ We will prove that the limit point satisfies
\begin{align} \label{equ:limit points satisfy GP hierarchy, 2d coupled focusing GP,initial condition}
Tr J_{x}^{(k)}\wt{\gamma}_{x}^{(k)}(0)=Tr J_{x}^{(k)}|\phi_{0}\rangle \langle \phi_{0}|^{\otimes k},
\end{align}
and
\begin{align} \label{equ:limit points satisfy GP hierarchy, 2d coupled focusing GP,BBGKY}
Tr J_{x}^{(k)}\wt{\gamma}_{x}^{(k)}(t)=&Tr J_{x}^{(k)}U^{(k)}(t)\wt{\gamma}_{x}^{(k)}(0)\\
&+i\sum_{j=1}^{k}\int_{0}^{t}U^{(k)}(t-s)Tr_{x_{k+1}}Tr_{z}\lrc{\delta(x_{j}-x_{k+1})g(z_{j}-z_{k+1}),\wt{\gamma}^{(k+1)}(s)}ds,\notag
\end{align}
where we adopt the notation $g(z)=\int V(x,z)dx$ for simplicity.

To end this, we use the coupled focusing BBGKY hierarchy, which is
\begin{align}
Tr J_{x}^{(k)}\wt{\gamma}_{x,N,L}^{(k)}(t)=A-\frac{i}{N-1}\sum_{i<j}^{k}B-i\lrs{1-\frac{k+1}{N-1}}\sum_{j=1}^{k}D,
\end{align}
where
\begin{align*}
&A=Tr J_{x}^{(k)}U^{(k)}(t)\wt{\gamma}_{x,N,L}^{(k)}(0),\\
&B=\int_{0}^{t}Tr J_{x}^{(k)}U^{(k)}(t-s)\lrc{\wt{V}_{N,L}(r_{i}-r_{j}),\wt{\gamma}_{N,L}^{(k)}(s)}ds,\\
&D=\int_{0}^{t}Tr J_{x}^{(k)}U^{(k)}(t-s)\lrc{\wt{V}_{N,L}(r_{j}-r_{k+1}),\wt{\gamma}_{N,L}^{(k+1)}(s)}ds.
\end{align*}
By $(\ref{equ:limit point GP hierarchy trace z})$, we have
\begin{align}
&\lim_{\substack{N,1/L\to \infty\\L(N/L)^{\beta}\to 1^{-}}}Tr J_{x}^{(k)}\wt{\gamma}_{x,N,L}^{(k)}(t)=Tr J_{x}^{(k)}\wt{\gamma}_{x}^{(k)}(t),\\
&\lim_{\substack{N,1/L\to \infty\\L(N/L)^{\beta}\to 1^{-}}}Tr J_{x}^{(k)}U^{(k)}(t)\wt{\gamma}_{x,N,L}^{(k)}(0)=
Tr J_{x}^{(k)}U^{(k)}(t)\wt{\gamma}_{x}^{(k)}(0).
\end{align}
By the argument in \cite{lieb2002proof}, we know, from assumption $(ii)$ in Theorem $\ref{thm:the main theorem}$,
$$\wt{\gamma}_{N,L}^{(1)}(0)\to \frac{2}{\pi}\phi_{0}(x_{1})\overline{\phi_{0}}(x_{1}')\cos(z_{1})\cos(z_{1}'),\quad strongly\ in\ trace\ norm;$$
that is,
$$\wt{\gamma}_{N,L}^{(k)}(0)\to \prod_{j=1}^{k}\frac{2}{\pi}\phi_{0}(x_{j})
\overline{\phi}_{0}(x_{j}')\cos(z_{j})\cos(z_{j}'),\quad strongly\ in\ trace\ norm.$$
Thus we have checked $(\ref{equ:limit points satisfy GP hierarchy, 2d coupled focusing GP,initial condition})$, the left-hand side of $(\ref{equ:limit points satisfy GP hierarchy, 2d coupled focusing GP,BBGKY})$, and the first term on the right-hand side of $(\ref{equ:limit points satisfy GP hierarchy, 2d coupled focusing GP,BBGKY})$ for the limit point. We are left to prove that
\begin{align}
&\lim_{\substack{N,1/L\to \infty\\L(N/L)^{\beta}\to 1^{-}}}\frac{B}{N-1}=0,\\
&\lim_{\substack{N,1/L\to \infty\\L(N/L)^{\beta}\to 1^{-}}}\lrs{1-\frac{k+1}{N-1}}D=\int_{0}^{t}TrJ_{x}^{(k)}U^{(k)}(t-s)\lrc{\delta(x_{j}-x_{k+1})g(z_{j}-z_{k+1}),\wt{\gamma}^{(k+1)}(s)}ds.
\end{align}
First, we will show the boundedness of $|B|$ and $|D|$ for every finite time $t$. Noting that $[U^{(k)},\lra{\nabla_{r_{i}}}]=0,$ we have
\begin{align*}
|B|\leq& \int_{0}^{t}\babs{Tr J_{x}^{(k)}U^{(k)}(t-s)\lrc{\wt{V}_{N,L}(r_{i}-r_{j}),\wt{\gamma}_{N,L}^{(k)}(s)}}ds\\
=&\int_{0}^{t}ds\babs{Tr \lra{\nabla_{r_{i}}}^{-1}\lra{\nabla_{r_{j}}}^{-1}J_{x}^{(k)}\lra{\nabla_{r_{i}}}\lra{\nabla_{r_{j}}}U^{(k)}(t-s)W_{ij}\lra{\nabla_{r_{i}}}\lra{\nabla_{r_{j}}}\wt{\gamma}_{N,L}^{(k)}(s)\lra{\nabla_{r_{i}}}\lra{\nabla_{r_{j}}}\\
&-Tr\lra{\nabla_{r_{i}}}\lra{\nabla_{r_{j}}}J_{x}^{(k)}\lra{\nabla_{r_{i}}}^{-1}\lra{\nabla_{r_{j}}}^{-1}U^{(k)}(t-s)\lra{\nabla_{r_{i}}}\lra{\nabla_{r_{j}}}\wt{\gamma}_{N,L}^{(k)}\lra{\nabla_{r_{i}}}\lra{\nabla_{r_{j}}}W_{ij}}\\
\leq& \int_{0}^{t}ds\n{\lra{\nabla_{r_{i}}}^{-1}\lra{\nabla_{r_{j}}}^{-1}J_{x}^{(k)}\lra{\nabla_{r_{i}}}\lra{\nabla_{r_{j}}}}_{op}\n{U^{(k)}}_{op}\n{W_{ij}}_{op}Tr\lra{\nabla_{r_{i}}}^{2}\lra{\nabla_{r_{j}}}^{2}\wt{\gamma}_{N,L}^{(k)}(s)\\
&+\int_{0}^{t}ds \n{\lra{\nabla_{r_{i}}}\lra{\nabla_{r_{j}}}J_{x}^{(k)}\lra{\nabla_{r_{i}}}^{-1}\lra{\nabla_{r_{j}}}^{-1}}_{op}\n{U^{(k)}}_{op}\n{W_{ij}}_{op}Tr \lra{\nabla_{r_{i}}}^{2}\lra{\nabla_{r_{j}}}^{2}\wt{\gamma}_{N,L}^{(k)}(s)\\
\leq& C_{J}t.
\end{align*}
$|D|$ can be estimated in the same way as $|B|$ and hence
$$|D|\leq C_{J}t.$$
That is,
$$\lim_{\substack{N,1/L\to \infty\\L(N/L)^{\beta}\to 1^{-}}}\frac{B}{N-1}=\lim_{\substack{N,1/L\to \infty\\L(N/L)^{\beta}\to 1^{-}}}\frac{(k+1)D}{N-1}=0.$$

Next, we will use Lemma $\ref{lemma:convergence indentity approximation}$ to prove
\begin{align} \label{equ:limit point GP hierarchy the key limit}
\lim_{\substack{N,1/L\to \infty\\L(N/L)^{\beta}\to 1^{-}}}D=\int_{0}^{t}TrJ_{x}^{(k)}U^{(k)}(t-s)\lrc{\delta(x_{j}-x_{k+1})g(z_{j}-z_{k+1}),\wt{\gamma}^{(k+1)}(s)}ds
\end{align}
Let $\eta\in L^{1}(\T^{2})$ be a smooth probability density function compactly supported on $(-\pi,\pi)^{2}$ and define $\eta_{\varepsilon}(x)=\varepsilon^{-2}\eta(x/\varepsilon)$. For simplicity, we adopt the notation $M_{s,t}^{(k)}=J_{x}^{(k)}U^{(k)}(t-s)$. Then, we expand
\begin{align*}
&Tr J_{x}^{(k)}U^{(k)}(t-s)\lrs{\wt{V}_{N,L}(r_{j}-r_{k+1})\wt{\gamma}_{N,L}^{(k+1)}(s)-\delta(x_{j}-x_{k+1})g(z_{j}-z_{k+1}) \wt{\gamma}^{(k+1)}(s)}\\
=&I+II+III+IV,
\end{align*}
where
\begin{align*}
&I=Tr M_{s,t}^{(k)}\lrs{\wt{V}_{N,L}(r_{j}-r_{k+1})-\delta(x_{j}-x_{k+1})g(z_{j}-z_{k+1})}\wt{\gamma}_{N,L}^{(k+1)}(s),\\
&II=Tr M_{s,t}^{(k)}\lrs{\delta(x_{j}-x_{k+1})-\eta_{\varepsilon}(x_{j}-x_{k+1})}g(z_{j}-z_{k+1})\wt{\gamma}_{N,L}^{(k+1)}(s),\\
&III=Tr M_{s,t}^{(k)}\eta_{\varepsilon}(x_{j}-x_{k+1})g(z_{j}-z_{k+1})\lrs{\wt{\gamma}_{N,L}^{(k+1)}(s)-\wt{\gamma}^{(k+1)}(s)},\\
&IV=Tr M_{s,t}^{(k)}\lrs{\eta_{\varepsilon}(x_{j}-x_{k+1})-\delta(x_{j}-x_{k+1})}g(z_{j}-z_{k+1})\wt{\gamma}^{(k+1)}(s).
\end{align*}
It needs only to prove $I-IV$ converge to $0$ as $N$, $1/L\rightarrow \infty$.
By Lemma $\ref{lemma:convergence indentity approximation}$, we have
\begin{align}
|I|=&\bbabs{ Tr M_{s,t}^{(k)}\lrs{\wt{V}_{N,L}(r_{j}-r_{k+1})-\delta(x_{j}-x_{k+1})g(z_{j}-z_{k+1})}\wt{\gamma}_{N,L}^{(k+1)}(s)}\\
\leq&\frac{C}{(N/L)^{\beta \kappa}} \sup_{z}\int |V(x,z)||x|^{\kappa}dx \lrs{\bbn{\lra{\nabla_{x_{j}}}J_{x}^{(k)}\lra{\nabla_{x_{j}}}^{-1}}_{op}+\bbn{\lra{\nabla_{x_{j}}}^{-1}J_{x}^{(k)}\lra{\nabla_{x_{j}}}}_{op}}\notag\\
&\times Tr \lra{\nabla_{x_{j}}}\lra{\nabla_{x_{k+1}}}\wt{\gamma}_{N,L}^{(k+1)}(s)\lra{\nabla_{x_{k+1}}}\lra{\nabla_{x_{j}}}\notag
+C_{J}\n{\wt{V}_{1,\lambda}-\wt{V}_{1,1}}_{L^{1}}Tr \lra{\nabla_{r_{j}}}^{2}\lra{\nabla_{r_{k+1}}}^{2}\gamma_{N,L}^{(k+1)}\\
\leq& C_{J}\lrc{\frac{\sup_{z}\int |V(x,z)||x|^{\kappa}dx}{(N/L)^{\beta\kappa}}
+\n{L(N/L)^{\beta}\wt{V}(x,L(N/L)^{\beta}z)-\wt{V}(x,z)}_{L^{1}}}.\notag
\end{align}
Similarly, for $II$ and $IV$, via Lemma \ref{lemma:convergence indentity approximation}, we have
\begin{align}
|II|=&\babs{Tr M_{s,t}^{(k)}\lrs{\delta(x_{j}-x_{k+1})-\eta_{\varepsilon}(x_{j}-x_{k+1})}g(z_{j}-z_{k+1})\wt{\gamma}_{N,L}^{(k+1)}(s)}\\
\leq &C\varepsilon^{\kappa} \n{V}_{L_{z}^{\infty}L_{x}^{1}}\int \eta(x)|x|^{\kappa}dx\lrs{\bbn{\lra{\nabla_{x_{j}}}J_{x}^{(k)}\lra{\nabla_{x_{j}}}^{-1}}_{op}+\bbn{\lra{\nabla_{x_{j}}}^{-1}J_{x}^{(k)}\lra{\nabla_{x_{j}}}}_{op}}\notag\\
&\times Tr \lra{\nabla_{x_{j}}}\lra{\nabla_{x_{k+1}}}\wt{\gamma}_{N,L}^{(k+1)}(s)\lra{\nabla_{x_{k+1}}}\lra{\nabla_{x_{j}}}\notag\\
\leq &C_{J} \varepsilon^{\kappa},\notag
\end{align}
where the last inequality follows from Corollary $\ref{High Energy estimates when k>1:energy bound}$, and
\begin{align}
|IV|=&\babs{Tr M_{s,t}^{(k)}\lrs{\eta_{\varepsilon}(x_{j}-x_{k+1})-\delta(x_{j}-x_{k+1})}g(z_{j}-z_{k+1})\wt{\gamma}^{(k+1)}(s)}\\
\leq&C\varepsilon^{\kappa} \n{V}_{L_{z}^{\infty}L_{x}^{1}}\int \eta(x)|x|^{\kappa}dx\lrs{\bbn{\lra{\nabla_{x_{j}}}J_{x}^{(k)}\lra{\nabla_{x_{j}}}^{-1}}_{op}+\bbn{\lra{\nabla_{x_{j}}}^{-1}J_{x}^{(k)}\lra{\nabla_{x_{j}}}}_{op}}\notag\\
&\times Tr \lra{\nabla_{x_{j}}}\lra{\nabla_{x_{k+1}}}\wt{\gamma}^{(k+1)}(s)\lra{\nabla_{x_{k+1}}}\lra{\nabla_{x_{j}}}\notag\\
\leq &C_{J} \varepsilon^{\kappa},\notag
\end{align}
where the last inequality follows from Corollary $\ref{cor:Compactness of the BBGKY sequence,limit points variable-separated form}$. That is,
$$|II|\leq C_{J}\varepsilon^{\kappa}\ and\ |IV|\leq C_{J}\varepsilon^{\kappa}.$$
Hence $II$ and $IV$ converge to $0$ as $\varepsilon\to 0$, uniformly in $N$ and $L$.

For $III$,
\begin{align}
|III|=&\bbabs{Tr M_{s,t}^{(k)}\eta_{\varepsilon}(x_{j}-x_{k+1})g(z_{j}-z_{k+1})\lrs{\wt{\gamma}_{N,L}^{(k+1)}(s)-\wt{\gamma}^{(k+1)}(s)}}\\
\leq& \bbabs{Tr M_{s,t}^{(k)}\eta_{\varepsilon}(x_{j}-x_{k+1})g(z_{j}-z_{k+1})\frac{1}{1+\theta\lra{\nabla_{r_{k+1}}} }\lrs{\wt{\gamma}_{N,L}^{(k+1)}(s)-\wt{\gamma}^{(k+1)}(s)}}\notag\\
&+\bbabs{Tr M_{s,t}^{(k)}\eta_{\varepsilon}(x_{j}-x_{k+1})g(z_{j}-z_{k+1})\frac{\theta \lra{\nabla_{r_{k+1}}}}{1+\theta \lra{\nabla_{r_{k+1}}}}\lrs{\wt{\gamma}_{N,L}^{(k+1)}(s)-\wt{\gamma}^{(k+1)}(s)}}.\notag
\end{align}
The first term in the above estimate goes to zero as $N$, $1/L\to \infty$ for every $\theta>0$, since we have condition ($\ref{equ:limit point GP hierarchy}$) and $M_{s,t}^{(k)}\eta_{\varepsilon}(x_{j}-x_{k+1})g(z_{j}-z_{k+1})\frac{1}{1+\theta \lra{\nabla_{r_{k+1}}}}$ is a compact operator. Due to the energy bounds on $\wt{\gamma}_{N,L}^{(k+1)}$ and $\wt{\gamma}^{(k+1)}$, the second term tends to zero as $\theta \to 0,$ uniformly in $N$ and $L$.

Putting together the estimates for $I-IV$, we have established $(\ref{equ:limit point GP hierarchy the key limit})$. Hence, we have obtained Theorem $\ref{thm:limit points satisfy GP hierarchy, 2d coupled focusing GP}.$
\end{proof}

Combining Corollary $\ref{cor:Compactness of the BBGKY sequence,limit points variable-separated form}$ and Theorem $\ref{thm:limit points satisfy GP hierarchy, 2d coupled focusing GP}$, we see that $\wt{\gamma}_{x}^{(k)}$ solves the $2D$ Gross-Pitaevskii hierarchy with the desired coupling constant $g_{0}=\frac{4}{\pi^{2}}\int \int V(x,z_{1}-z_{2})dx|\cos(z_{1})\cos(z_{2})|^{2}dz_{1}dz_{2}.$
\begin{corollary}\label{thm:limit points satisfy GP hierarchy, 2d focusing GP}
Assume $L(N/L)^{\beta}\to 1^{-},$ and let $\Gamma(t)=\lr{\wt{\gamma}^{(k)}}_{k=1}^{\infty}$ be a limit point of $$\lr{\Gamma_{N,L}(t)=\lr{\wt{\gamma}_{N,L}^{(k)}(t)}_{k=1}^{N}}$$ with respect to the product topology $\tau_{prod}.$ Then $\lr{\wt{\gamma}_{x}^{(k)}=Tr_{z}\wt{\gamma}^{(k)}}_{k=1}^{\infty}$ is a solution to the 2D Gross-Pitaevskii hierarchy subject to initial data $\wt{\gamma}_{x}^{(k)}(0)=|\phi_{0}\rangle \langle \phi_{0}|^{\otimes k}$ with coupling constant
$$g_{0}=\frac{4}{\pi^{2}}\int \int V(x,z_{1}-z_{2})dx|\cos(z_{1})\cos(z_{2})|^{2}dz_{1}dz_{2},$$
 which, rewritten in integral form, is
\begin{align} \label{equ:2D GP hierarchy}
\wt{\gamma}_{x}^{(k)}(t)=&U^{(k)}(t)\wt{\gamma}_{x}^{(k)}(0)-ig_{0}\sum_{j=1}^{k}\int_{0}^{t}U^{(k)}(t-s)Tr_{x_{k+1}}\lrc{\delta(x_{j}-x_{k+1}),\wt{\gamma}_{x}^{(k+1)}(s)}ds.
\end{align}
\end{corollary}
\begin{proof}
The inhomogeneous term in hierarchy $(\ref{equ:limit points satisfy GP hierarchy, 2d coupled focusing GP})$ is
\begin{align*}
&i\int_{0}^{t}U^{(k)}(t-s)Tr_{x_{k+1}}Tr_{z}\lrc{\delta(x_{j}-x_{k+1})\int V(x,z_{j}-z_{k+1})dx,\wt{\gamma}^{(k+1)}(s)}ds\\
=&i\int_{0}^{t}U^{(k)}(t-s)Tr_{x_{k+1}}Tr_{z} \lrs{\delta(x_{j}-x_{k+1})\int V(x,z_{j}-z_{k+1})dx \wt{\gamma}^{(k+1)}(s)}ds\\
&-i\int_{0}^{t}U^{(k)}(t-s)Tr_{x_{k+1}}Tr_{z} \lrs{\delta(x'_{j}-x'_{k+1})\int V(x,z'_{j}-z'_{k+1})dx \wt{\gamma}^{(k+1)}(s)}ds\\
=&I-II.
\end{align*}
From Corollary $\ref{cor:Compactness of the BBGKY sequence,limit points variable-separated form}$, we have
\begin{align*}
I=&i\int_{0}^{t}U^{(k)}(t-s)Tr_{x_{k+1}}\lrs{\delta(x_{j}-x_{k+1}) \wt{\gamma}_{x}^{(k+1)}(s) Tr_{z} \lrs{\int V(x,z_{j}-z_{k+1})dx \prod_{j=1}^{k+1}\frac{2}{\pi}\cos(z_{j})\cos(z_{j}')}}ds\\
=&ig_{0}\int_{0}^{t}U^{(k)}(t-s)Tr_{x_{k+1}}\lrs{\delta(x_{j}-x_{k+1}) \wt{\gamma}_{x}^{(k+1)}(s)}ds.
\end{align*}

In the same manner we can see that
\begin{align*}
II=&ig_{0}\int_{0}^{t}U^{(k)}(t-s)Tr_{x_{k+1}}\lrs{\delta(x_{j}'-x_{k+1}') \wt{\gamma}_{x}^{(k+1)}(s)}ds.
\end{align*}

In summary, we have
\begin{align*}
\wt{\gamma}_{x}^{(k)}(t)=&U^{(k)}(t)\wt{\gamma}_{x}^{(k)}(0)-ig_{0}\sum_{j=1}^{k}\int_{0}^{t}U^{(k)}(t-s)Tr_{x_{k+1}}\lrc{\delta(x_{j}-x_{k+1}),\wt{\gamma}_{x}^{(k+1)}(s)}ds.
\end{align*}
\end{proof}

\subsection{Uniqueness of the 2D GP Hierarchy} \label{section 3.3}
By Bourgain \cite{bourgain1993fourier}, as we are below the Gagliardo-Nirenberg threshold here, we have the $H^{1}$ global wellposedness for the $\T^{2}$ focusing cubic NLS $(\ref{eq:focusing cubic NLS})$. Thus, when $\wt{\gamma}_{x}^{(k)}(0)=|\phi_{0}\rangle \langle \phi_{0}|^{\otimes k}$, we know one solution to the focusing GP hierarchy ($\ref{equ:BBGKY hierarchy rescaled}$), namely $|\phi \rangle \langle \phi |^{\otimes k}$, where $\phi$ solves  $(\ref{eq:focusing cubic NLS})$.

\begin{theorem}
[\cite{herr2016gross,herr2019unconditional,kirkpatrick2011derivation}] \label{thm: uniqueness theorem}
There is at most one nonnegative operator sequence
\begin{align*}
\lr{\gamma_{x}^{(k)}}_{k=1}^{\infty}\in \bigotimes_{k\geq 1}C\lrs{[0,T],\cl_{k}^{1}(\T^{2})}
\end{align*}
that solves the $2D$ Gross-Pitaevskii hierarchy $(\ref{equ:2D GP hierarchy})$ subject to the energy condition
\begin{align}\label{eq:uniqueness energy condition}
Tr\lrs{\prod_{j=1}^{k}(1-\Delta_{x_{j}})}\gamma_{x}^{(k)}\leq C^{k}.
\end{align}
\end{theorem}

From Theorem $\ref{thm: uniqueness theorem}$, we conclude that the compact sequence
$$\lr{\Gamma_{N,L}(t)=\lr{\wt{\gamma}_{N,L}^{(k)}}_{k=1}^{N}}$$
has only one $L(N/L)^{\beta}\to 1^{-}$ limit point, namely
$$\wt{\gamma}^{(k)}=\prod_{j=1}^{k}\frac{2}{\pi}\phi(t,x_{j})\overline{\phi}(t,x_{j}')\cos(z_{j})\cos(z_{j}').$$
We then infer that as trace class operators
$$\wt{\gamma}_{N,L}^{(k)}\rightarrow \prod_{j=1}^{k}\frac{2}{\pi}\phi(t,x_{j})\overline{\phi}(t,x_{j}')\cos(z_{j})\cos(z_{j}')\ weak^*.$$
 Since the limit point $\wt{\gamma}^{(k)}$ is an orthogonal projection, the well-known argument in \cite[$P_{296}$]{erdos2010derivation} upgrades the weak* convergence to strong, by using Gr$\ddot{u}$mm's convergence theorem \cite[Theorem\ 2.19]{simon2010trace}.\\

\textbf{Acknowledgements} The author would like to thank his PhD advisors Xuwen Chen and Zhifei Zhang. The author received partial support from the Graduate School of Peking University.

\appendix
\section{Basic operator facts and Sobolev-type lemmas}
\begin{lemma}[Hoffman-Ostenhof inequality\cite{hoffmann1977schrodinger}]\label{lemma:Hoffman-Ostenhof inequality appendix}
For $\psi_{N}\in H^{1}(\T^{dN})$, we have
\begin{align}\label{equ:Hoffman-Ostenhof inequality appendix}
\n{\nabla_{x}\sqrt{\rho_{N}}}_{L^{2}(\T^{d})}\leq \n{\nabla_{x_{1}} \psi_{N}}_{L^{2}(\T^{dN})}.
\end{align}
with the one-particle density
\begin{align*}
\rho_{N}(x)=\int_{\T^{d}}\ccc\int_{\T^{d}}|\psi_{N}(x,x_{2},...,x_{N})|^{2}dx_{2}\ccc dx_{N}.
\end{align*}
\end{lemma}
\begin{proof}
We may assume $\psi_{N}$ is a test function, So $\rho_{N}(x)\in C^{\infty}(\T^{d}).$ By Cauchy-Schwarz inequality, we have
\begin{align*}
\n{\nabla_{x}\sqrt{\rho_{N}+\varepsilon}}_{L^{2}}^{2}=&\int \bbabs{\frac{\nabla_{x}\rho_{N}}{2\sqrt{\rho_{N}+\varepsilon}}}^{2}dx\\
=&\int \frac{\abs{\int \psi_{N}\nabla_{x}\psi_{N}dx_{2}\ccc dx_{N}}^{2}}{\rho_{N}(x)+\varepsilon}dx\\
\leq& \int  |\nabla_{x}\psi_{N}|^{2}dx_{2}\ccc dx_{N} \frac{\rho_{N}(x)}{\rho_{N}(x)+\varepsilon}dx\\
\leq& \int  |\nabla_{x} \psi_{N}|^{2}dxdx_{2}\ccc dx_{N}.
\end{align*}
Thus, $\sqrt{\rho_{N}+\varepsilon}$ is uniform bounded in $H^{1}(\T^{d})$ norm. We note that $-\Delta_{x}$ is a self-adjoint operator on $H^{1}(\T^{d})$ and $\sqrt{\rho_{N}+\varepsilon}$ converges to $\sqrt{\rho_{N}}$ in $L^{2}(\T^{d})$ norm as $\varepsilon\to 0$. From the definition of adjoint operator, we deduce that $\sqrt{\rho_{N}}\in H^{1}(\T^{d})$ and
$$\n{\nabla_{x}\sqrt{\rho_{N}}}_{L^{2}}^{2}\leq \int  |\nabla_{x} \psi_{N}|^{2}dxdx_{2}\ccc dx_{N}.$$

\end{proof}

\begin{lemma}\label{lemma:appendix basic operator estimate}
Assume $L(N/L)^{\beta}\to 1^{-}$. For $\delta\in (0,1)$, the multiplication operator $V_{N,L}(r_{1}-r_{2})$ on $L^{2}(\Om_{L}^{\otimes 2})$ satisfies
\begin{align}
&L|V_{N,L}(r_{1}-r_{2})|\leq C_{\delta}(N/L)^{\delta}\n{V}_{L_{z}^{\infty}L_{x}^{1+\delta}}(1-\Delta_{x_{1}}),\label{equ:appendix basic operator estimate: one body estimate}\\
&L|V_{N,L}(r_{1}-r_{2})|\leq C_{\delta}\n{V}_{L_{z}^{\infty}L_{x}^{1}}(1-\Delta_{x_{1}})^{1/2+\delta}(1-\Delta_{x_{2}})^{1/2+\delta},\label{equ:appendix basic operator estimate}\\
&\n{\lra{\nabla_{x_{1}}}^{-1}\lra{\nabla_{x_{2}}}^{-1}V(r_{1}-r_{2})\lra{\nabla_{x_{1}}}^{-1}\lra{\nabla_{x_{2}}}^{-1}}_{op}\leq C\n{V}_{L_{z}^{\infty}L_{x}^{1}},\label{equ:appendix operator norm estimate}\\
&S_{1}^{2}LV_{N,L}(r_{1}-r_{2})+LV_{N,L}(r_{1}-r_{2})S_{1}^{2}\geq -C_{\delta}(V)(N/L)^{\beta+\delta}S_{1}^{2}S_{2}^{2}\label{equ:appendix basic operator estimate: commutator estimate},
\end{align}
where $C_{\delta}(V)$ is dependent on $V$.

\end{lemma}
\begin{proof}
For estimate $(\ref{equ:appendix basic operator estimate: one body estimate})$,
by H\"{o}lder and Sobolev inequality, we get
\begin{align}\label{equ:appendix basic operator estimate: one body estimate, proof}
\lra{LV_{N,L}(r_{1}-r_{2})\varphi_{L},\varphi_{L}}\leq& \n{LV_{N,L}}_{L_{z}^{\infty}L_{x}^{1+}}\n{\varphi_{L}}_{L^{2}L_{x_{1}}^{\infty-}}^{2}\\
\leq& C_{\delta}L(N/L)^{\beta+\delta}\n{V}_{L_{z}^{\infty}L_{x}^{1+\delta}}\n{\sqrt{1-\Delta_{x_{1}}}\varphi_{L}}_{L^{2}}^{2}\notag\\
=& C_{\delta}L(N/L)^{\beta+\delta}\n{V}_{L_{z}^{\infty}L_{x}^{1+\delta}}\lra{(1-\Delta_{x_{1}})\varphi_{L},\varphi_{L}}.\notag
\end{align}
With $L(N/L)^{\beta}\to 1^{-}$, we obtain estimate $(\ref{equ:appendix basic operator estimate: one body estimate})$.

For estimate $(\ref{equ:appendix basic operator estimate})$, we recall Littlewood-Paley projectors defined by $(\ref{eq:Littlewood-Paley projectors})$ and decompose $\varphi_{L}$ that
 \begin{align}
 \varphi_{L}=\varphi_{L,1}+\varphi_{L,2},
 \end{align}
 where
 $$\varphi_{L,1}=\sum_{0\leq m_{1}<m_{2}}P_{x_{1},m_{1}}P_{x_{2},m_{2}}\varphi_{L},$$
 and
 $$\varphi_{L,2}=\sum_{0\leq m_{2}\leq m_{1}}P_{x_{1},m_{1}}P_{x_{2},m_{2}}\varphi_{L}.$$
Then we have
\begin{align}
&\n{\sqrt{1-\Delta_{x_{1}}}\varphi_{L,1}}_{L^{2}}\leq \n{\sqrt{1-\Delta_{x_{2}}}\varphi_{L,1}}_{L^{2}},\label{equ:appendix basic operator estimate, Littlewood-Paley x1}\\
&\n{\sqrt{1-\Delta_{x_{2}}}\varphi_{L,2}}_{L^{2}}\leq \n{\sqrt{1-\Delta_{x_{1}}}\varphi_{L,2}}_{L^{2}}\label{equ:appendix basic operator estimate, Littlewood-Paley x2}.
\end{align}
By estimates ($\ref{equ:appendix basic operator estimate, Littlewood-Paley x1}$), ($\ref{equ:appendix basic operator estimate, Littlewood-Paley x2}$) and Sobolev inequality, we find that
\begin{align}\label{equ:appendix basic operator estimate,proof}
&|\lra{LV_{N,L}(r_{1}-r_{2})\varphi_{L},\varphi_{L}}|\\
\leq& \int L|V_{N,L}(r_{1}-r_{2})|(|\varphi_{L,1}|^{2}+|\varphi_{L,2}|^{2})dr_{1}dr_{2}\notag\\
\leq& L\n{V_{N,L}}_{L_{z}^{\infty}L_{x}^{1}}(\n{\varphi_{L,1}}_{L^{2}L_{x_{1}}^{\infty}}^{2}+\n{\varphi_{L,2}}_{L^{2}L_{x_{2}}^{\infty}}^{2})\notag\\
\leq& C_{\delta}L(N/L)^{\beta}\n{V}_{L_{z}^{\infty}L_{x}^{1}}\lrs{\n{(\sqrt{1-\Delta_{x_{1}}})^{1+2\delta}\varphi_{L,1}}_{L^{2}}^{2}+
\n{(\sqrt{1-\Delta_{x_{2}}})^{1+2\delta}\varphi_{L,2}}_{L^{2}}^{2}}\notag\\
\leq& C_{\delta}L(N/L)^{\beta} \n{V}_{L_{z}^{\infty}L_{x}^{1}}\lra{(1-\Delta_{x_{1}})^{1/2+\delta}(1-\Delta_{x_{2}})^{1/2+\delta}\varphi_{L},\varphi_{L}}.\notag
\end{align}
With $L(N/L)^{\beta}\to 1^{-}$, we obtain estimate $(\ref{equ:appendix basic operator estimate})$.

For estimate $(\ref{equ:appendix operator norm estimate})$, in the same manner as estimate $(\ref{equ:appendix basic operator estimate})$, we get
\begin{align}
\lra{|V(r_{1}-r_{2})|\varphi,\varphi}\leq C\n{V}_{L_{z}^{\infty}L_{x}^{1}}\lra{(1-\Delta_{x_{1}})(1-\Delta_{x_{2}})\varphi,\varphi},
\end{align}
which is equivalent to
\begin{align}
\lra{\nabla_{x_{1}}}^{-1}\lra{\nabla_{x_{2}}}^{-1}|V(r_{1}-r_{2})|\lra{\nabla_{x_{1}}}^{-1}\lra{\nabla_{x_{2}}}^{-1}\leq  C\n{V}_{L_{z}^{\infty}L_{x}^{1}}.
\end{align}
By Property $2$ in Lemma $\ref{lemma:appendix standard operator inequalities}$, we obtain estimate $(\ref{equ:appendix operator norm estimate})$.

For estimate $(\ref{equ:appendix basic operator estimate: commutator estimate})$,  we decompose
\begin{align}\label{equ:appendix basic operator estimate: commutator estimate,A+B}
&\lra{S_{1}^{2}LV_{N,L}(r_{1}-r_{2})\varphi_{L},\varphi_{L}}\\
=&\lra{(1-\Delta_{x_{1}})(LV_{N,L}(r_{1}-r_{2})\varphi_{L}),\varphi_{L}}+
\lra{(-\pa_{z_{1}}^{2}-1/L^{2})(LV_{N,L}(r_{1}-r_{2})\varphi_{L}),\varphi_{L}}\notag\\
=& A+B.\notag
\end{align}
It needs only to control $A$ and $B$. Using the identity $\nabla_{x_{1}}V_{N,L}(r_{1}-r_{2})=-\nabla_{x_{2}}V_{N,L}(r_{1}-r_{2})$ and integration by parts, we get
\begin{align*}
A=&\lra{\nabla_{x_{1}}(LV_{N,L}(r_{1}-r_{2})\varphi_{L}),\nabla_{x_{1}}\varphi_{L}}\\
=&\lra{LV_{N,L}(r_{1}-r_{2})\nabla_{x_{1}}\varphi_{L},\nabla_{x_{1}}\varphi_{L}}+\lra{L(\nabla_{x_{1}}V_{N,L})(r_{1}-r_{2})\varphi_{L},\nabla_{x_{1}}\varphi_{L}}\\
=&\lra{LV_{N,L}(r_{1}-r_{2})\nabla_{x_{1}}\varphi_{L},\nabla_{x_{1}}\varphi_{L}}-\lra{L(\nabla_{x_{2}}V_{N,L})(r_{1}-r_{2})\varphi_{L},\nabla_{x_{1}}\varphi_{L}}\\
=&\lra{LV_{N,L}(r_{1}-r_{2})\nabla_{x_{1}}\varphi_{L},\nabla_{x_{1}}\varphi_{L}}+\int LV_{N,L}(r_{1}-r_{2})\nabla_{x_{2}}\varphi_{L}\nabla_{x_{1}}\overline{\varphi}_{L}dr_{1}dr_{2}\\
&+\int LV_{N,L}(r_{1}-r_{2})\varphi_{L}\nabla_{x_{2}}\nabla_{x_{1}}\overline{\varphi}_{L}dr_{1}dr_{2}\\
=&A_{1}+A_{2}+A_{3}.
\end{align*}

By H\"{o}lder and Sobolev inequality, we have
\begin{align} \label{equ:appendix basic operator estimate: commutator estimate,A1}
|A_{1}|\leq& L\n{V_{N,L}}_{L_{z}^{\infty}L_{x}^{2}}\n{\nabla_{x_{1}}\varphi_{L}}_{L^{2}L_{x_{2}}^{4}}^{2}\\
\leq&L(N/L)^{2\beta} \n{V}_{L_{z}^{\infty}L_{x}^{2}}\n{\sqrt{1-\Delta_{x_{2}}}\nabla_{x_{1}}\varphi_{L}}_{L^{2}}\n{\nabla_{x_{1}}\varphi_{L}}_{L^{2}}\notag\\
\leq&L(N/L)^{2\beta} \n{V}_{L_{z}^{\infty}L_{x}^{2}}\lra{S_{1}^{2}S_{2}^{2}\varphi_{L},\varphi_{L}}.\notag
\end{align}

Estimated in the same way as $A_{1}$,
\begin{align}\label{equ:appendix basic operator estimate: commutator estimate,A2}
|A_{2}|\leq& \int L|V_{N,L}(r_{1}-r_{2})|(|\nabla_{x_{1}}\varphi_{L}|^{2}+\nabla_{x_{2}}\varphi_{L}|^{2})dr_{1}dr_{2}\\
\leq &L(N/L)^{2\beta} \n{V}_{L_{z}^{\infty}L_{x}^{2}}\lra{S_{1}^{2}S_{2}^{2}\varphi_{L},\varphi_{L}}.\notag
\end{align}

For $A_{3}$, we have
\begin{align} \label{equ:appendix basic operator estimate: commutator estimate,A3}
|A_{3}|\leq& \n{LV_{N,L}(r_{1}-r_{2})\varphi_{L}}_{L^{2}}\n{\nabla_{x_{1}}\nabla_{x_{2}}\varphi_{L}}_{L^{2}}\\
\leq&CL(N/L)^{2\beta}\n{V}_{L_{z}^{\infty}L_{x}^{2}}\lra{(1-\Delta_{x_{1}})(1-\Delta_{x_{2}})\varphi_{L},\varphi_{L}},\notag
\end{align}
where we have used estimate $(\ref{equ:appendix basic operator estimate})$ with $\delta=1/2$ and $V$ replaced by $V^{2}$ in the last inequality.

Hence we have
\begin{align}\label{equ:appendix basic operator estimate: commutator estimate,A}
|A|\leq CL(N/L)^{2\beta}\n{V}_{L_{z}^{\infty}L_{x}^{2}}\lra{(1-\Delta_{x_{1}})(1-\Delta_{x_{2}})\varphi_{L},\varphi_{L}}.
\end{align}

Next, we decompose B into two terms
\begin{align}\label{equ:appendix basic operator estimate: commutator estimate,B1+B2}
B=&\lra{(-\pa_{z_{1}}^{2}-1/L^{2})LV_{N,L}(r_{1}-r_{2})\varphi_{L},P_{z_{1},>1}\varphi_{L}}\\
=&\lra{-\pa_{z_{1}}^{2}(LV_{N,L}(r_{1}-r_{2})\varphi_{L}),P_{z_{1},>1}\varphi_{L}}-L^{-2}\lra{LV_{N,L}(r_{1}-r_{2})\varphi_{L},P_{z_{1},>1}\varphi_{L}}\notag\\
=&B_{1}+B_{2}.\notag
\end{align}
For $B_{1}$, we expand
\begin{align}
B_{1}=B_{11}+B_{12},
\end{align}
where
\begin{align*}
&B_{11}=\lra{L(\pa_{z_{1}}V_{N,L}(r_{1}-r_{2}))\varphi_{L},\pa_{z_{1}}P_{z_{1},>1}\varphi_{L}},\\
&B_{12}=\lra{LV_{N,L}(r_{1}-r_{2})\pa_{z_{1}}\varphi_{L},
\pa_{z_{1}}P_{z_{1},>1}\varphi_{L}}
\end{align*}
For $B_{11},$ applying H\"{o}lder inequality at $x_{2}$, we obtain
\begin{align}\label{equ:appendix basic operator estimate: commutator estimate,B11}
|B_{11}|\leq & \n{L\pa_{z_{1}}V_{N,L}(r_{1}-r_{2})}_{L_{z_{2}}^{\infty}L_{x_{2}}^{1+}}
\n{\varphi_{L}}_{L^{2}L_{x_{2}}^{\infty-}}\n{\pa_{z_{1}}P_{z_{1},>1}\varphi_{L}}_{L^{2}L_{x_{2}}^{\infty-}}\\
\leq& C_{\delta}L(N/L)^{2\beta+\delta}\n{\pa_{z}V}_{L_{z}^{\infty}L_{x}^{1+\delta}}\n{S_{2}\varphi_{L}}_{L^{2}}
\n{S_{2}\pa_{z_{1}}P_{z_{1},>1}\varphi_{L}}_{L^{2}}\notag\\
\leq& C_{\delta}L(N/L)^{2\beta+\delta}\n{\pa_{z}V}_{L_{z}^{\infty}L_{x}^{1+\delta}}\lra{S_{1}^{2}S_{2}^{2}\varphi_{L},\varphi_{L}}.\notag
\end{align}
where we have used estimate $(\ref{equ:Appendix operator inequality 1})$ in the last inequality.

Computing in the same way, we have
\begin{align} \label{equ:appendix basic operator estimate: commutator estimate,B12}
|B_{12}|\leq & \n{LV_{N,L}}_{L_{z_{2}}^{\infty}L_{x_{2}}^{1+}}\n{\pa_{z_{1}}\varphi_{L}}_{L^{2}L_{x_{2}}^{\infty-}}
\n{\pa_{z_{1}}P_{z_{1},>1}\varphi_{L}}_{L^{2}L_{x_{2}}^{\infty-}} \\
\leq&\n{LV_{N,L}}_{L_{z_{2}}^{\infty}L_{x_{2}}^{1+}} \n{S_{2}\pa_{z_{1}}\varphi_{L}}_{L^{2}}\n{S_{2}\pa_{z_{1}}P_{z_{1},>1}\varphi_{L}}_{L^{2}}\notag\\
\leq&C_{\delta}(N/L)^{\beta+\delta}\n{V}_{L_{z}^{\infty}L_{x}^{1+\delta}} \n{S_{2}S_{1}\varphi_{L}}_{L^{2}}^{2},\notag
\end{align}
where we have used $(\ref{equ:Appendix operator inequality 0})$ and $(\ref{equ:Appendix operator inequality 1})$ in the last inequality.

Hence, with $L(N/L)^{\beta}\to 1^{-}$, we have
\begin{align}\label{equ:appendix basic operator estimate: commutator estimate,B1}
|B_{1}|\leq C_{\delta}(V)L(N/L)^{2\beta+\delta}\n{S_{1}S_{2}\varphi_{L}}_{L^{2}}^{2}.
\end{align}

For $B_{2}$, we have
\begin{align}\label{equ:appendix basic operator estimate: commutator estimate,B2}
|B_{2}|\leq& L^{-2}\n{LV_{N,L}}_{L_{z_{2}}^{\infty}L_{x_{2}}^{1+}}\n{\varphi_{L}}_{L^{2}L_{x_{2}}^{\infty-}}\n{P_{z_{1},>1}\varphi_{L}}_{L^{2}L_{x_{2}}^{\infty-}}\\
\leq&C_{\delta}L^{-1}(N/L)^{\beta+\delta}\n{V}_{L_{z}^{\infty}L_{x}^{1+\delta}}\n{S_{2}\varphi_{L}}_{L^{2}}\n{S_{2}P_{z_{1},>1}\varphi_{L}}_{L^{2}}\notag\\
\leq&C_{\delta}(N/L)^{\beta+\delta}\n{V}_{L_{z}^{\infty}L_{x}^{1+\delta}}\lra{S_{1}^{2}S_{2}^{2}\varphi_{L},\varphi_{L}},\notag
\end{align}
where we used $L^{-2}P_{z_{1},>1}\leq S_{1}^{2}P_{z_{1},>1}$ in the last inequality.

Putting $(\ref{equ:appendix basic operator estimate: commutator estimate,A})$ $(\ref{equ:appendix basic operator estimate: commutator estimate,B1})$ and
$(\ref{equ:appendix basic operator estimate: commutator estimate,B2})$
 together, with $L(N/L)^{\beta}\to 1^{-}$, we obtain estimate $(\ref{equ:appendix basic operator estimate: commutator estimate}).$
\end{proof}

\begin{lemma}\label{lemma:appendix interpolation inequality}
Let $A:D(A)\mapsto H$, be a positive selfadjoint operator in the Hilbert space $H$, and let $0< \alpha< 1.$ Then
\begin{align*}
A^{\alpha}\leq (1-\alpha)\eta^{-1} A+\alpha\eta^{\frac{\alpha}{1-\alpha}},
\end{align*}
for $\eta\in (0,1).$
\end{lemma}
\begin{proof}
By the spectral representation and the inequality
$$\lambda^{\alpha}\leq \alpha \eta^{-1}\lambda+(1-\alpha)\eta^{\frac{\alpha}{1-\alpha}},\quad \forall \lambda\in (0,\infty),$$
we obtain
\begin{align*}
\lra{A^{\alpha}u,u}_{H}
=&\int_{0}^{\infty}\lambda^{\alpha}d\n{E_{\lambda}u}_{H}^{2}\\
\leq& (1-\alpha)\eta^{-1} \int_{0}^{\infty} \lambda d\n{E_{\lambda}u}_{H}^{2}+\alpha\eta^{\frac{\alpha}{1-\alpha}}\int_{0}^{\infty}d\n{E_{\lambda}u}_{H}^{2}\\
=& (1-\alpha)\eta^{-1}\lra{Au,u}_{H}+\alpha \eta^{\frac{\alpha}{1-\alpha}}\lra{u,u}_{H}.
\end{align*}

\end{proof}

Recall Lemma $\ref{lemma:convergence indentity approximation}$, we now give the proof as follows.
\begin{lemma}\label{lemma:Appendix indentity approximation}
Let $\rho\in L^{1}(\Om)$ be a function compactly supported on $\Om$ such that
$$\sup_{z}\int |\rho(x,z)||x|dx<\infty$$
and define $\rho_{\varepsilon,\lambda}(x,z)=\varepsilon^{-2}\lambda^{-1}\rho(x/\varepsilon,z/\lambda)$ and $g(z)=\int \rho(x,z)dx$. Then, for every $\kappa\in [0,1)$, there exists $C>0$, such that
\begin{align*}
&\bbabs{TrJ^{(k)}(\rho_{\varepsilon,\lambda}(r_{j}-r_{k+1})-\delta(x_{j}-x_{k+1})g(z_{j}-z_{k+1}))\gamma^{(k+1)}}\\
\leq& C\varepsilon^{\kappa}\sup_{z}\int |\rho(x,z)||x|^{\kappa}dx\lrs{\n{\lra{\nabla_{x_{j}}}^{-1}J^{(k)}\lra{\nabla_{x_{j}}}}_{op}+\n{\lra{\nabla_{x_{j}}}J^{(k)}\lra{\nabla_{x_{j}}}^{-1}}_{op}}
Tr \lra{\nabla_{x_{j}}}^{2}\lra{\nabla_{x_{k+1}}}^{2}\gamma^{(k+1)}\\
&+C_{J}\n{\rho_{1,\lambda}-\rho_{1,1}}_{L^{1}}Tr \lra{\nabla_{r_{j}}}^{2}\lra{\nabla_{r_{k+1}}}^{2}\gamma^{(k+1)}
\end{align*}
for all nonnegative $\gamma^{(k+1)}\in \cl_{k+1}^{1}.$
\end{lemma}
\begin{proof}
We present a proof by modifying the proof in \cite{kirkpatrick2011derivation}. Such a method has been used by various authors, for example \cite{chen2013rigorous2dfrom3d}. It suffices to prove the estimate for $k=1$. Since the observable $J^{(1)}$ can be written as a sum of a self-adjoint operator and an anti-self-adjoint operator, we may assume $J^{(1)}$ is self-adjoint. We represent $\gamma^{(2)}$ by $\gamma^{(2)}=\sum_{j}\lambda_{j}|\varphi_{j}\rangle\langle\varphi_{j}|$, where $\varphi_{j}\in L^{2}(\Om^{2})$ and $\lambda_{j}\geq 0$. Then, we have
\begin{align*}
&Tr J^{(1)}(\rho_{\varepsilon,\lambda}(r_{1}-r_{2})-\delta(x_{1}-x_{2})g(z_{1}-z_{2}))\gamma^{(2)}\\
=&\sum_{j}\lambda_{j}\lra{\varphi_{j},J^{(1)}(\rho_{\varepsilon}(r_{1}-r_{2})-\delta(x_{1}-x_{2})g(z_{1}-z_{2}))\varphi_{j}}\\
=&\sum_{j}\lambda_{j}\lra{\psi_{j},(\rho_{\varepsilon,\lambda}(r_{1}-r_{2})-\delta(x_{1}-x_{2})g(z_{1}-z_{2}))\varphi_{j}},
\end{align*}
where $\psi_{j}=(J^{(1)}\otimes 1)\varphi_{j}.$ Then, we decompose
$$\lra{\psi_{j},(\rho_{\varepsilon,\lambda}(r_{1}-r_{2})-\delta(x_{1}-x_{2})g(z_{1}-z_{2}))\varphi_{j}}=A_{j}+B_{j},$$
where
\begin{align*}
&A_{j}=\lra{\psi_{j},(\rho_{\varepsilon,1}(r_{1}-r_{2})-\delta(x_{1}-x_{2})g(z_{1}-z_{2}))\varphi_{j}},\\
&B_{j}=\lra{\psi_{j},(\rho_{\varepsilon,\lambda}(r_{1}-r_{2})-\rho_{\varepsilon,1}(r_{1}-r_{2}))\varphi_{j}}.\\
\end{align*}

For $A_{j}$, switching to Fourier space in the $x$-direction, we find
\begin{align*}
|A_{j}|=&\bbabs{\int \overline{\widehat{\psi}}_{j}(n_{1},n_{2};z_{1},z_{2})\widehat{\varphi}_{j}(m_{1},m_{2};z_{1},z_{2})\rho_{1,\lambda}(x,z_{1}-z_{2})(e^{i\varepsilon x\cdot(n_{1}-m_{1})}-1)\\
&\times \delta(n_{1}+n_{2}-m_{1}-m_{2})dxdz_{1}dz_{2}dn_{1}dn_{2}dm_{1}dm_{2}}\\
\leq&\int |\widehat{\psi}_{j}(n_{1},n_{2};z_{1},z_{2})||\widehat{\varphi}_{j}(m_{1},m_{2};z_{1},z_{2})|\delta(n_{1}+n_{2}-m_{1}-m_{2})\\
&\times \bbabs{\int \rho(x,z_{1}-z_{2})(e^{i\varepsilon x\cdot(n_{1}-m_{1})}-1) dx}dz_{1}dz_{2}dn_{1}dn_{2}dm_{1}dm_{2}.
\end{align*}
Using the inequality that $\forall \kappa\in (0,1)$
\begin{align*}
\bbabs{e^{i\varepsilon x\cdot(n_{1}-m_{1})}-1}\leq &\varepsilon^{\kappa}|x|^{\kappa}|n_{1}-m_{1}|^{\kappa}\\
\leq& \varepsilon^{\kappa}|x|^{\kappa}\lrs{|n_{1}|^{\kappa}+|m_{1}|^{\kappa}},
\end{align*}
we get
\begin{align*}
&|\lra{\psi_{j},(\rho_{\varepsilon,1}(r_{1}-r_{2})-\delta(x_{1}-x_{2})g(z_{1}-z_{2}))\varphi_{j}}|\\
\leq&\varepsilon^{\kappa}\sup_{z}\int |\rho(x,z)||x|^{\kappa}dx\int\delta(n_{1}+n_{2}-m_{1}-m_{2})\\
& (|n_{1}|^{\kappa}+|m_{1}|^{\kappa}) |\widehat{\psi}_{j}(n_{1},n_{2};z_{1},z_{2})||\widehat{\varphi}_{j}(m_{1},m_{2};z_{1},z_{2})|dz_{1}dz_{2}dn_{1}dn_{2}dm_{1}dm_{2}\\
=&\varepsilon^{\kappa}\sup_{z}\int |\rho(x,z)||x|^{\kappa}dx(I+II),
\end{align*}
where
\begin{align*}
&I=\int\delta(n_{1}+n_{2}-m_{1}-m_{2}) |n_{1}|^{\kappa}|\widehat{\psi}_{j}(n_{1},n_{2};z_{1},z_{2})||\widehat{\varphi}_{j}(m_{1},m_{2};z_{1},z_{2})|dz_{1}dz_{2}dn_{1}dn_{2}dm_{1}dm_{2},\\
&II=\int\delta(n_{1}+n_{2}-m_{1}-m_{2})|m_{1}|^{\kappa} |\widehat{\psi}_{j}(n_{1},n_{2};z_{1},z_{2})||\widehat{\varphi}_{j}(m_{1},m_{2};z_{1},z_{2})|dz_{1}dz_{2}dn_{1}dn_{2}dm_{1}dm_{2}.
\end{align*}
The estimates for $I$ and $II$ are similar, so we only deal with $I$ explicitly.
\begin{align*}
I\leq &\int \delta(n_{1}+n_{2}-m_{1}-m_{2})\frac{\lra{n_{1}}\lra{n_{2}}}{\lra{m_{1}}\lra{m_{2}}}|\widehat{\psi}_{j}(n_{1},n_{2};z_{1},z_{2})|\\
&\times \frac{\lra{m_{1}}\lra{m_{2}}}{\lra{n_{1}}^{1-\kappa}\lra{n_{2}}}|\widehat{\varphi}_{j}(m_{1},m_{2};z_{1},z_{2})|dz_{1}dz_{2}dn_{1}dn_{2}dm_{1}dm_{2}\\
\leq&\theta\int \delta(n_{1}+n_{2}-m_{1}-m_{2})\frac{\lra{n_{1}}^{2}\lra{n_{2}}^{2}}{\lra{m_{1}}^{2}\lra{m_{2}}^{2}}|\widehat{\psi}_{j}(n_{1},n_{2};z_{1},z_{2})|^{2}
dz_{1}dz_{2}dn_{1}dn_{2}dm_{1}dm_{2}\\
&+\theta^{-1}\int \delta(n_{1}+n_{2}-m_{1}-m_{2})
\frac{\lra{m_{1}}^{2}\lra{m_{2}}^{2}}{\lra{n_{1}}^{2(1-\kappa)}\lra{n_{2}}^{2}}|\widehat{\varphi}_{j}(m_{1},m_{2};z_{1},z_{2})|^{2}dz_{1}dz_{2}dn_{1}dn_{2}dm_{1}dm_{2}\\
\leq&\theta\lra{\psi_{j},(1-\Delta_{x_{1}})(1-\Delta_{x_{2}})\psi_{j}}\sup_{n}\sum_{m_{2}} \frac{1}{\lra{m_{2}-n}^{2}\lra{m_{2}}^{2}}\\
&+\theta^{-1}\lra{\varphi_{j},(1-\Delta_{x_{1}})(1-\Delta_{x_{2}})\varphi_{j}}\sup_{n}\sum_{m_{2}}\frac{1}{\lra{m_{2}-n}^{2(1-\kappa)}\lra{m_{2}}^{2}}.
\end{align*}
Since
\begin{align*}
&\sup_{n}\sum_{m\in \Z^{2}} \frac{1}{\lra{m-n}^{2(1-\kappa)}\lra{m}^{2}}<\infty
\end{align*}
for $\kappa\in [0,1)$, we have
\begin{align}\label{eq:appendix approxiamation identity A}
|A_{j}|\leq \theta \lra{\psi_{j},(1-\Delta_{x_{1}})(1-\Delta_{x_{2}})\psi_{j}}+\theta^{-1}\lra{\varphi_{j},(1-\Delta_{x_{1}})(1-\Delta_{x_{2}})\varphi_{j}}.
\end{align}
Hence, we obtain
\begin{align} \label{eq:appendix approxiamation identity I}
&\bbabs{Tr J^{(1)}(\rho_{\varepsilon,1}(r_{1}-r_{2})-\delta(x_{1}-x_{2})g(z_{1}-z_{2}))\gamma^{(2)}}\\
\leq&C\varepsilon^{\kappa}\sup_{z}\int |\rho(x,z)||x|^{\kappa}dx\lrs{\theta Tr\lra{\nabla_{x_{1}}}^{2}\lra{\nabla_{x_{2}}}^{2}\gamma^{(2)}+\theta^{-1} Tr J^{(1)}\lra{\nabla_{x_{1}}}^{2}\lra{\nabla_{x_{2}}}^{2}J^{(1)}\gamma^{(2)}}\notag\\
=& C\varepsilon^{\kappa}\sup_{z}\int |\rho(x,z)||x|^{\kappa}dx
(\theta Tr \lra{\nabla_{x_{1}}}^{2}\lra{\nabla_{x_{2}}}^{2}\gamma^{(2)}\notag\\
&+\theta^{-1} Tr \lra{\nabla_{x_{1}}}^{-1}\lra{\nabla_{x_{2}}}^{-1}J^{(1)}
\lra{\nabla_{x_{1}}}^{2}J^{(1)}\lra{\nabla_{x_{1}}}^{-1}\lra{\nabla_{x_{1}}}
\lra{\nabla_{x_{2}}}^{2}\gamma^{(2)}\lra{\nabla_{x_{1}}}\lra{\nabla_{x_{2}}})\notag\\
\leq &C\varepsilon^{\kappa}\sup_{z}\int |\rho(x,z)||x|^{\kappa}dx\lrs{\theta + \theta^{-1}\n{\lra{\nabla_{x_{1}}}^{-1}J^{(1)}\lra{\nabla_{x_{1}}}}_{op}\n{\lra{\nabla_{x_{1}}}J^{(1)}\lra{\nabla_{x_{1}}}^{-1}}_{op}}Tr \lra{\nabla_{x_{1}}}^{2}\lra{\nabla_{x_{2}}}^{2}\gamma^{(2)}\notag\\
= &C\varepsilon^{\kappa}\sup_{z}\int |\rho(x,z)||x|^{\kappa}dx\lrs{\n{\lra{\nabla_{x_{1}}}^{-1}J^{(1)}\lra{\nabla_{x_{1}}}}_{op}+\n{\lra{\nabla_{x_{1}}}J^{(1)}\lra{\nabla_{x_{1}}}^{-1}}_{op}}Tr \lra{\nabla_{x_{1}}}^{2}\lra{\nabla_{x_{2}}}^{2}\gamma^{(2)},\notag
\end{align}
where we have taken $\theta=\n{\lra{\nabla_{x_{1}}}J^{(1)}\lra{\nabla_{x_{1}}}^{-1}}_{op}$ in the last line.

For $B_{j}$, we use the operator inequality (also see \cite[(A.63)]{elgart2006gross})
$$|V(r_{1}-r_{2})|\leq C\n{V}_{L^{1}}(1-\Delta_{r_{1}})(1-\Delta_{r_{2}}),$$
which can be estimated in the same way as estimate $(\ref{equ:appendix operator norm estimate}).$ Then, we get
\begin{align}
|B_{j}|\leq& \theta\lra{\psi_{i},|(\rho_{\varepsilon,\lambda}-\rho_{\varepsilon,1})(r_{1}-r_{2})|\psi_{i}} +\theta^{-1}\lra{\varphi_{i},|(\rho_{\varepsilon,\lambda}-\rho_{\varepsilon,1})(r_{1}-r_{2})|\varphi_{i}}\\
\leq& C\n{\rho_{\varepsilon,\lambda}-\rho_{\varepsilon,1}}_{L_{z}^{1}L_{x}^{1}}\lrs{\theta\lra{\psi_{i},(1-\Delta_{r_{1}})(1-\Delta_{r_{2}})\psi_{i}}
+\theta^{-1}\lra{\varphi_{i},(1-\Delta_{r_{1}})(1-\Delta_{r_{2}})\varphi_{i}}}\notag\\
=& C\n{\rho_{1,\lambda}-\rho_{1,1}}_{L_{z}^{1}L_{x}^{1}}\lrs{\theta\lra{\psi_{i},(1-\Delta_{r_{1}})(1-\Delta_{r_{2}})\psi_{i}}
+\theta^{-1}\lra{\varphi_{i},(1-\Delta_{r_{1}})(1-\Delta_{r_{2}})\varphi_{i}}}.\notag
\end{align}
Similarly, we have
\begin{align} \label{eq:appendix approxiamation identity II}
|Tr J^{(1)}(\rho_{\varepsilon,\lambda}(r_{1}-r_{2})-\rho_{\varepsilon,\lambda}(r_{1}-r_{2}))\gamma^{(2)}|\leq C_{J}\n{\rho_{1,\lambda}-\rho_{1,1}}_{L^{1}}Tr \lra{\nabla_{r_{1}}}^{2}\lra{\nabla_{r_{2}}}^{2}\gamma^{(2)}.
\end{align}

Together with estimate $(\ref{eq:appendix approxiamation identity I})$ and $(\ref{eq:appendix approxiamation identity II})$, we complete the proof.

\end{proof}

\begin{lemma} \label{lemma:Appendix operator inequality}
Recall
\begin{align*}
&S^{2}=1-\Delta_{r}-1/L^{2},\\
&\wt{S}^{2}=1-\Delta_{x}-\pa_{z}^{2}/L^{2}-1/L^{2}.
\end{align*}
We have
\begin{align}
\label{equ:Appendix operator inequality 0}&(1-\Delta_{r})\leq 2L^{-2}S^{2},\\
\label{equ:Appendix operator inequality 1}&(1-\Delta_{r})P_{z,> 1}\leq 2S^{2}P_{z,> 1},\\
\label{equ:Appendix operator inequality 2}&\wt{S}^{2}\wt{P}_{1}\geq L^{-2}\wt{P}_{1},\\
\label{equ:Appendix operator inequality 3}&2\wt{S}^{2}\geq 1-\Delta_{r}
\end{align}
\end{lemma}
\begin{proof}
For $(\ref{equ:Appendix operator inequality 0})$, we have
$$1-\Delta_{r}=S^{2}+1/L^{2}\leq 2L^{-2}S^{2}.$$

For $(\ref{equ:Appendix operator inequality 1})$, we note that
$$L^{-2}P_{z,>1}\leq (-\pa_{z}^{2}-1/L^{2})P_{z,>1}\leq S^{2}P_{z,>1}.$$
Thus, we obtain
$$(1-\Delta_{r})P_{z,>1}=(S^{2}+1/L^{2})P_{z,>1}\leq 2S^{2}P_{z,>1}.$$

For $(\ref{equ:Appendix operator inequality 2})$, we have
\begin{align*}
L^{-2}\wt{P}_{1}\leq (-\pa_{z}^{2}/L^{2}-1/L^{2})\wt{P}_{1}\leq \wt{S}^{2}\wt{P}_{1}.
\end{align*}

For $(\ref{equ:Appendix operator inequality 3})$, we note that
\begin{align}
2(-\pa_{z}^{2}-1)\wt{P}_{1}= -\pa_{z}^{2}\wt{P}_{1}+(-\pa_{z}^{2}-2)\wt{P}_{1}\geq -\pa_{z}^{2}\wt{P}_{1}.
\end{align}
Then we have
\begin{align}\label{equ:Appendix operator inequality equ 1}
2\wt{S}^{2}\geq&2(1-\Delta_{x})+2L^{-2}(-\pa_{z}^{2}-1)\wt{P}_{1}\\
\geq& 2(1-\Delta_{x})-L^{-2}\pa_{z}^{2}\wt{P}_{1}\notag\\
\geq& 2(1-\Delta_{x})-\pa_{z}^{2}\wt{P}_{1}.\notag
\end{align}
Noting that $-\pa_{z}^{2}\wt{P}_{0}=\wt{P}_{0},$ we obtain
\begin{align}\label{equ:Appendix operator inequality equ 2}
1-\Delta_{x}
\geq 1 -\pa_{z}^{2}\wt{P}_{0}.
\end{align}
Putting $(\ref{equ:Appendix operator inequality equ 1})$ and $(\ref{equ:Appendix operator inequality equ 2})$ together, we establish $(\ref{equ:Appendix operator inequality 3}).$
\end{proof}
We need the following facts as well. The proofs are elementary and we omit them.
\begin{lemma} \label{lemma:appendix standard operator inequalities}
$1.$ Suppose that $A\geq 0,$ $P_{j}=P_{j}^{*}$, and $I=P_{0}+P_{1}$. Then $A\leq 2P_{0}AP_{0}+2P_{1}AP_{1}.$\\
$2.$ If $A\geq B\geq 0,$ and $AB=BA$, then $A^{\alpha}\geq B^{\alpha}$ for any $\alpha>0$. Especially, if $\alpha=2$, then $\n{A}_{op}\geq \n{B}_{op}.$\\
$3.$ If $A_{1}\geq A_{2}\geq 0 $, $B_{1}\geq B_{2}\geq 0$ and $A_{i}B_{j}=B_{j}A_{i}$ for all $1\leq i,$ $j\leq 2,$ then $A_{1}B_{1}\geq A_{2}B_{2}$.\\
$4.$ If $A\geq 0$ and $AB=BA,$ then $A^{1/2}B=BA^{1/2}$.

\end{lemma}

\begin{lemma} \label{lemma:Appendix kernel and trace}
Suppose $\sigma:L^{2}(\Om^{k})\mapsto L^{2}(\Om^{k})$ has kernel
$$\sigma(\mathbf{r}_{k},\mathbf{r}_{k}')=\int \psi(\mathbf{r}_{k},\mathbf{r}_{N-k})\overline{\psi}(\mathbf{r}_{k}',\mathbf{r}_{N-k})d\mathbf{r}_{N-k},$$
for some $\psi\in L^{2}(\Om^{k}),$ and let $A,$ $B:L^{2}(\Om^{k})\mapsto L^{2}(\Om^{k}).$ Then the composition $A\sigma B$ has kernel
$$(A\sigma B)(\mathbf{r}_{k},\mathbf{r}_{k}')=\int (A\psi)(\mathbf{r}_{k},\mathbf{r}_{N-k})(\overline{B^{*}\psi})(\mathbf{r}_{k}',\mathbf{r}_{N-k})d
\mathbf{r}_{N-k}.$$
It follows that
$$TrA\sigma B=\lra{A\psi,B^{*}\psi}.$$
\end{lemma}

\begin{lemma}\label{lemma:Appendix cut-off operator}
Let $P_{M}^{j}$ be the orthogonal projection onto the sum of the first $M$ eigenspaces with respect to the spectral decomposition of $L^{2}(\Om)$ to the operator $-\Delta_{r_{j}}$ and
$$P_{M}^{(k)}=\prod_{j=1}^{k}P_{M}^{j}.$$
$1.$ Suppose that $J^{(k)}$ is a compact operator. Then $J_{M}^{(k)}:=P_{M}^{(k)}J^{(k)}P_{M}^{(k)}$ converges to $J^{(k)}$ in the operator norm.\\
$2.$ $\Delta_{r_{j}}J_{M}^{(k)}$ and $J_{M}^{(k)}\Delta_{r_{j}}$ are bounded operators.\\
$3.$ There exists a countable dense subset $\lr{T_{i}}$ of the closed unit ball in the space of bounded operators on $L^{2}(\Om^{k})$ such that $T_{i}$ is compact and in fact for each i there exists $M$ (depending on $i$) such that $T_{i}=P_{M}^{(k)}T_{i}P_{M}^{(k)}.$
\end{lemma}

\bibliographystyle{plain}
\nocite{*}
\bibliography{appendix}

\end{document}